\documentclass[a4paper,12pt]{article}
\usepackage[utf8]{inputenc}
\usepackage[T1]{fontenc}
\usepackage{mathtools}
\usepackage[left=1.5cm,top=2.5cm,right=1.5cm,bottom=2.5cm]{geometry}

\usepackage{url}

\usepackage{pgf,tikz,pgfplots}
\usepackage{overpic}

\usepackage{cite}
\usepackage{amsthm}
\usepackage{amsmath,amssymb}
\usepackage{enumerate}
\usepackage{comment}

\usepackage{appendix} 
\newtheorem{teo}{Theorem}[section] %El teo es el contador que me numera los lemas
\newtheorem*{teo*}{Theorem}
\newtheorem{lem}[teo]{Lemma} % ahora el teo me dao una numeración distinta a los lemas
\newtheorem{prop}[teo]{Proposition}

\newtheorem{rem}{Remark}[section]
\newtheorem{definicion}{Definition}[section]

\newcommand{\R}{\mathbb{R}}\newcommand{\Rn}{\R^n}
\newcommand{\N}{\mathbb{N}}
\DeclareMathOperator{\supp}{supp}
\DeclareMathOperator{\diver}{div}

\DeclareMathOperator{\pv}{pv}
\DeclareMathOperator{\dist}{dist}
\DeclareMathOperator{\loc}{loc}

\DeclareMathOperator{\diam}{diam}

\newcommand{\weakc}{\rightharpoonup}
\renewcommand{\O}{\Omega}
\renewcommand{\a}{\alpha}
\renewcommand{\b}{\beta}
\renewcommand{\d}{\delta}
\renewcommand{\l}{\lambda}
\newcommand{\s}{\sigma}
\newcommand{\g}{\gamma}
\newcommand{\p}{\partial}
\newcommand{\F}{\mathcal{F}}
\newcommand{\e}{\varepsilon}
\newcommand{\mc}{\mathcal}
\newcommand{\f}{\varphi}

\usepackage{amsmath}
\usepackage{latexsym}
\usepackage{amssymb}
\def\XXint#1#2#3{{\setbox0=\hbox{$#1{#2#3}{\int}$}
		\vcenter{\hbox{$#2#3$}}\kern-.5\wd0}}

\title{Nonlocal gradients in bounded domains motivated by Continuum Mechanics: Fundamental Theorem of Calculus and embeddings}
\author{J. C. Bellido \\
{\small E.T.S.I.\ Industriales, Department of Mathematics, Universidad de Castilla-La Mancha,} \\
{\small 13071-Ciudad Real, Spain. Email \url{JoseCarlos.Bellido@uclm.es}},\\
J. Cueto \\
{\small E.T.S.I.\ Industriales, Department of Mathematics, Universidad de Castilla-La Mancha,} \\
{\small 13071-Ciudad Real, Spain. Email \url{Javier.Cueto@uclm.es}}, \\
C. Mora-Corral \\
{\small Departamento de Matem\'aticas, Universidad Aut\'onoma de Madrid,} \\
{\small 28049 Madrid, Spain. Email \url{Carlos.Mora@uam.es}}}

\date{}

\begin{document}
\maketitle

\pagestyle{empty}

\begin{abstract}

In this paper we develop a new set of results based on a nonlocal gradient jointly inspired by the Riesz $s$-fractional gradient and Peridynamics, in the sense that its integration domain depends on a ball of radius $\d>0$ (horizon of interaction among particles, in the terminology of Peridynamics), while keeping at the same time the singularity of the Riesz potential in its integration kernel. Accordingly, we define a functional space suitable for nonlocal models in Calculus of Variations and partial differential equations. Our motivation is to develop the proper functional analysis framework in order to tackle nonlocal models in \textit{Continuum Mechanics}, which requires working with bounded domains, while retaining the good mathematical properties of Riesz $s$-fractional gradients. This functional space is defined consistently with Sobolev and Bessel fractional ones:  we consider the closure of smooth functions under the natural norm obtained as the sum of the $L^p$ norms of the function and its nonlocal gradient. Among the results showed in this investigation we highlight a nonlocal version of the Fundamental Theorem of Calculus (namely, a representation formula where a function can be recovered from its nonlocal gradient), which allows us to prove inequalities in the spirit of Poincar\'e, Morrey, Trudinger and Hardy as well as the corresponding compact embeddings. These results are enough to show the existence of minimizers of general energy functionals under the assumption of convexity. Equilibrium conditions in this nonlocal situation are also established, and those can be viewed as a new class of nonlocal partial differential equations 
in bounded domains. 
%Driven by applications in Peridynamics, Solid Mechanics and models in which it is essential to work with bounded domains we develop a new set of results based on a nonlocal gradient, inspired by the Riesz s-fractional one, but whose interaction domain depends on a ball of radius $\d>0$ (horizon, in the terminology of Peridynamics). 
%Accordingly, a functional space suitable for nonlocal models in Calculus of Variations and partial differential equations is introduced. These functional spaces are defined in analogy with the Sobolev and Bessel fractional ones:  we consider the closure of smooth functions under the natural norm obtained as the sum of the Lp norms of a function and its nonlocal gradient. Among the results introduced we highlight a Nonlocal Version of the Fundamental Theorem of Calculus (a representation formula where a function can be recovered from its nonlocal gradient), which allows us to prove inequalities in the spirit of Poincaré, Morrey, Trudinger and Hardy as well as the corresponding compact embeddings. These results are enough to show the existence of minimizers of the associated energy functionals under the assumption of convexity. The Euler-Lagrange equations are also shown. 
\end{abstract}


\noindent{\bf 2020 MSC: } Primary:
26A33, %Fractional derivatives and integrals
35R11, %Fractional partial differential equations
46E35, %Sobolev spaces and other spaces of "smooth'' functions, embedding theorems, trace theorems
49J45, %Methods involving semicontinuity and convergence; relaxation
74A70. %Peridynamics
Secondary:
35Q74, %PDEs in connection with mechanics of deformable solids
42B20, %Singular and oscillatory integrals (Calder�n-Zygmund, etc.)
49K21, %Optimality conditions for problems involving relations other than differential equations
74B20, %Nonlinear elasticity
74G65. %Energy minimization in equilibrium problems in solid mechanics

% %\addcontensline{toc}{chapter}{Contents}
%\tableofcontents
% %\markboth{{\bf Contents}}{{\bf Contents}}

%\newpage

\pagestyle{plain}

\section[Introduction]{Introduction}

In the last decades, models based on differential equations are increasingly sharing their prominence with those based on integral or integro-differential equations. This is due to the fact that they can catch some information that the local ones cannot, such as long range interactions or multiscale behaviour; they usually require less regularity of the functions, allowing for more general admissible solutions. In fact, they can overcome some drawbacks of the local models since a typical feature is to be able to provide an effective modelling for discontinuities or singularities.

One such example is Peridynamics, a nonlocal alternative model in Solid Mechanics proposed by Silling \cite{Silling2000}; see also \cite{SiEpWeXuAs07,LeSi08,SiLe10}.
One of its goals was to unify elastic and singularity phenomena, such as fracture or cavitation.
The development of this theory in the last years has been impressive.
As general expositions, we can mention the review paper \cite{JaMoOtOt}, the two books \cite{Gerstle,madenci_oterkus} and the collaborative handbooks \cite{BoFoGeSi17,Voyiadjis19}.
Several aspects of these models have been studied such as localization \cite{MeD,MeS,FossRaduYu}, existence and regularity \cite{FoRaWr20}, computational issues \cite{DeDuGlGuTiZh20,DeDuGun}, function spaces involved \cite{Mengesha12,MeD16}, or linear theories \cite{EmWe07,ZhDu10,DuGuLeZh13,TiQi13,MeDu14,DuTi18,ScMe19}.

Accordingly, nonlocal models are gaining attention in the modelling of various phenomena in physics, biology, geometry and more.
As a consequence, it is required a more thorough mathematical analysis of the new objects and operators involved. Some of those objects are of diffusion type, where the fractional Laplacian stands out: this is an operator that generalizes the standard Laplacian to a degree of differentiability beyond derivatives of integer order (see, e.g., \cite{CaSi07,DeGu13,RoSe14,Pozrikidis16,MoRaSe16,Ros-Oton,AbVa19,ChLiMa20} among hundreds of possible references). Others, on the other hand, are of gradient type, which may provide a better scope for nonlocal vector calculus \cite{DeGuOlKa21,MeD,MeS,DeGuMeSc21,BeCuMC,ShS2015,ShS2018,COMI2019,DGLZ}. One of the features of both kind of operators is that they require less regularity of the functions than their classical counterparts. 
In this paper we focus on the latter type, in particular, those called one-point gradients, or weighted nonlocal gradients in the terminology of \cite{DeGuOlKa21}. These nonlocal gradients are usually written in terms of a kernel $\rho$, typically with a singularity at the origin.
In general, for a function $u:\Omega \to \R$ they are defined as
\begin{equation} \label{eq: general nonlocal gradient}
	\mathcal{G}_\rho u(x)= \int_\O \frac{u(x)-u(y)}{|x-y|} \frac{x-y}{|x-y|} \rho(x-y)\, dy.
\end{equation}
A particular case of nonlocal gradient where this analysis has experienced a great interest since the works of Shieh and Spector \cite{ShS2015,ShS2018} is Riesz' $s$-fractional gradient. 
%It can appear in a variety of models involving diffusion processes.
%Mathematically, it enjoys natural generalizations of classical properties of the Laplacian, such as its behaviour under the Fourier transform: $\widehat{(-\Delta)^s u}(\xi)= |2\pi \xi|^{2s} \hat{u}(\xi)$.
%
%Besides fractional operators, generalizations of Sobolev spaces to those with a measure of the fractional differentiability, $W^{s,p}$, have also received an important attention. For $0 < s < 1$ and $p\geq 1$, the usual Sobolev fractional space is defined as
%\begin{equation*}
%W^{s,p}(\Rn)= \left\{ u \in L^p(\Rn) \, : \frac{|u(x)-u(y)|}{|x-y|^{\frac{n}{p}+s}} \in  L^p(\Rn \times \Rn) \right\}.
%\end{equation*}
%Despite the fact that these spaces measure a fractional degree of differentiability and the convergence (see \cite{BoBrMi2001}) of the $W^{s,p}$ seminorm (times a correction factor) to the norm of the classical gradient, one would still wonder about the existence of a proper fractional operator of gradient type.  In this case, the Riesz $s$-fractional gradient stands out. 
For $0 < s <1$ and $u: \Rn \rightarrow \R$ a smooth enough function, its $s$-fractional gradient is defined as
\begin{equation} \label{def: fractional gradient}
	D^s u(x)= c_{n,s}  \int_{\Rn} \frac{u(x)-u(y)}{|x-y|^{n+s}}\frac{x-y}{|x-y|} \, dy ,
\end{equation}
where $c_{n,s}$ is a suitable normalizing constant. It follows the same formula as in \eqref{eq: general nonlocal gradient} where the integration domain is considered to be $\Rn$ and the kernel is $\rho(x) = c_{n,s} \frac{1}{|x|^{n-1+s}}$. This object has attracted a great interest in the last years, as will be reported below, mainly due to its nice properties from a functional analysis perspective. However, the fact of being defined over the whole space is a remarkable drawback for applications in realistic physical models; furthermore, this feature obviously leads to some extra difficulties from a computational or numerical point of view.

Thus, given the interest of working with bounded domains, and inspired by \eqref{eq: general nonlocal gradient} and \eqref{def: fractional gradient} we propose the operator
\begin{equation} \label{introduction non local gradient}
	D_\delta^s u (x)= c_{n,s} \int_{B(x,\delta)} \frac{u(x)-u(y)}{|x-y|}\frac{x-y}{|x-y|}\frac{w_\d(x-y)}{|x-y|^{n-1+s}} \, dy ,
\end{equation}
as a particular case of \eqref{eq: general nonlocal gradient}, (with $\rho(x) =c_{n,s} \frac{1}{|x|^{n-1+s}}w_\d(x)$ and $w_\d \in C^\infty_c(B(0,\d))$ a cut-off function), except for the fact that in this case, $u$ is assumed to be defined, at least, in $\O_\d= \O + B(0,\d)$ so that the integral in \eqref{introduction non local gradient} is completely defined in the whole ball. This entails dealing with the `collar' $\O_\d \setminus \O$ as a nonlocal boundary. The nonlocal gradient \eqref{introduction non local gradient} is defined so that it keeps the fractional index $s$ of differentiability, while at the same time the interaction of particles is restricted to a distance no larger than $\delta$ (the {\it horizon} parameter in Peridynamics).

To put it into context, the fractional gradient \eqref{def: fractional gradient}, as recently addressed by several authors \cite{Schikorra2017,COMI2019,RoSa19,BeCuMC,BeCuMC21,KrSc22}, seems to be the suitable notion, from a merely mathematical perspective, for such a differential object. In particular, it has been proved in \cite{Silhavy2019} that formula \eqref{def: fractional gradient} determines up to a multiplicative
constant the unique object fulfilling some minimal consistency requirements from the physical and mathematical point of view, such as invariance under rotations and translations, $s$-homogeneity  under dilations and some weak continuity properties. Moreover, the classical gradient can be recovered when $s$ goes to $1$ in \eqref{def: fractional gradient}. This operator is closely related to the Riesz potential, $I_{1-s} (x)= \frac{c_{n,s}}{n-1+s} \left| x \right|^{-(n-1+s)}$, and particularly in the case of smooth functions $u \in C_c^{\infty}(\Rn)$, it can be written as a convolution of this kernel with the classical gradient: $D^s u = I_{1-s} * \nabla u$. This implies that its Fourier transform can be computed as $\widehat{D^s u}(\xi)=\frac{2\pi i \xi}{|2\pi \xi|} |2\pi \xi|^s \hat{u}(\xi)$, which gives another insight to the unfamiliar reader.

In the case of the nonlocal gradient $D^s_\d u$, among all the properties mentioned earlier and systematized in \cite{Silhavy2019} characterising the fractional gradient \eqref{def: fractional gradient}, the one that is not fulfilled by \eqref{introduction non local gradient} is the $s$-homogeneity under dilations, in favour of considering bounded domains (equivalently, a compactly supported kernel).
Similar operators have been studied in works like \cite{MeS,MeD}, where \eqref{introduction non local gradient} could fit after normalizing its kernel.
In particular, it was shown in \cite{MeS} that these operators converge to the classical gradient when the nonlocality vanishes.
Furthermore, we will see that a function $Q_\d^s$ plays the role of the Riesz potential in this framework, sharing the same singularity at zero than the Riesz potential, but with compact support, making $Q_\d^s$ integrable. Thus, the nonlocal gradient can be written as a convolution with the classical one: $D_\d^s u = Q_\d^s *\nabla u$ for sufficiently smooth functions. As it will be seen later, this shows that the nonlocal gradient (defined or extended to every point in $\Rn$ if necessary) of test functions such as $C^\infty_c(\Rn)$ or $\mathcal{S}$ (the Schwartz space) remains in such spaces, as opposed to the fractional gradient, making it more suitable for defining the nonlocal gradient of a distribution. 

Regarding functional spaces, essential for Calculus of Variations, the one associated to Riesz fractional gradients is the Bessel space $H^{s,p} (\Rn)$.
Among the several equivalent definitions, the most intuitive in this context is that based on the completion of $C^{\infty}_c (\Rn)$ functions under the norm
\[
\| u \|_{H^{s,p} (\Rn)} = \left( \| u \|_{L^p (\Rn)}^p + \| D^s u \|_{L^p (\Rn)}^p \right)^{\frac{1}{p}} .
\]
There are, of course, many spaces between $L^p (\Rn)$ and the Sobolev space $W^{1,p} (\Rn)$ that possess a degree of differentiability of order $s$.
The most familiar one is possibly the Gagliardo space $W^{s,p} (\Rn)$, which is equipped with the seminorm
\[
[u]_{W^{s,p} (\Rn)} = \left( \int_{\Rn} \int_{\Rn}  \frac{|u(x)-u(y)|^p}{|x-y|^{n+s p}} \, dx \, dy\right)^{\frac{1}{p}}
\]
and the norm
\[
\| u\|_{W^{s,p} (\Rn)} = \left( \| u \|_{L^p (\Rn)}^p + [u]_{W^{s,p} (\Rn)}^p \right)^{\frac{1}{p}} .
\]
A great difference between $H^{s,p} (\Rn)$ and $W^{s,p} (\Rn)$ is that in the latter there is no suitable concept of fractional gradient, even though it possibly is the natural space to define the fractional Laplacian.
Moreover, despite the analogy of the seminorms in those spaces,
\[
\| D^s u \|_{L^p (\Rn)} = c_{n,s} \left( \int_{\Rn} \left| \int_{\Rn} \frac{u(x)-u(y)}{|x-y|^{n+s}}\frac{x-y}{|x-y|} \, dy \right|^p dx \right)^{\frac{1}{p}} \quad \text{and} \quad [u]_{W^{s,p} (\Rn)} ,
\]
the fact that in $\| D^s u \|_{L^p (\Rn)}$ the absolute value affects the inner integral, while in $[u]_{W^{s,p} (\Rn)}$ the absolute value affects the integrand, reveals that the inclusions between these spaces are not obvious.
We mention, in passing, that in \cite{Adams,ShS2015} it is shown the embeddings $H^{s_2, p} (\Rn) \subset W^{s,p} (\Rn) \subset H^{s_1, p} (\Rn)$ for $0 < s_1 < s < s_2 <1$, as well as the equality $H^{s, 2} (\Rn) = W^{s,2} (\Rn)$.
This feature of the absolute value affecting the inner integral in $\| D^s u \|_{L^p (\Rn)}$ has several consequences in the proofs of properties of $H^{s,p} (\Rn)$, since, in general, it hampers a direct application of the elementary inequality that the absolute value of the integral is less than the integral of the absolute value, since that inequality cannot be reversed. 
For example, one cannot apply directly the techniques of \cite{Pon,Ponce2004}, which are suitable for seminorms in the style of $W^{s,p}$, but with general kernels.
Although the definitions of the Riesz gradient and the Bessel spaces are rather old, it was the study \cite{ShS2015} that initiated the attention in the community of nonlocal problems in partial differential equations and Calculus of Variations.
In fact, in \cite{ShS2015,ShS2018} it was shown the relationship between Riesz gradients and Bessel spaces, as well as a series of inequalities and embeddings mimicking those of Sobolev spaces, which constitute the basis for an analysis of the equations and minimization problems naturally related to the fractional gradient.

While \cite{ShS2018} treated the existence of minimizers for convex scalar problems using the direct method of the Calculus of Variations, the following two papers make extensions and applications for vectorial problems: in \cite{BeCuMC} we showed that the concept of polyconvexity is also suitable in these problems, while in \cite{KrSc22} it was shown the analogue for the concept of quasiconvexity.
These notions, polyconvexity and quasiconvexity, are classical in the Calculus of Variations (see, e.g., \cite{dacorogna}).
In these three works, the functional to minimize is of the form
\begin{equation}\label{eq:IntroDsu}
	\int_{\Rn} W(x, u(x), D^su(x)) \, dx ,
\end{equation}
with the integrand $W$ satisfying similar assumptions as in local problems.

%Of course, there is a myriad of nonlocal models in the Calculus of Variations apart from those based on the Riesz gradient $D^s u$ or in general nonloncal gradients $\mathcal{G}_{\rho} u$ as in \eqref{eq: general nonlocal gradient}.
%We mention here those related to the nonlocal model in Solid Mechanics known as Peridynamics, which was proposed by Silling \cite{Silling2000}; see also \cite{SiEpWeXuAs07,LeSi08,SiLe10}.
%One of its goals was to unify elastic and singularity phenomena, such as fracture.
%The development of this theory in the last years has been impressive.
%As general expositions, we can mention the review paper \cite{JaMoOtOt}, the two books \cite{Gerstle,madenci_oterkus} and the collaborative handbooks \cite{BoFoGeSi17,Voyiadjis19}.
%Several aspects of these models have been studied such as localization \cite{MeD,MeS,FossRaduYu}, existence and regularity \cite{FoRaWr20}, computational issues \cite{DeDuGlGuTiZh20,DeDuGun}, function spaces involved \cite{Mengesha12,MeD16}, or linear theories \cite{EmWe07,ZhDu10,DuGuLeZh13,TiQi13,MeDu14,DuTi18,ScMe19}.

This article can be seen as a follow-up to other works in the search of a nonlocal model suitable for hyperelasticity but also introducing a theory applicable to other phenomena. Following previous works \cite{BeMC14,BeMC18,BeMCPe} by some of the authors of this paper, we showed in \cite{BeCuMC20} that (bond-based) Peridynamics models based on energy functionals of the form 
\begin{equation*}
	\int_{\O} \int_{\O\cap B(x,\d)} w(x-y,u(x)-u(y)) \, dy \, dx,
\end{equation*}
although defined for bounded domains, do not fit in nonlinear Solid Mechanics, since very few local nonlinear models are limit of  nonlocal ones when $\d \to 0$.
As mentioned earlier, $\d$ is the \emph{horizon}: the interaction distance between the particles.
On the other hand, going back to the analysis in $H^{s,p} (\Rn)$ we showed in \cite{BeCuMC21} that the limit when $s \to 1$ of integral \eqref{eq:IntroDsu} based on the Riesz gradient is the local model
\[
\int_{\Rn} W(x, u(x), \nabla u(x)) \, dx .
\]

In this paper we propose a model that combines the good properties of the Riesz gradient and the space $H^{s,p}$ with the requirement that the energy is defined in a bounded domain of $\Rn$, since it is in this case where its interpretation of an elastic energy is physically meaningful, and so we think it could fit in state-based Peridynamics.

In order to reproduce similar existence results for energy functionals like 
\begin{equation*}
	\int_{\O} W(x, u(x), D_\d^su(x)) \, dx ,
\end{equation*}
it is necessary the study of the functional space associated to the nonlocal gradient \eqref{introduction non local gradient}. To be precise, we define the space $H^{s,p,\d}(\O)$ in concordance with Bessel and Sobolev spaces, as the completion of $C^{\infty}_c (\Rn)$ under the norm
\[
\lVert u\rVert_{H^{s,p,\d}(\O)} = \left( \left\| u \right\|_{L^p (\O_{\d})}^p + \left\| D_\delta^s u \right\|_{L^p (\O)}^p \right)^{\frac{1}{p}} ,
\]
where $\O_{\d}$ is the union of $\O$ with a tubular neighbourhood of the boundary of radius $\d$.
A subspace $H^{s,p,\d}_0 (\O)$ representing roughly $H^{s,p,\d} (\O)$ functions with zero `boundary' conditions (in truth, with zero values in another tubular neighbourhood of the boundary) is also studied.

We highlight as one of the main contributions of this work a nonlocal version of the Fundamental Theorem of Calculus, which is obtained despite the lack of homogeneity and semi-group properties of the kernels involved. As in other frameworks, given that it allows us to recover a function from its nonlocal gradient, we see it as a versatile tool that may prove itself useful in other situations. In particular, it helps in overcoming the aforementioned problem of an absolute value affecting the inner integral in the seminorms we are considering, and therefore, it is a key ingredient in the process of obtaining several inequalities and embeddings.

Thus, this article can be regarded as a first step to explore properties in $H^{s,p,\d}$ that are known in $W^{1,p}$ and $H^{s,p}$, such as integration by parts, Fundamental Theorem of Calculus, Poincar\'e inequalities and compact embeddings.
In fact, it is illustrative to compare those definitions and properties in the three contexts: classical, fractional and nonlocal. 
In what follows, \emph{classical} will typically refer to properties for Sobolev $W^{1,p}$ or even smooth functions involving the (classical or distributional) gradient $\nabla$, \emph{fractional} to properties in $H^{s,p}$ involving Riesz' $s$-fractional gradient $D^s$, and \emph{nonlocal} to properties in $H^{s,p,\d}$ involving the nonlocal gradient $D^s_{\d}$.
A partial list of this comparison is as follows:
\begin{itemize}
	\item Gradient. The classical gradient is just the pointwise or distributional gradient.
	The fractional gradient is \eqref{def: fractional gradient}, while the nonlocal gradient is \eqref{introduction non local gradient}.
	
	\item Divergence. The classical divergence is the pointwise divergence.
	The fractional divergence is
	\[
	\diver^s \phi (x) % = -c_{n,s} \pv_x \int\frac{\phi(x)+\phi(y)}{|x-y|^{n+s}}\cdot\frac{x-y}{|x-y|}dy 
	= c_{n,s} \int_{\Rn} \frac{\phi(x) - \phi(y)}{|x-y|^{n+s}}\cdot\frac{x-y}{|x-y|} dy ,
	\]
	while the nonlocal divergence is
	\begin{align*}
		\diver_{\delta}^s \phi(x) %& = -\pv_x c_{n,s} \int_{B(x, \d)} \frac{\phi(x)+\phi(y)}{|x-y|}\cdot\frac{x-y}{|x-y|}\frac{w_\d(x-y)}{|x-y|^{n-1+s}} \, dy \\ &
		= c_{n,s} \int_{B(x, \d)} \frac{\phi(x)-\phi(y)}{|x-y|}\cdot\frac{x-y}{|x-y|}\frac{w_\d(x-y)}{|x-y|^{n-1+s}} \, dy .
	\end{align*}
	
	\item Integration by parts. For $u \in C^1_c (\O)$ and $\phi \in C^1_c (\O, \Rn)$,
	\begin{align*}
		\text{Classical: } & \int_{\O} \nabla u \cdot \phi = - \int_{\O} u \diver \phi . \\
		\text{Fractional: } & \int_{\Rn} D^s u \cdot \phi = -\int_{\Rn} u \diver^s \phi . \\
		\text{Nonlocal: } & \int_{\O} D^s_{\d} u \cdot \phi = - \int_{\O} u \diver^s_{\d} \phi .
	\end{align*}
	
	\item Fundamental Theorem of Calculus.
	\begin{align*}
		\text{Classical: } & u(x) = \frac{1}{\sigma_{n-1}} \int_{\Rn} \nabla u(y) \cdot \frac{x-y}{|x-y|^n} \, dy . \\
		\text{Fractional: } & u (x) = c_{n,-s} \int_{\Rn} D^s u (y) \cdot \frac{x-y}{|x-y|^{n-s+1}} \, dy . \\
		\text{Nonlocal: } & u(x)=  \int_{\Rn} D_\d^s u (y) \cdot V_\d^s(x-y) \, dy .
	\end{align*}
\end{itemize}

At this point we mention the attempt to unify fractional and nonlocal theories recently explored in \cite{DeGuOlKa21}, in the context of a general vector calculus (following the earlier works \cite{GuLe10,DGLZ}) and in \cite{DeGuMeSc21}, which focuses on nonlocal gradients.
We also point out the work \cite{FeSu20}, where a different approach to the fractional fundamental theorem of calculus in dimension one is addessed, as well as a study of the function spaces involved.

The role of the Fourier transform in this analysis deserves a special mention.
It was pointed out in \cite{ShS2015} that the fractional gradient behaves nicely under Fourier transform: $\widehat{D^s u}(\xi) = 2\pi i\xi \left| 2\pi\xi \right|^{s-1}\hat{u} (\xi)$.
This fact was used in \cite{BeCuMC21} to obtain some properties that would otherwise require a much longer argument.
In this paper we also use Fourier transform, which is of no surprise having in mind the fundamental theorems of Calculus above expressing $u$ as a convolution, and, in fact, the constant presence of convolutions in this work.
%A neat expression for $\widehat{D^s_{\d} u}$, however, will not be available.
Again in \cite{ShS2015} the properties of the Riesz potential and its Fourier transform were used in connection with the Riesz gradient.
In this paper we also use a potential playing the role of Riesz'.
In our case, this potential will no longer have the semigroup property, but yet we will succeed in capturing its main features to prove the nonlocal fundamental theorem of Calculus.

The similarities with the fractional case show that a parallel theory could potentially be developed. Actually, it turns out that when restricted to $\Omega_\d$, functions from Bessel spaces belong to $H^{s,p,\d}(\O)$ (Proposition \ref{prop: simple inclussions}). This leads to wonder if functions in Bessel spaces $H^{s,p}(\Rn)$ are always an extension to $\Rn$ of functions in $H^{s,p,\d}(\O)$. This inclusion also implies that functions exhibiting fracture or cavitation phenomena are admitted in $H^{s,p,\d}(\O)$ \cite[Sect.\ 2.1]{BeCuMC}. 

The outline of the paper is the following.
Section \ref{se:notation} fixes some notation used throughout the article.
In Section \ref{se: functional analysis} the new versions of nonlocal gradient, divergence and integration by parts are established.
We also define the associated function space $H^{s, p, \d} (\O)$ and state its basic properties.
Section \ref{se: Non Local FTC} proves the nonlocal version of the fundamental theorem of Calculus.
Its proof, nevertheless, depends on the existence of the kernel $V^s_{\d}$, which is addressed in Section \ref{se: Inverse Kernel}.
Then, in Section \ref{se: Poincare} we first define the space $H^{s, p, \d}_0 (\O)$ and then use the nonlocal Fundamental Theorem of Calculus to prove versions in this context of the inequalities by Poincar\'e, Morrey, Trudinger and Hardy.
In Section \ref{se: Compact Embedding} we establish the compact embeddings from $H^{s, p, \d}_0 (\O)$ to $L^q (\O)$.
Section \ref{se: existence of minimizers} shows the existence of minimizers of scalar convex variational problems involving $D^s_{\d}$, as well as the corresponding Euler--Lagrange equation.
The article finishes with two appendices: in Appendix \ref{se:1D} we point out the necessary changes needed in Section \ref{se: Inverse Kernel} for the case $n=1$, while in Appendix \ref{se:Fourier} we state some Fourier analysis facts used throughout the paper for which we have not found a reference.

\section{Notation}\label{se:notation}

\subsection{General notation}

In all this work, we fix the dimension $n \in \N$ of the space ($n \geq 1$), an open bounded set $\O$ of $\Rn$ representing the body, a number $0< s < 1$ quantifying the degree of differentiability, a $\d>0$ indicating the horizon (the interaction distance between the particles of the body), and an exponent $1 \leq p < \infty$ of integrability.
Sometimes we will additionally require $p > 1$.
The H\"older conjugate exponent of $p$ is $p' = \frac{p}{p-1}$.

The notation for Sobolev $W^{1,p}$ and Lebesgue $L^p$ spaces is standard.
So is the notation for functions that are continuous $C$, and of class $C^k$ for $k$ an integer or infinity.
Their version of compact support are $C^k_c$.
The set of continuous functions vanishing at infinity is $C_0$.
We will indicate the domain of the functions, as in $C^1 (\O)$; the target is indicated only if it is not $\R$.
When using the norm in those spaces, the target is omitted, as in $\left\| \cdot \right\|_{L^p (\O)}$.

We write $B(x,r)$ for the open ball centred at $x \in \Rn$ of radius $r >0$.
The complement of a subset $A \subset \Rn$ is denoted by $A^c$, its closure by $\bar{A}$ and its boundary by $\p A$.

We denote by $\sigma_{n-1}$ the area of the unit sphere, while the surface area in integrals is indicated by $\mathcal{H}^{n-1}$.

We will use the multiindex notation: for $\a \in \N^n$, we give the standard meaning to the partial derivative $\p^{\a}$, the size $|\a|$, the monomial $x^{\a}$ for $x \in \Rn$, the ordering $\b \leq \a$ and the combinatorial number $\binom{\a}{\b}$; see, e.g., \cite[Sect.\ 2.2]{Grafakos08a}.

The vectors of the canonical basis of $\Rn$ are $e_j$, $j = 1, \ldots, n$.

The operation of convolution is denoted by $*$.
We indicate the duality product between tempered distributions and Schwartz functions as $\langle \cdot , \cdot \rangle$.

%a.e.

\subsection{Fourier transform}\label{subse:Fourier}

The convention for the Fourier transform of a function $f$ is
\begin{equation*}
\hat{f}(\xi)=\int_{\Rn} f(x) \, e^{-2\pi i x \cdot \xi} \, dx
\end{equation*}
for $f \in L^1 (\Rn)$.
This definition is extended by continuity and duality to other function and distribution spaces, notably, as isomorphisms in the Schwartz space $\mathcal{S}$ and in the space of tempered distributions $\mathcal{S}'$.
Sometimes we will also use the alternative notation $\F (f)$ for $\hat{f}$.
The variable in the Fourier space is generically designed by $\xi$.
The reflection of the function $f : \Rn \to \R$ is $\tilde{f} (x) = f (-x)$, and one has $\tilde{f}= \mathcal{F} (\hat{f})$, in principle, for functions $f \in L^1 (\Rn)$ for which  $\hat{f} \in L^1 (\Rn)$, but then by continuity and duality this property is extended to a larger class of functions and distributions.
Classical texts in Fourier analysis are \cite{duon2000,Grafakos08a}.

\subsection{Radial functions}

We recall the following definitions regarding radial functions.
\begin{definicion}\label{de:radial}
	We will say that
	\begin{enumerate}[a)]
		\item a function $f:\Rn \rightarrow \R$ is \emph{radial} if there exists $\bar{f}:[0,\infty) \to  \R$ such that $f(x)=\bar{f}(|x|)$ for every $x \in \Rn$. In such a case, $\bar{f}$ is the radial representation of $f$. 
		\item a radial function  $f:\Rn \rightarrow \R$ is \emph{radially decreasing} if its radial representation $\bar{f}:[0,\infty) \to  \R$ is a decreasing function.
		\item a function $\phi:\Rn \rightarrow \Rn$ is \emph{vector radial} if there exists a radial function $\bar{\phi}:[0,\infty)\to  \R$ such that $\phi(x)=\bar{\phi}(|x|)x$ for every $x \in \Rn$.
	\end{enumerate}
\end{definicion}

It is known (see, e.g., \cite[App.\ B.5]{Grafakos08a}) that the Fourier transform of a radial (respectively, vector radial) function is radial (respectively, vector radial).

\section{Function space: nonlocal gradient, divergence and integration by parts} \label{se: functional analysis}

In this section we define the nonlocal gradient and divergence, and state their basic properties, notably, the integration by parts.
We also set the natural functional space associated to the nonlocal gradient. The framework is the following. As typical in nonlocal models \cite{AnMaRoTo09,GuLe10,AkMe10,AkPa11,DGLZ,HiRa12}, `boundary' conditions are usually of volumetric type.
In our case, we fix a distance $\d>0$ and consider a bounded, open domain $\O \subset \Rn$.
The set $\O$ itself is regarded as a nonlocal interior domain, while $\O_{\d} := \Omega + B(0,\delta)$ is considered as its nonlocal closure. Accordingly, the set $\O_{B, \d} := \O_{\d} \setminus \O$ plays the role of nonlocal boundary; see Figure \ref{fi:Omega}.
The set $\O_{-\d} = \{ x \in \O : \dist(x, \p \O) > \d \}$ will also be relevant along this work. Thus, we consider $\d$ small enough so that $\O_{-\d}$ is not empty.
\begin{figure}[hb]
\begin{center}
	\begin{tikzpicture}
	\node (a) at (1.9,1) {$\O$};
	\node (b) at (1.9,-0.3) {{\footnotesize $\O_{B,\d}$}};
	\node (c) at (0.6,1.8) {$\O_{\d}$};
	\draw[double distance=1.4mm] plot [smooth cycle] coordinates {(0,0) (0.5,1) (1,1) (1.5,2) (2,2) (2.5,1.2) (3,0.7) (2.6, -0.1) (1.8, 0.1) (0.5,-0.3)};
	\draw (1.8,0.17) -- (1.8,0.03);
	\node[right] (d) at (1.73,0.07) {{\tiny $\d$}};
	\draw[densely dotted] (1.53,0.15) circle(0.14);
	\end{tikzpicture}
\end{center}
\small \caption{The sets $\O$, $\O_{\d}$ and $\O_{B,\d}$, together with the distance $\d$.\label{fi:Omega}}
\end{figure}
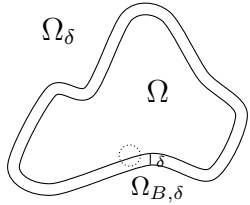

Let $w_\d: \Rn\to [0,+\infty)$ be a cut-off function, and  $\rho_\d: \Rn\to [0,+\infty)$ defined as 
\[\rho_\delta(x)=\frac{1}{\gamma (1-s)|x|^{n-1+s}}w_{\d}(x),\] 
where the constant $\gamma(s)$ is given by
\begin{equation} \label{eq: Riesz Potential constant}
	\gamma(s)=\frac{\pi^{\frac{n}{2}}2^s\Gamma\left(\frac{s}{2}\right)}{\Gamma\left(\frac{n-s}{2}\right)}
\end{equation}
and $\Gamma$ is Euler's gamma function. We assume the following conditions over $w_{\d}$:
\begin{enumerate}[a)]
	\item $w_\d$ is radial and nonnegative; $\bar{w}_\d$ is its radial representation.
	\item $w_\d \in C_c^\infty(B(0,\delta))$.
	\item There are constants $a_0>0$ and $0<b_0<1$ such that $0\leq w_\d \leq a_0$, with $w_\d =a_0$ in $B(0, b_0\delta)$.
	\item\label{item:wdecreasing} $\bar{w}_\d$ is decreasing.
%	\item $\displaystyle\int_{B(0,\delta)} \rho_\d(x) \, dx=1.$
\end{enumerate}

In fact, it will be apparent in the proof of Lemma \ref{le:radialsin} that condition \ref{item:wdecreasing}) can be considerably weakened.
Note, in addition, that $\rho_\delta \in L^1 (\Rn)$.

Given a function $f:\Omega \rightarrow \R$ and $x \in \O$ such that $f \in L^1(\O \setminus B(x,r))$ for every $r>0$, the principal value centred at $x$ of $\int_{\O} f $, denoted by
\[
\pv_x \int_{\O} f
\]
is defined as 
\[
\lim_{r\rightarrow 0} \int_{\O \setminus{B(x,r)} } f,
\]
whenever this limit exists. 

The definitions of the nonlocal gradient and divergence for smooth functions are the following.
\begin{definicion} \label{def: nonlocal gradient}
Set
	\begin{equation*}
		c_{n,s}:= \frac{n-1+s}{\gamma(1-s)}.
	\end{equation*}
	\begin{enumerate}[a)]
		\item \label{item:Dsdu} Let $u\in C_c^{\infty} (\Rn)$. The nonlocal gradient $D_\delta^s u$ is defined as
		\begin{equation} \label{eq: definition of nonlocal gradient}
			D_\delta^s u(x)= c_{n,s} \int_{B(x,\delta)} \frac{u(x)-u(y)}{|x-y|}\frac{x-y}{|x-y|}\frac{w_\d(x-y)}{|x-y|^{n-1+s}} \, dy , \qquad x \in \Rn.
		\end{equation}
		\item \label{item:divsdu} Let $\phi \in C^1_c (\Rn,\Rn)$. The nonlocal divergence is defined as
		\begin{equation*}
			\diver_{\delta}^s \phi(x)= -\pv_x c_{n,s} \int_{B(x, \d)} \frac{\phi(x)+\phi(y)}{|x-y|}\cdot\frac{x-y}{|x-y|}\frac{w_\d(x-y)}{|x-y|^{n-1+s}} \, dy , \qquad x \in \Rn .
		\end{equation*}
	\end{enumerate}
\end{definicion}
Notice that the integral in \eqref{eq: definition of nonlocal gradient} is absolutely convergent because $u$ is Lipschitz and $\rho_{\d} \in L^1 (\Rn)$.
It is also immediate from the definition that $\supp D_\delta^s u \subset \supp u + B (0, \d)$.
On the other hand, by odd symmetry,
\begin{equation} \label{eq: equiv integrals NL divergence}
	-\pv_x \int_{B(x,\delta)} \frac{\phi(x)+\phi(y)}{|x-y|}\cdot\frac{x-y}{|x-y|}\frac{w_\d(x-y)}{|x-y|^{n-1+s}} \, dy=
	\int_{B(x,\delta)} \frac{\phi(x)-\phi(y)}{|x-y|}\cdot\frac{x-y}{|x-y|}\frac{w_\d(x-y)}{|x-y|^{n-1+s}} \, dy
\end{equation}
and this last integral is absolutely convergent.

%Usaremos m\'as adelante $\left\| \diver_{\delta}^s \phi \right\|_{L^{\infty}} \leq \| D \phi \|_{L^{\infty}}$ for $\phi \in C^1 (\bar{\O}, \Rn)$.

Note also that, for each $x \in \O$,
\[
 \int_{B(x,\delta)} \frac{u(x)-u(y)}{|x-y|}\frac{x-y}{|x-y|}\frac{w_\d(x-y)}{|x-y|^{n-1+s}} \, dy = \int_{\O_{\d}} \frac{u(x)-u(y)}{|x-y|}\frac{x-y}{|x-y|}\frac{w_\d(x-y)}{|x-y|^{n-1+s}} \, dy ,
\]
and similarly for the integral in \eqref{eq: equiv integrals NL divergence}, since $B(x, \d) \subset \O_{\d}$ and $\supp w_\d \subset B(0, \d)$.

%We can make some additional comments about this definition. Notice that there is an analogy with the case of classical derivatives which are defined for points in the interior set. With respect to the integration domain, we can equally write $\Omega_\d$ as long as the kernel has compact support in the ball of radius $\delta$  since $x\in \Omega$. %, the integral can be defined in the whole ball for every $x\in \Omega$ as $u$ is defined for every $x$ in $\Omega_\d$. Actually it would be equivalent to define the integral over the whole $\Omega_C$ as long as the kernel has compact support in the ball of radius $\delta$.

%Definition \ref{def: nonlocal gradient}\,\ref{item:Dsdu}) naturally extends to vector fields, as shown in \cite{MeS}. Given $u \in  C^\infty_c(\Rn,\R^m)$ and, in the same setting of Definition \ref{def: nonlocal gradient}, its nonlocal gradient, for $x \in \O$, is
%\begin{equation} \label{eq:Ds nonlocal uvector}
%	D_\delta^s u(x)=  c_{n,s} \int_{B(x,\delta)} \frac{u(x)-u(y)}{|x-y|}\otimes\frac{x-y}{|x-y|}\frac{w_\d(x-y)}{|x-y|^{n-1+s}} \, dy .
%\end{equation}

The operators of Definition \ref{def: nonlocal gradient} are dual operators in the sense of integration by parts.
Actually, several versions of integration by parts formulas for related fractional or nonlocal operators have already appeared in the literature \cite{DGLZ,MeS,COMI2019,Silhavy2019}. For the purposes of this work, we will use a particular case of \cite[Th.\ 1.4]{MeS}, which, for convenience, we restate here in our context.

\begin{prop}\label{th:PartsMS}
Assume $u \in C^{\infty}_c (\Rn)$ and $\phi \in C^1_c (\O, \Rn)$.
Then
\[
 \int_{\O} \int_{\O} \frac{u(x)-u(y)}{|x-y|}\frac{x-y}{|x-y|} \cdot \phi (x) \rho_\d(x-y) \, dy \, dx = \int_{\O} u (x) \pv_x \int_{\O} \frac{\phi(x) + \phi(y)}{|x-y|} \cdot \frac{x-y}{|x-y|}\rho_\delta(x-y) \, dy \, dx .
\]
\end{prop}

The integration by parts formula of interest in this investigation is the following. Notice the presence of a boundary term, which is due to the fact that $u$ is not assumed to have compact support in $\O$. Note that the minus sign in the boundary term makes sense since the vector $x-y$ points inwards.

\begin{teo}\label{th:Nl parts}
Suppose that $u \in C^{\infty}_c (\Rn)$ and $\phi \in C^1_c (\O, \Rn)$.
Then $D^s_{\d} u \in L^{\infty} (\Rn, \Rn)$ and $\diver^s_{\d} \phi \in L^{\infty} (\Rn)$.
Moreover,
	\begin{align*}
		\int_{\O} D_\delta^s u(x) \cdot \phi(x) \, dx = - \int_{\O} u(x) \diver_\delta^s \phi (x) \, dx - (n-1+s) \int_{\O} \int_{\O_{B, \d}} \frac{u(y)\phi(x)}{|x-y|} \cdot \frac{x-y}{|x-y|}\rho_\delta(x-y) \, dy \, dx
	\end{align*}
and these three integrals are absolutely convergent.
\end{teo}
%***Este teorema solo lo necesitamos para $\phi \in C^1_c (\O_{-\d}, \Rn)$.
\begin{proof}
Denoting by $L>0$ the Lipschitz constant of $u$, we have, for each $x \in \Rn$,
\[
 \left| D_\delta^s u (x) \right| \leq c_{n,s} L \int_{B(x,\delta)} \frac{w_\d(x-y)}{|x-y|^{n-1+s}} \, dy = (n-1+s) L ,
\]
so $D^s_{\d} u \in L^{\infty} (\Rn, \Rn)$.
Analogously, the integral of the right-hand side of \eqref{eq: equiv integrals NL divergence} is absolutely convergent and  $\diver^s_{\d} \phi \in L^{\infty} (\Rn)$.

We have
\[
 \int_{\O} D_\delta^s u(x) \cdot \phi(x) \, dx = (n-1+s) \int_{\O} \int_{\O_\d} \frac{u(x)-u(y)}{|x-y|}\frac{x-y}{|x-y|}\cdot \phi(x) \rho_\delta(x-y) \, dy \,dx
\]
with
\begin{align*}
 \int_{\O} \int_{\O_\d} \frac{u(x)-u(y)}{|x-y|}\frac{x-y}{|x-y|}\rho_\delta(x-y) \cdot \phi(x) \, dy \,dx & = \int_{\O} \int_{\O} \frac{u(x)-u(y)}{|x-y|}\frac{x-y}{|x-y|} \cdot \phi(x) \rho_\delta(x-y)\, dy \, dx \\
& + \int_{\O} \int_{\O_{B, \d}} \frac{u(x)-u(y)}{|x-y|}\frac{x-y}{|x-y|} \cdot \phi(x) \rho_\delta(x-y) \, dy \, dx.
\end{align*}
By Proposition \ref{th:PartsMS},
\[
 \int_{\O} \int_{\O} \frac{u(x)-u(y)}{|x-y|}\frac{x-y}{|x-y|} \cdot \phi (x) \rho_\d(x-y) \, dy \, dx = \int_{\O} u (x) \pv_x \int_{\O} \frac{\phi(x) + \phi(y)}{|x-y|} \cdot \frac{x-y}{|x-y|}\rho_\delta(x-y) \, dy \, dx .
\]

On the other hand,
\begin{align*}
 - \int_{\O} u(x) \diver_\delta^s \phi (x) \, dx = (n-1+s) \int_{\O} u(x) \pv_x \int_{B(x, \d)} \frac{\phi(x)+\phi(y)}{|x-y|}\cdot\frac{x-y}{|x-y|} \rho_\d(x-y) \, dy \, dx .
\end{align*}
Now, for each $x \in \O$,
\begin{align*}
 \pv_x \int_{B(x, \d)} \frac{\phi(x)+\phi(y)}{|x-y|}\cdot\frac{x-y}{|x-y|} \rho_\d(x-y) \, dy = &\pv_x \int_{\O} \frac{\phi(x)+\phi(y)}{|x-y|}\cdot\frac{x-y}{|x-y|} \rho_\d(x-y) \, dy \\
 & + \pv_x \int_{\O_{B,\d}} \frac{\phi(x)+\phi(y)}{|x-y|}\cdot\frac{x-y}{|x-y|} \rho_\d(x-y) \, dy
\end{align*}
and, since $\phi$ vanishes in $\O_{B, \d}$,
\begin{align*}
 \pv_x \int_{\O_{B, \d}} \frac{\phi(x)+\phi(y)}{|x-y|}\cdot\frac{x-y}{|x-y|} \rho_\d(x-y) \, dy & = \int_{\O_{B, \d}} \frac{\phi(x)-\phi(y)}{|x-y|}\cdot\frac{x-y}{|x-y|} \rho_\d(x-y) \, dy
 \end{align*}
and this last integral is absolutely convergent, as explained in \eqref{eq: equiv integrals NL divergence}.

Putting together the formulas above, we have obtained that
\begin{align*}
 & \int_{\O} D_\delta^s u(x) \cdot \phi(x) \, dx + \int_{\O} u(x) \diver_\delta^s \phi (x) \, dx \\
 & \qquad \qquad = (n-1+s) \int_{\O} \int_{\O_{B, \d}} \left[ \frac{u(x)-u(y)}{|x-y|} \, \phi(x) - u(x) \, \frac{\phi(x)-\phi(y)}{|x-y|} \right]  \cdot \frac{x-y}{|x-y|} \, \rho_\d(x-y) \, dy \, dx ,
\end{align*}
being the three integrals absolutely convergent.
Finally, for each $x \in \O$ and $y \in \O_{B,\d}$,
\begin{equation}\label{eq:identityboundary}
 \frac{u(x)-u(y)}{|x-y|} \, \phi(x) - u(x) \, \frac{\phi(x)-\phi(y)}{|x-y|} = \frac{-u(y)}{|x-y|} \, \phi(x) + u(x) \, \frac{\phi(y)}{|x-y|} = \frac{-u(y)}{|x-y|} \, \phi(x) 
\end{equation}
since $\phi \in C^1_c (\O, \Rn)$.
This concludes the proof.
\end{proof}
Although we have proved $L^{\infty}$ regularity for $D^s_{\d} u$, much more is true, since Proposition \ref{Lemma: convolución con gradiente clásico} will show that $D_\delta^s u  \in C_c^{\infty} (\Rn)$.

%We have included the assumption $\phi \in C^1 (\overline{\O_{\d}}, \Rn)$ in Theorem \ref{th:Nl parts} instead of the more usual one  $C^{\infty}_c (\O, \Rn)$, since we will need that regularity in Theorem \ref{th:EL}. Of course, any funcion in $C^{\infty}_c (\O, \Rn)$ is identified with a function in $C^1 (\overline{\O_{\d}}, \Rn)$.

We now extend Definition \ref{def: nonlocal gradient}\,\emph{\ref{item:Dsdu})} to a broader class of functions.

\begin{definicion}\label{def: nonlocal gradient 2}
	\begin{enumerate}[a)]
		\item
		Let $u \in L^1 (\O_\d)$ be such that there exists a sequence of $\{ u_j \}_{j \in \N} \subset C^{\infty}_c (\Rn)$ converging to $u$ in $L^1 (\O_\d)$ and for which $\{ D_\d^s u_j \}_{j \in \N}$ converges to some $U$ in $L^1 (\O, \Rn)$.
		We define $D_\d^s u$ as $U$.
		
		\item\label{item:divs delta u}
		Let $\phi \in L^1 (\O_\d, \Rn)$ be such that there exists a sequence of $\{ \phi_j \}_{j \in \N} \subset C^{\infty}_c (\Rn, \Rn)$ converging to $\phi$ in $L^1 (\O_\d, \Rn)$ and for which $\{ \diver_\d^s \phi_j \}_{j \in \N}$ converges to some $\Phi$ in $L^1 (\O)$.
		We define $\diver_\d^s \phi$ as $\Phi$.
	\end{enumerate}
\end{definicion}

The following result shows that the above definitions are independent of the sequence chosen.

\begin{lem}%\label{le:Dusdelta well}
	\begin{enumerate}[a)]
		\item\label{item:Dsdeltau independent}
		Let $u \in L^1 (\O_\d)$ be such that there exist sequences $\{ u_j \}_{j \in \N} $ and $\{ v_j \}_{j \in \N}$ in $C^{\infty}_c (\Rn)$ such that $u_j \to u$ and $v_j \to u$ in $L^1 (\O_\d)$, and for which $D_\d^s u_j \to U$ and $D^s v_j \to V$ in $L^1 (\O, \Rn)$ as $j \to \infty$.
Then $U=V$ a.e.\ in $\O$.
		
		\item\label{item:divsdeltaphi independent}
		Let $\phi \in L^1 (\O_\d, \Rn)$ be such that there exist sequences $\{ \phi_j \}_{j \in \N} $ and $\{ \theta_j \}_{j \in \N}$ in $C^{\infty}_c (\Rn, \Rn)$ such that $\phi_j \to \phi$ and $\theta_j \to \phi$ in $L^1 (\O_\d, \Rn)$, and for which $\diver^s \phi_j \to \Phi$ and $\diver^s \theta_j  \to \Theta$ in $L^1 (\O)$ as $j \to \infty$. Then $\Phi = \Theta$ a.e.\ in $\O$.
	\end{enumerate}
\end{lem}
\begin{proof}
	We prove \emph{\ref{item:Dsdeltau independent})}, the proof of \emph{\ref{item:divsdeltaphi independent})} being analogous.
	Let $\phi \in C^1_c (\O, \Rn)$.
By Theorem \ref{th:Nl parts},
	\begin{align*}
		\int_\O U \cdot \phi &= \lim_{j \to \infty} \int_\O D^s_{\d} u_j \cdot \phi \\
		&= - \lim_{j \to \infty} \left(\int_\O u_j \diver^s_{\d} \phi + (n-1+s) \int_{\O} \int_{\O_{B, \d}} \frac{u_j(y) \, \phi(x)}{|x-y|} \cdot \frac{x-y}{|x-y|} \, \rho_\d(x-y) \, dy \, dx\right) \\
		&= - \left(\int_\O u \diver^s_{\d} \phi + (n-1+s) \int_{\O} \int_{\O_{B, \d}} \frac{u(y) \, \phi(x)}{|x-y|} \cdot \frac{x-y}{|x-y|} \, \rho_\d(x-y) \, dy \, dx\right) .
	\end{align*}
Among the previous limits, only the equality
\begin{equation}\label{eq:limitboundary}
 \lim_{j \to \infty} \int_{\O} \int_{\O_{B, \d}} \frac{u_j(y) \, \phi(x)}{|x-y|} \cdot \frac{x-y}{|x-y|} \, \rho_\d(x-y) \, dy \, dx = \int_{\O} \int_{\O_{B, \d}} \frac{u(y) \, \phi(x)}{|x-y|} \cdot \frac{x-y}{|x-y|} \, \rho_\d(x-y) \, dy \, dx
\end{equation}
requires some explanation, but, assuming its validity, since the same reasoning can be done for $V$, we conclude that
	\[
	\int_\O U \cdot \phi = \int_\O V \cdot \phi
	\]
	for all $\phi \in C^1_c (\O, \Rn)$, whence $U=V$ a.e.\ in $\O$.

It remains to justify limit \eqref{eq:limitboundary}.
For this, it is enough to show that the function
\[
 F(y) := \int_{\O} \frac{\phi(x)}{|x-y|} \cdot \frac{x-y}{|x-y|} \, \rho_\d(x-y) \, dx , \qquad y \in \O_{B,\d}
\]
is in $L^{\infty}(\O_{B,\d})$, and this is the case, since, denoting by $L$ the Lipschitz constant of $\phi$ and using that $\phi$ vanishes in $\O_{B,\d}$, we have
\[
 \left| F(y) \right| \leq \int_{\O} \frac{|\phi(x)- \phi(y)|}{|x-y|} \rho_\d(x-y) \, dx \leq L .
\]
This concludes the proof.
\end{proof}

It is quite natural to consider the space of $L^p$ functions whose nonlocal gradient is also an $L^p$ function. Taking into account the previous definitions, it is also natural to define such a space as the closure of smooth, compactly supported functions. Notice that this is analogous to the definition of the Bessel space $H^{s,p}(\Rn)$ \cite{BeCuMC21,COMI2019}.

\begin{definicion}\label{de:Hspd}
We define the space $H^{s,p,\d}(\O)$ as
	\begin{equation*}
		H^{s,p,\d}(\O):=\overline{C^{\infty}_c(\Rn)}^{\|\cdot\|_{H^{s,p,\d}(\O)}}
	\end{equation*}
	equipped with the norm 
	\begin{displaymath}
		\lVert u\rVert_{H^{s,p,\d}(\O)} = \left( \left\| u \right\|_{L^p (\O_{\d})}^p + \left\| D_\delta^s u \right\|_{L^p (\O)}^p \right)^{\frac{1}{p}} .
	\end{displaymath}
%	\textcolor{gray}{O tambi\'en podemos definir el espacio de la siguiente forma (aunque creo que es equivalente), que es lo que sal\'ia al probar la densidad en este espacio.
%		\begin{equation*}
%			H^{s,p,\d}(\O):=\overline{C^{\infty}(\overline{\O}_\d)}^{\|\cdot\|_{H^{s,p,\d}(\O)}}.
%	\end{equation*}}
\end{definicion}

Similarly to Sobolev spaces, this space satisfies reflexivity and separability properties.  
\begin{prop} \label{prop: espacio separable y reflexivo}
Let $1\leq p<\infty$. Then the space $H^{s,p,\d}(\Omega)$ is a separable Banach space.
Moreover, when $p>1$, it is reflexive.
\end{prop}
\begin{proof}

That $H^{s,p,\d}(\O)$ is a Banach space is immediate since its has been defined as a closure.

For the rest of the proof, we apply a standard argument; see, for example, \cite[Th.\ 2.1]{MeD} for the nonlocal case and \cite[Prop.\ 8.1]{Brezis} for the local case.

We have that the space $F_p = L^p(\O_{\d}) \times L^p(\Omega, \mathbb{R}^n)$ is separable and, if $p>1$, it is reflexive. Now we define the map $T:H^{s,p,\d}(\O)\rightarrow F_p$ by $T(u) = \left( u , D_\delta^s u \right)$.
Then $T$ is an isometry since
	\begin{displaymath}
	\lVert T(u)\rVert_{F_p}^p = \lVert u\rVert_{L^p(\O_{\d})}^p + \lVert D_\delta^s u\rVert_{L^p(\O)}^p = \lVert u\rVert_{H^{s,p,\d}(\O)}^p.
	\end{displaymath}
By Definitions \ref{def: nonlocal gradient 2} and \ref{de:Hspd}, it is clear that $T(H^{s,p,\d}(\O))$ is a closed subspace of $F_p$.
Since every closed subspace of a reflexive space is reflexive (see, e.g., \cite[Prop.\ 3.20]{Brezis}) and every subset of a separable space is separable (e.g., \cite[Prop.\ 3.25]{Brezis}), it follows that $T(H^{s,p,\d}(\O))$ is separable and, if $p>1$, it is reflexive.
The conclusion follows since $T$ is an isometry.
\end{proof}
In the next result we compare the spaces $H^{s,p,\d}(\O)$ for different exponents $p$, as well as with the better-known Bessel space $H^{s,p}(\Rn)$.
\begin{prop} \label{prop: simple inclussions}
	Let $1\leq q \leq p <\infty$. Then:
	\begin{enumerate}[a)]
		\item $H^{s,p,\d}(\O) \subset H^{s,q,\d}(\O)$. \label{simple inclussion 1}
		\item $H^{s,p}(\Rn) \subset H^{s,p,\d}(\O)$, with continuous embedding. \label{simple inclussion 2}
		% In particular, there exists $C>0$ such that for every $u \in H^{s,p}(\Rn)$,
		%\[\|u\|_{H^{s,p,\d}(\O)} \leq C \|u\|_{H^{s,p}(\Rn)} .\]
	\end{enumerate}
\end{prop}
\begin{proof}
	The proof of \emph{\ref{simple inclussion 1})} is obtained in a straightforward manner applying the known inclusions $L^p(\O_{\d}) \subset L^q(\O_{\d})$ and $L^p(\O) \subset L^q(\O)$ to the norms of $u$ and $D_\d^su$.
	
	Regarding \emph{\ref{simple inclussion 2})}, we first prove the corresponding inequality for smooth functions. Thus, let $u \in C^{\infty}_c(\Rn)$. We have that, for $x \in \O$,
	\begin{align*}
			D_\delta^s & u(x)= c_{n,s} \int_{B(x,\delta)} \frac{u(x)-u(y)}{|x-y|}\frac{x-y}{|x-y|}\frac{w_\d(x-y)}{|x-y|^{n-1+s}} \, dy \\
			&=c_{n,s} a_0\int_{\Rn} \frac{u(x)-u(y)}{|x-y|}\frac{x-y}{|x-y|}\frac{1}{|x-y|^{n-1+s} }\, dy- c_{n,s} \int_{\Rn} \frac{u(x)-u(y)}{|x-y|}\frac{x-y}{|x-y|}\frac{a_0-w_\d(x-y)}{|x-y|^{n-1+s}} \, dy \\
			& = a_0 D^s u (x) - c_{n,s} \int_{B(x, b_0 \d)^c} \frac{u(x)-u(y)}{|x-y|}\frac{x-y}{|x-y|}\frac{a_0-w_\d(x-y)}{|x-y|^{n-1+s}} \, dy \\
 & = a_0 D^s u (x) + c_{n,s} \int_{B(x, b_0 \d)^c} \frac{u(y)}{|x-y|}\frac{x-y}{|x-y|}\frac{a_0-w_\d(x-y)}{|x-y|^{n-1+s}} \, dy .
	\end{align*}
We recall that $a_0$ and $b_0$ are the constants from the definition of $w_\d$ and that $w_\d = a_0$ in $B(0,b_0 \d)$. Therefore, applying H\"older inequality, we have that
\[
 \left| D_\delta^s u(x) \right| \leq a_0 \left| D^s u (x) \right| + a_0 \, c_{n,s} \int_{B(x, b_0 \d)^c} \frac{|u(y)|}{|x-y|^{n+s}} \, dy  \leq a_0 \left| D^s u (x) \right| + c_1 \left\| u \right\|_{L^p (B(x, b_0 \d)^c)}
\]
for some constant $c_1>0$.
Consequently,
\[
  \left\| D_\delta^s u \right\|_{L^p (\O)} \leq a_0 \left\| D^s u \right\|_{L^p (\O)} + c_1 |\O|^{\frac{1}{p}} \left\| u \right\|_{L^p (B(x, b_0 \d)^c)}\leq c_2 \left\| u \right\|_{H^{s,p}(\Rn)}
\]
for some constant $c_2>0$.
Since we also have that $\|u\|_{L^p(\O_\d)} \leq \| u \|_{L^p(\Rn)}$, we obtain that there exists $C>0$ such that for every $u \in C^{\infty}_c(\Rn)$,
\[\|u\|_{H^{s,p,\d}(\O)} \leq C \|u\|_{H^{s,p}(\Rn)} .\]
Being the spaces $H^{s,p,\d}(\O)$ and $H^{s,p}(\Rn)$ defined as the closure of $C^{\infty}_c(\Rn)$ with their respective norms, the result follows. 
%Now we consider $u \in H^{s,p}(\Rn)$. Then, by density, there exists $\{u_j\}_{j\in \N} \subset C^{\infty}_c(\Rn)$ such that $u_j \to u $ in $H^{s,p}(\Rn)$. Furthermore, it holds that
%\[
%\|u_j\|_{H^{s,p,\d}(\O)} \leq C \|u_j\|_{H^{s,p}(\Rn)}
%\]
%for every $j \in \N$, which means that there exists a subsequence (not relabelled) and $v \in H^{s,p,\d}(\O)$ such that $u_j \to v$ in $H^{s,p,\d}(\O)$. As we already know that $u_j \to u$ in $L^p(\Rn)$, this implies that $u \in H^{s,p,\d}(\O)$ with $D^s_\d u=D^s_\d v$ a.e. (recall definition \ref{def: nonlocal gradient 2}) and the result follows.
\end{proof}

The inclusion of Proposition \ref{prop: simple inclussions}\,\emph{\ref{simple inclussion 2})} is one of the main motivations for the definition of our space, since it roughly suggests that $H^{s,p,\d}(\O)$ consists of Bessel $H^{s,p}$ functions defined only on $\O$ and without any integrability requirement at infinity.
As a matter of fact, the examples of \cite[Sect.\ 2.1]{BeCuMC} in the context of Solid Mechanics, together with the inclusion $H^{s,p}(\Rn) \subset H^{s,p,\d}(\O)$ show that $H^{s,p,\d}(\O, \Rn)$ contains deformations exhibiting fracture or cavitation, for some range of  $s$ and $p$.
One of the advantages of the space $H^{s,p,\d}(\O)$ is that, contrary to $H^{s,p}(\Rn)$, it contains linear functions, such as the identity, which would be relevant in a future linearization process.

\section{Nonlocal Fundamental Theorem of Calculus} \label{se: Non Local FTC}

We start by recalling the following classical representation theorem, which can be seen in \cite[Lemma 7.14]{GiTr01} or \cite[Prop.\ 4.14]{ponce_book}.
\begin{prop} \label{prop: classical representation result}
	For every $\varphi \in C_c^{\infty}(\Rn)$ and every $x \in \Rn$, we have 
	\begin{equation*}
	\varphi(x)=\frac{1}{\sigma_{n-1}} \int_{\Rn} \nabla \varphi(y) \cdot \frac{x-y}{|x-y|^n} \, dy .
	\end{equation*}
\end{prop}
This result may be understood as a fundamental theorem of Calculus, in the sense that we recover a function from its gradient by integration; more precisely, by convolution. A fractional version of it, involving the Riesz fractional gradient is also known \cite{ponce_book,ShS2015}.
This section is devoted to a novel nonlocal version of Proposition \ref{prop: classical representation result}, where a  function can be recovered from its nonlocal gradient $D^s_{\d}$ through a convolution with a suitable kernel $V^s_{\d}$.

Our approach is inspired by the proofs of the fractional fundamental theorem of Calculus previously referred in  \cite{ponce_book,ShS2015}.
However, those partly rely on the semigroup properties of Riesz potentials, which our kernels do not enjoy. Therefore, our procedure is much more involved. 

To begin with, we show that the kernel in the definition of $D_{\d}^su$ (see formula \eqref{eq: definition of nonlocal gradient}) can be seen, in a certain sense, as the gradient of a certain function. In order to do so, we introduce the following kernels.

\begin{definicion}\label{de:q}
%Let $0<s<1$ and $\d>0$.
Define
\[
 \bar{q}_\d : [0, \infty) \to \R, \qquad q_\d : \Rn \to \R \quad \text{and} \quad Q_\d^s : \Rn \setminus \{0\} \to \R
\]
as
	\begin{equation*}
	\bar{q}_\d (t)=(n-1+s)t^{n-1+s} \int_{t}^{\d}\frac{\bar{w}_\d(r)}{r^{n+s}} \, dr , \qquad q_{\d} (x) = \bar{q}_\d (|x|) \ \quad \text{and} \ \quad Q_\d^s(x)= \frac{1}{\gamma(1-s)|x|^{n-1+s}} q_\d(x).
	\end{equation*}
\end{definicion}

In the following result, $a_0$ and $b_0$ are the constants from the definition of $w_\d$.

\begin{lem} \label{lem: kernel primitive}
%Let $0<s<1$ and $\d>0$.
%Then:
\begin{enumerate}[a)]
\item\label{item:qa}
$\bar{q}_{\d}$ is $C^{\infty} ((0, \infty))$, $C^{n-1} ([0, \infty))$, $\supp \bar{q}_{\d} \subset [0, \d)$ and there exists a constant $z_0 \in \R$ such that for every $t \in [0,b_0 \d]$,
\begin{equation*}
	\bar{q}_\d(t)= a_0 +z_0 \, t^{n-1+s}.
\end{equation*}

\item\label{item:qb}
$Q_\d^s$ is radially decreasing, $Q_\d^s \in L^1 (\Rn)$, $\supp Q_\d^s \subset B(0, \d)$ and
	\begin{equation} \label{eq: kernel primitive}
		\frac{-1}{n-1+s}\nabla Q_\d^s(x)=\frac{\rho_\delta{(x)}}{|x|}\frac{x}{|x|}.
\end{equation}

\end{enumerate}
\end{lem}
\begin{proof}
We start with \emph{\ref{item:qa})}.
The function $\bar{q}_{\d}$ is clearly $C^{\infty}$ in $(0,\infty)$ as a product of $C^{\infty}$ functions in $(0,\infty)$.
We have that
	\begin{equation} \label{eq: derivative of q}
		\left(\frac{\bar{q}_\d(t)}{t^{n-1+s}}\right)'=-(n-1+s)\frac{\bar{w}_\d(t)}{t^{n+s}}, \qquad t > 0 .
	\end{equation}
Since $\bar{q}_{\d} (\d) = 0$ and $\supp \bar{w}_\d \subset [0, \d)$, we obtain that $\supp \bar{q}_{\d} \subset [0, \d)$.
Now, for $0 < t < \d b_0$ we have that
\[
 \left(\frac{\bar{q}_\d(t)}{t^{n-1+s}}\right)'=-(n-1+s)\frac{a_0}{t^{n+s}} = \left(\frac{a_0}{t^{n-1+s}}\right)' ,
\]
so the existence of $z_0$ in the statement follows.
In particular, $\bar{q}_{\d}$ is $C^{n-1} ([0, \infty))$.

We now show \emph{\ref{item:qb})}.
We get immediately from \eqref{eq: derivative of q} that $Q_\d^s$ is radially decreasing and
\begin{equation*}
	\nabla Q_\d^s(x)=-\frac{n-1+s}{\gamma(1-s)} \frac{\bar{w}_\d(|x|)}{|x|^{n+s}} \frac{x}{|x|} , \qquad x \in \Rn \setminus \{ 0 \} ,
\end{equation*}
so \eqref{eq: kernel primitive} holds.
As $\supp \bar{q}_{\d} \subset [0, \d)$ we obtain that $\supp Q_\d^s \subset B(0, \d)$.
Consequently, $Q_\d^s \in L^1 (\Rn)$ because of the boundedness of $\bar{q}_{\d}$.
\end{proof}

In the following proposition we write the nonlocal gradient as a convolution of the classical one with the kernel $Q_\d^s$.
Its fractional version can be found in \cite[Lemma 15.9]{ponce_book} and \cite[Th.\ 1.2]{ShS2015}.
In fact, our proof is inspired by that of \cite{ponce_book}: its idea is based on an integration by parts starting with Definition \ref{def: nonlocal gradient} and \eqref{eq: kernel primitive}.

\begin{prop} \label{Lemma: convolución con gradiente clásico}
	For every $u \in C_c^{\infty}(\Rn)$ and $x \in \Rn$ we have 
	\begin{equation}\label{eq:Dsd=-n}
	D_\delta^s u (x) = \int_{\Rn} \nabla u (y) \, Q_\d^s(x-y) dy 
	\end{equation}
and $D_\delta^s u  \in C_c^{\infty} (\Rn)$.
\end{prop}
\begin{proof}
Let $K$ be a ball containing $\supp u$ and let $K_{\d}=K+ B(0,\d)$. If $x \in K_{\d}^c$ then both terms of \eqref{eq:Dsd=-n} are zero since $\supp D_\d^s u \subset \supp u + B (0, \d) \subset K_{\d}$ and $\supp Q^s_{\d} \subset B (0, \d)$.
Thus, we consider $x\in K_{\d}$, $e \in \Rn$ with $|e| = 1$ and the vector field
	\[
	\beta : K_{\d} \setminus \{ x\} \rightarrow \Rn
	\] 
	defined by 
	\[
	\beta(y)=(u(x)-u(y)) \, Q_\d^s(x-y) \, e.
	\]
Let $\varepsilon > 0$ be such that $\bar{B} (x, \varepsilon) \subset K_{\d}$.
From Lemma \ref{lem: kernel primitive} we have that
	\begin{equation} \label{eq: divergence of the integrand}
	\diver \beta(y)=(n-1+s)\frac{u(x)-u(y)}{|x-y|} \rho_\delta(x-y) \frac{x-y}{|x-y|}\cdot e - Q_\d^s(x-y) \, \nabla u(y) \cdot e , \qquad y \in K_{\d} \setminus B(x, \varepsilon) 
	\end{equation}
and notice that $\diver \beta$ is integrable in $K_{\d} \setminus B(x, \varepsilon)$.
Applying the divergence theorem we obtain
\[
 \int_{K_{\d} \setminus B(x, \varepsilon)} \diver \beta (y) \, dy = \int_{\partial K_{\d}} \beta (y) \cdot \nu_y \, d \mathcal{H}^{n-1} (y) + \int_{\partial B(x, \varepsilon)} \beta (y) \cdot \frac{x-y}{|x-y|} \, d \mathcal{H}^{n-1} (y) ,
\]
where $\nu_y$ is the outer normal vector to $K_{\d}$.
Now we show that $\beta (y) = 0$ for all $y \in \partial K_{\d}$.
Indeed, if $x \in K_{\d}\setminus K$ then $u(x) = u(y) = 0$ for all $y \in \partial K_{\d}$, whereas if $x \in K$, then $Q_\d^s(x-y)=0$ for every $y \in \partial K_{\d}$.
Thus,
\[
 \int_{K_{\d} \setminus B(x, \varepsilon)} \diver \beta (y) \, dy = \int_{\partial B(x, \varepsilon)} \beta (y) \cdot \frac{x-y}{|x-y|} \, d \mathcal{H}^{n-1} (y) .
\]
We estimate the integrand in the right-hand side.
As $u$ is Lipschitz, using the mean value theorem and the definition of $Q_\d^s$ (see Definition \ref{de:q} and Lemma \ref{lem: kernel primitive}) we find that, for all $y \in \partial B(x, \varepsilon)$,
\begin{equation*}
 \left| \beta (y) \cdot \frac{x-y}{|x-y|}\right| \leq \left| \beta (y) \right| \leq \left\| \nabla u \right\|_{\infty} |x-y| \left| Q_\d^s(x-y) \right|  \leq \left\| \nabla u \right\|_{\infty}\frac{c}{|x-y|^{n+s-2}}  = \left\| \nabla u \right\|_{\infty}\frac{c}{\e^{n+s-2}} 
\end{equation*}
for some constant $c>0$, so
\begin{align*}
	\left|\int_{\partial B(x,\varepsilon)} \beta (y) \cdot \frac{x-y}{|x-y|} \, d \mathcal{H}^{n-1}(y)\right| \leq \left\| \nabla u \right\|_{\infty}  c \, \s_{n-1} \, \varepsilon^{1-s} ,
	\end{align*}
which goes to $0$ when $\varepsilon$ goes to $0$.
	Therefore, 
	\[
	\lim_{\varepsilon \rightarrow 0} \int_{K_{\d} \setminus B(x, \varepsilon)} \diver \beta (y) dy =0.
	\]
As a result, using \eqref{eq: divergence of the integrand} we obtain that
	\begin{equation*}
	\lim_{\varepsilon \rightarrow 0} \int_{K_{\d} \setminus B(x, \varepsilon)}(n-1+s)\frac{u(x)-u(y)}{|x-y|} \rho_\delta(x-y) \frac{x-y}{|x-y|}\cdot e \, dy = \lim_{\varepsilon \rightarrow 0} \int_{K_{\d} \setminus B(x, \varepsilon)} Q_\d^s(x-y) \, \nabla u(y) \cdot e \, dy,
	\end{equation*}
provided that both limits exists, which is actually true as both integrals are absolutely convergent in $K_{\d}$;
see the comment after Definition \ref{def: nonlocal gradient} for the left integral and notice that $Q_\d^s \in L^1 (\Rn)$ (see Lemma \ref{lem: kernel primitive}) for the right integral.
Thus,
\[
 \int_{K_{\d}}(n-1+s)\frac{u(x)-u(y)}{|x-y|} \rho_\delta(x-y) \frac{x-y}{|x-y|} \cdot e \, dy =  \int_{K_{\d}} Q_\d^s(x-y) \, \nabla u(y) \cdot e \, dy .
\]
As this is true for every $e \in \Rn$ with $|e|=1$, we conclude that
\[
 \int_{K_{\d}} (n-1+s)\frac{u(x)-u(y)}{|x-y|} \frac{x-y}{|x-y|}\rho_\delta(x-y) \, dy =  \int_{K_{\d}} \nabla u(y) \, Q_\d^s(x-y) \, dy ,
\]
and formula \eqref{eq:Dsd=-n} is proved.

We have thus shown that $D_\delta^s u  =  \nabla u * Q_\d^s$. As $Q_\d^s \in L^1 (\Rn)$, $\nabla u \in C^{\infty} (\Rn)$ and both have compact support, we conclude that $D_\delta^s u  \in C_c^{\infty} (\Rn)$.
\end{proof}
Proposition \ref{Lemma: convolución con gradiente clásico} shows that for $u \in C^{\infty}_c (\Rn)$, its nonlocal gradient $D^s_{\d} u$ is also in $u \in C^{\infty}_c (\Rn)$, unlike its fractional gradient $D^s u$, which is $C^{\infty}$ (see \cite[Prop.\ 5.2]{Silhavy2019} or \cite{MaSaPe99}) but not of compact support (an easy counterexample can be found in \cite[Sect.\ 2.2]{KrSc22}).
The reason for this difference is that, while Proposition \ref{Lemma: convolución con gradiente clásico} shows that $D_\delta^s u  =  \nabla u * Q_\d^s$ with $Q_\d^s$ of compact support, the fractional analogue states that $D^s u  =  \nabla u * I_{1-s}$, with $I_{1-s}$ the Riesz potential (see \eqref{eq: Riesz Potential} below for the definition), which is not of compact support.

We continue by restating the main result of Section \ref{se: Inverse Kernel} (Theorem \ref{Th: inverse Fourier trasnform as a function}).
\begin{prop} \label{pr:inversekernelbrief}
%Let $0<s<1$ and $\d>0$.
There exists a function $V_\d^s \in C^{\infty}(\Rn \setminus \{0\}, \Rn)$ such that
	\begin{equation} \label{eq: convolution of kernels}
	\int_{\Rn}V_\d^s(z) \, Q_\d^s(y-z) \, dz = \frac{1}{\sigma_{n-1}}\frac{y}{|y|^n}, \qquad y \in \Rn \setminus \{ 0 \} .
	\end{equation}
Moreover, $V_\d^s \in L^1_{\loc} (\Rn, \Rn)$, and for every $R>0$ there exists $M >0$ such that 
\begin{equation*}
	|V_\d^s(x)| \leq \frac{M}{|x|^{n-s}} , \qquad x \in B(0,R) \setminus \{ 0 \}.
\end{equation*}
\end{prop}
We will study further properties of $V_\d^s$ in Theorem \ref{Th: inverse Fourier trasnform as a function}.
The proof of Proposition \ref{pr:inversekernelbrief} is long and comprises the whole of Section \ref{se: Inverse Kernel}.
With this, the main result of this section (a nonlocal version of the fundamental theorem of Calculus) reads as follows. Its proof follows the lines from \cite[Prop.\ 15.8]{ponce_book}, whereas the main differences are gathered in Proposition \ref{pr:inversekernelbrief}.

\begin{teo} \label{Theo: nonlocal fundamental theorem of calculus}
%Let $0<s<1$ and $\d>0$.
Let $V_{\d}^s$ be the function of Proposition \ref{pr:inversekernelbrief}.
Then, for every $u \in C_c^\infty(\Rn)$ and $x \in \Rn$,
	\begin{equation} \label{eq: representation theorem}
	 u(x)=  \int_{\Rn} D_\d^s u (y) \cdot V_\d^s(x-y) \, dy .
	\end{equation}
\end{teo}
\begin{proof}
Let $F(x)$ denote the right hand side of \eqref{eq: representation theorem}.
This integral is absolutely convergent since $V_\d^s \in L^1_{\loc} (\Rn, \Rn)$ (Proposition \ref{pr:inversekernelbrief}) and $D_\d^s u$ is bounded with compact support (Proposition \ref{Lemma: convolución con gradiente clásico}).
In fact, Proposition \ref{Lemma: convolución con gradiente clásico} allows us to write the equality
	\begin{equation*}
	F(x)= \int_{\Rn} \int_{\Rn} \nabla u (z) \, Q_\d^s(y-z)\cdot V_\d^s(x-y)\,dz\,dy.
	\end{equation*}
Next we make the changes of variables $\eta=x-y$ and $\xi=x-z$ to obtain
	\begin{equation*}
	F(x)= \int_{\Rn}  \nabla u (x-\xi) \cdot \int_{\Rn}V_\d^s(\eta) \, Q_\d^s(\xi-\eta)\,d\eta\, d\xi .
	\end{equation*}
By Proposition \ref{pr:inversekernelbrief}, 
\[
	\int_{\Rn}V_\d^s(\eta) \, Q_\delta^s(\xi-\eta) \, d\eta = \frac{1}{\sigma_{n-1}}\frac{\xi}{|\xi|^n}.
\]
Thus, thanks to  Proposition \ref{prop: classical representation result},
	\begin{equation*}
	F(x)=\int_{\Rn}  \nabla u (x-\xi) \cdot \frac{\xi}{\sigma_{n-1}|\xi|^n} \, d\xi = u(x)
	\end{equation*}
and the proof is complete.
\end{proof}

\section{Existence of $V_\d^s$}\label{se: Inverse Kernel}

This section is devoted to the proof of Proposition \ref{pr:inversekernelbrief}, as well as to the derivation of further properties of $V_\d^s$.
The idea of the proof is to convert equation \eqref{eq: convolution of kernels} into
\[
 \hat{V}_{\d}^s (\xi) \, \hat{Q}_{\d}^s(\xi) = -i \frac{\xi}{|\xi|}\frac{1}{|2\pi \xi|}
\]
through Fourier transform (see Lemma \ref{lemma: Fourier transform vector Riesz potential}\,\emph{\ref{item:Fourier transform of x/|x|^n})} for the Fourier transform of the right-hand side of \eqref{eq: convolution of kernels}).
Thus, the candidate for $V_\d^s$ should satisfy
\[
 \hat{V}_{\d}^s (\xi)  = -i \frac{\xi}{|\xi|}\frac{1}{|2\pi \xi|} \frac{1}{\hat{Q}_{\d}^s(\xi)} .
\]
In the first half of the section, we show that $\hat{Q}_\d^s$ is positive and, consequently, the formula above is well defined. 
Taking inverse Fourier transform, we then conclude that $V_\d^s$ is at least a tempered distribution.
In the second half, we see that $V_\d^s$ is actually a function.

As seen in the introduction, it is illustrative to compare formula \eqref{eq: representation theorem} with the fractional fundamental theorem of Calculus in $H^{s,p}$ (see \cite[Th.\ 3.11]{COMI2019}, \cite[Prop.\ 15.8]{ponce_book} or \cite[Th.\ 1.12]{ShS2015}):
\begin{equation}\label{eq:FFTC}
 u (x) = c_{n,-s} \int_{\Rn} D^s u (y) \cdot \frac{x-y}{|x-y|^{n-s+1}} \, dy .
\end{equation}
Here $D^s u$ is Riesz' $s$-fractional gradient.
Thus, $D^s_{\d} u$ and $V^s_{\d}$ in our context play the role of $D^s u$ and $\frac{x}{|x|^{n-s+1}}$ in $H^{s,p}$, respectively.
In fact, in the analysis in $H^{s,p}$ an essential part is performed by the \emph{Riesz potential}.
We recall (see \cite{Stein70,ShS2015}) that given $0<s<n$, the Riesz potential $I_s:\Rn \setminus \{ 0 \} \rightarrow \R$ and its Fourier transform are 
\begin{equation}\label{eq: Riesz Potential}
I_{s}(x)=\frac{1}{\gamma(s)} \frac{1}{|x|^{n-s}} \quad \text{ and } \quad \hat{I_s}(\xi)= |2\pi \xi|^{-s} , 
\end{equation}
where $\gamma(s)$ is defined in \eqref{eq: Riesz Potential constant}.
A key study in a great part of this and the following section will be the comparison, first, of $\hat{I_s}$ with $\hat{V}^s_{\d}$, and, then, of $I_s$ with $V^s_{\d}$.
In fact, in the comment after Proposition \ref{Lemma: convolución con gradiente clásico} it was hinted that here $Q^s_{\d}$ plays the role of $I_{1-s}$ in the fractional case, so it becomes natural that in this section we also compare $\hat{Q}^s_{\d}$ with $\hat{I}_{1-s}$.

\subsection{Positivity of $\hat{Q}_\d^s$ and existence of $V_\d^s$ as a distribution}

We start with an analysis of the Fourier transform of the function $q_{\d}$ of Definition \ref{de:q}.

\begin{lem} \label{lem: Fourier transform of q delta}
%Let $0<s<1$ and $\d>0$. 
The function $\hat{q}_{\d}$ is analytic, $C_0(\Rn)$ and $L^1 (\Rn)$.
\end{lem}
\begin{proof}
Given that $\bar{q}_\d \in C^{n-1}_c([0, \infty))$ (see Lemma \ref{lem: kernel primitive}), we have that $q_{\d} \in L^1(\Rn)$ and has compact support.
Therefore, by known facts in Fourier analysis, $\hat{q}_{\d}$ belongs to $C_0(\Rn)$ and is analytic.

It remains to show that $\hat{q}_{\d} \in L^1 (\Rn)$, and for this we will previously check that $q_{\d} \in W^{2n-1,1}(\Rn)$. Indeed, as a consequence of Lemma \ref{lem: kernel primitive}, for $1 \leq j \leq 2n-1$ there exists $z_j \in \R$ such that
	\[
	\bar{q}_{\d}^{j)} (t) = z_j \, t^{n-1+s- j} , \qquad t \in (0, b_0 \d) ,
	\]
	where the superindex $j$ indicates the $j$-th derivative.
	On the other hand, $\bar{q}_{\d}^{j)}$ is bounded in $[b_0 \d, \d]$ and vanishes in $[\d, \infty)$.
	This implies that the a.e.\ and weak derivative of order $j$ of $q_{\d}$ coincide and they satisfy, for some constant $C_j>0$,
	\[
	\left| D^{j)} q_{\d} (x) \right| \leq C_j \left| x \right|^{n-1+s- j} , \qquad x \in B (0, b_0 \d) \setminus \{ 0 \} ,
	\]
while $\left| D^{j)} q_{\d} \right|$ is bounded in $B (0, \d) \setminus B (0, b_0 \d)$ and vanishes in $B (0, \d)^c$.
	This implies that $q_{\d} \in W^{2n-1,1}(\Rn)$.
	In particular, the Fourier transform of any partial derivative of order $2n-1$ of $q_{\d}$ is bounded, so there exists $C>0$ such that for any multiindex $\a$ of order $2n-1$ we have
	\[
	|(2\pi i \xi)^\a \hat{q}_{\d}(\xi)|=|\widehat{\partial^\a q_{\d}}(\xi) |\leq C ,
	\]
	and, hence,
	\[
	|\hat{q}_{\d}(\xi)|\leq \frac{C}{|2\pi \xi|^{2n-1}} .
	\]
	This decay at infinity of $\hat{q}_{\d}$, together with the fact that $\hat{q}_{\d}$ is continuous, implies that $\hat{q}_{\d} \in L^1(\Rn)$ for $n \geq 2$.

	In the rest of the proof, we assume that $n=1$.
	In this case, $q_{\d}$ is the even extension of $\bar{q}_{\d}$.
	As shown before, there exists $z_1\in \R$ such that $q_{\d}'(x)= \frac{z_1}{|x|^{1-s}}$ for $x \in B(0, b_0 \d)$.
If $z_1 = 0$ then $q_{\d}$ is $C^{\infty}_c (\Rn)$, so $\hat{q}_{\d}$ is in $\mathcal{S}$ and, in particular, in $L^1 (\Rn)$.
We assume from now on that $z_1 \neq 0$.
	
Consider a $\f \in C_c^{\infty}(\R)$ with $\f |_{B(0, \frac{1}{4})} = 1$  and  $\f |_{B(0, \frac{1}{2})^c} = 0$.
	Then,
	\begin{align*}
		|2\pi \xi|^{-s} - \frac{1}{z_1 \gamma(s)}\widehat{q'}_\d(\xi) & = \mathcal{F}\left(\frac{1}{\gamma(s)|x|^{1-s}} - \frac{1}{z_1 \gamma(s)}q_{\d}'(x) \right)  \\
		& = \mathcal{F}\left(\frac{\f}{\gamma(s)|x|^{1-s}} - \frac{1}{z_1 \gamma(s)}q_{\d}'(x) \right) +\mathcal{F}\left(\frac{1-\f}{\gamma(s)|x|^{1-s}} \right) .
	\end{align*}
	Looking at the expression of $q_{\d}'$, we notice that the functions $\frac{ \f}{\gamma(s)|x|^{1-s}}$ and $\frac{1}{z_1 \gamma(s)} q'_\d(x)$ coincide in $B(0, \min\{b_0 \d, \frac{1}{4} \})$, and both have compact support.
	Therefore, its difference is a smooth function of compact support.
	In particular, it is in the Schwartz space, as well as its Fourier transform:
	\[
	\mathcal{F}\left(\frac{ \f}{\gamma(s)|x|^{1-s}} - \frac{1}{z_1 \gamma(s)}q'_\d(x) \right) \in \mathcal{S} .
	\]
	On the other hand, the function $\mathcal{F}\left( \frac{1-\f}{\gamma(s)|x|^{1-s}} \right)$ is treated in \cite[Ex.\ 2.4.9]{Grafakos08a}, and it is concluded that its decay at infinity is faster than any negative power of $|\xi|$. Consequently, the decay at infinity of
	\[
	|2\pi \xi|^{-s} - \frac{1}{z_1 \gamma(s)}\widehat{q'}_\d(\xi)
	\]
	is also faster than any negative power of $|\xi|$.
	In particular, there exists $C_1'>0$ such that
	\[
	\left| \frac{2\pi \xi}{z_1 \gamma(s)}\hat{q}_{\d}(\xi) \right| = \left| \frac{1}{z_1 \gamma(s)}\widehat{q'}_\d(\xi) \right| \leq \frac{C_1'}{|2 \pi \xi|^s} ,
	\]
	which allows us to conclude that $\hat{q}_{\d} \in L^1(\R)$.
\end{proof}

In the following result we obtain relevant properties about $\hat{Q}_\d^s$.
Recall that $a_0$ is the constant from the definition of $w_\d$.

\begin{prop} \label{Prop: properties of the Fourier transform of Q}
%	Let $0<s<1$ and $\d>0$. Then
	\begin{enumerate}[a)]
		\item $\hat{Q}_{\d}^s$ is analytic, bounded, radial, and $\hat{Q}_{\d}^s(0)=\| Q_\d^s \|_{L^1(\Rn)}$. \label{property a)}
		
		\item $\partial^{\alpha} \hat{Q}_\d^s$ is bounded for every multiindex $\alpha$. \label{property b)}

		\item $\displaystyle \lim_{|\xi|\rightarrow \infty} \frac{\hat{Q}_{\d}^s(\xi)}{|2\pi \xi|^{-(1-s)}}=a_0$. \label{property c)}
	\end{enumerate}
\end{prop}
\begin{proof}	
The proof of part \emph{\ref{property a)})} comes directly from known facts in Fourier analysis.
Indeed, as $Q_\d^s \in L^1(\Rn)$ we have $\hat{Q}_{\d}^s \in L^\infty(\Rn)$.
As $Q_\d^s$ has compact support, $\hat{Q}_{\d}^s$ is analytic.
Since $Q_\d^s$ is radial, so is $\hat{Q}_{\d}^s$.
Finally, the equality $\hat{Q}_{\d}^s(0)=\| Q_\d^s \|_{L^1(\Rn)}$ is a straightforward consequence of the formula of the Fourier transform.

Regarding \emph{\ref{property b)})}, we have that $\partial^{\alpha} \hat{Q}_\d^s = \mathcal{F}((-2\pi i \xi)^{\alpha} Q_\d^s)$. Thus, $\partial^{\alpha} \hat{Q}_\d^s$ is the Fourier transform of an $L^1 (\Rn)$ function (since $Q_\d^s \in L^1 (\Rn)$ has compact support).
Therefore, $\partial^{\alpha} \hat{Q}_\d^s$ is bounded.
	
In order to show \emph{\ref{property c)})}, we apply the Fourier transform to the expression $Q_\d^s = I_{1-s} \, q_{\d}$ (see Definition \ref{de:q}). Since the Riesz potential $I_{1-s}$ is not an $L^1(\Rn)$ function and $q_{\d}$ is not Schwartz, the Fourier transform is, in principle, in the sense of tempered distributions.
To wit, as $I_{1-s} \in L^1(B(0,1))+L^\infty(B(0,1)^c)$, both factors $I_{1-s}$ and $q_{\d}$ can be seen as distributions; in addition, $q_{\d}$ has compact support, so we can use Lemma \ref{lem: Fourier transform of convolution of distributions} and obtain that
\begin{equation} \label{eq: conv hat Q}
	\hat{Q}_\d^s(\xi)=|2\pi \xi|^{-(1-s)}*\hat{q}_{\d}(\xi)
\end{equation}
in the sense of distributions.
Actually, by Young's inequality for the convolution we have that
\[
 \| \hat{I}_{1-s}*\hat{q}_{\d} \|_{L^{\infty} (\Rn)} \leq \| \hat{I}_{1-s} \|_{L^1 (B(0,1))} \left\| \hat{q}_{\d}\right\|_{L^{\infty} (\Rn)} + \| \hat{I}_{1-s} \|_{L^{\infty} (B(0,1)^c)} \left\| \hat{q}_{\d} \right\|_{L^1 (\Rn)} .
\]
Therefore, the integral defining $(\hat{I}_{1-s}*\hat{q}_{\d}) (\xi)$ is absolutely convergent for a.e.\ $\xi \in \Rn$. Consequently, equality \eqref{eq: conv hat Q} holds a.e.
%, and, since $\hat{Q}_\d^s$ is continuous, it holds everywhere. Esto lo he quitado porque creo que es falso, o, al menos, dif�cil de probar
 
 Then, we consider $\xi=\lambda \xi_0$ with $\xi_0 \in B(0, 1)^c$ fixed and $\lambda > 0$. 
 Using the change of variables $x=\lambda x'$ we have
\begin{align*}
\hat{Q}_\d^s(\lambda \xi_0)&= \int_{\Rn} |2 \pi (x - \lambda \xi_0)|^{-(1-s)} \hat{q}_{\d} (x) \, dx =  \int_{\Rn} |2 \pi (\lambda x - \lambda \xi_0)|^{-(1-s)} \hat{q}_{\d} ( \lambda x) \lambda^n \, dx \\
&= \lambda^{-(1-s)} \int_{\Rn} |2 \pi ( \xi_0 - x)|^{-(1-s)} \hat{q}_{\d} ( \lambda x) \lambda^n \, dx.
\end{align*}
As the function $\xi \mapsto \frac{\hat{Q}_{\d}^s(\xi)}{|2\pi \xi|^{-(1-s)}}$ is radial, in order for \emph{\ref{property c)})} to hold, it is enough that
\[
 \lim_{\l \to \infty} \frac{\hat{Q}_{\d}^s(\l \xi_0)}{|2\pi \l \xi_0|^{-(1-s)}}=a_0, 
\]
equivalently,
\[
 \lim_{\l \to \infty} \frac{\int_{\Rn} |2 \pi ( \xi_0 - x)|^{-(1-s)} \hat{q}_{\d} ( \lambda x) \lambda^n \, dx}{|2\pi \xi_0|^{-(1-s)}}=a_0. 
\]
Define now $g_\lambda(x)= \frac{1}{a_0} \hat{q}_{\d}(\lambda x) \lambda^n$ and $f(\xi)=|2\pi \xi|^{-(1-s)}$.
The limit above is equivalent to
\[
 \lim_{\l \to \infty} \frac{\int_{\Rn} f (\xi_0 - x) g_{\l} (x) \, dx}{f(\xi_0)}=1, 
\]
and, in turn, equivalent to
\[
 \lim_{\l \to \infty} \int_{\Rn} f (\xi_0 - x) g_{\l} (x) \, dx = f(\xi_0), 
\]
in other words,
\begin{equation}\label{eq:contconv}
  \lim_{\l \to \infty} f*g_\lambda(\xi_0) = f(\xi_0). 
\end{equation}

We recall from Lemma \ref{lem: kernel primitive} that $a_0 = \bar{q}_{\d} (0) = q_{\d} (0) = \int_{\Rn} \hat{q}_{\d}$; note that $\hat{q}_{\d} \in L^1 (\Rn)$ thanks to Lemma \ref{lem: Fourier transform of q delta}.
Thus, $\int_{\Rn} g_\lambda =1$ for each $\l >0$.
Then, by construction, $g_\lambda$ is a mollifier family tending to the Dirac delta at $0$, when $\lambda \rightarrow \infty$ in the sense of distributions.
Thus,
\begin{equation*}
\left| f*g_\lambda(\xi_0) - f(\xi_0) \right|= \left|\int_{\Rn} \left[ f(\xi_0 -x) - f(\xi_0) \right] g_\lambda(x) \, dx \right|\leq\int_{\Rn} \left|f(\xi_0 -x) - f(\xi_0)\right| |g_\lambda(x)| \, dx .
\end{equation*}
Let $\e>0$.
Since $f$ is uniformly continuous in $B(0,1/2)^c$, there exists $0 < r < \frac{1}{2}$ such that
\[
|f(\xi_0-x)- f(\xi_0)| < \e , \qquad \text{ for all } \xi_0 \in B(0,1)^c \text{ and } x \in B(0,r).
\]
Therefore, as $f \in L^\infty(B(0,1/2)^c)$,
\begin{align*}
\left| f*g_\lambda(\xi_0) - f(\xi_0) \right| &\leq\int_{B(0,r)} \left|f(\xi_0 -x) - f(\xi_0)\right| |g_\lambda(x)| \, dx  + \int_{B(0, r)^c} \left|f(\xi_0 -x) - f(\xi_0)\right| |g_\lambda(x)| \, dx  \\
&\leq \e \int_{B(0,r)}|g_\lambda(x)| \, dx + 2 \|f\|_{L^\infty(B(0,1/2)^c)} \int_{B(0, r)^c} |g_\lambda(x)| \, dx .
\end{align*}
Finally, we use that $\lim_{\lambda \rightarrow \infty} \int_{B(0, r)^c} |g_\lambda(x)| \, dx =0$. As a result, there exists $\lambda_0>0$ such that for every $\lambda> \lambda_0$, the inequality $\int_{B(0, r)^c} |g_\lambda(x)| \, dx<\e$ holds.
Consequently,
\[
\left| f*g_\lambda(\xi_0) - f(\xi_0) \right| \leq \left( \| g_\lambda \|_{L^1(\Rn)} +   2 \|f\|_{L^\infty(B(0,1/2)^c)} \right) \e . 
\]
As $\| g_\lambda \|_{L^1(\Rn)} = \| g_1 \|_{L^1(\Rn)}$, this proves convergence \eqref{eq:contconv}, and, hence, statement \emph{\ref{property c)})}.
\end{proof}

The following calculation is useful for showing the positivity of $\hat{Q}_{\d}^s$.

\begin{lem}\label{le:radialsin}
%	Let $0<s<1$ and $\d>0$. Then,
For all $j \in \{ 1, \ldots, n\}$ and $r >0$,
\[
\int_{\Rn} \frac{x_j}{|x|^{n+s+1}} w_\delta(x) \sin{(2\pi r x_j)} \, dx>0.
\]
\end{lem}
\begin{proof}
The integral is absolutely convergent since
\begin{equation}\label{eq:dominatedsin}
 \left| \frac{x_j}{|x|^{n+s+1}}w_\d(x) \sin{(2\pi r x_j)} \right| \leq  \frac{1}{|x|^{n-1+s}} \, w_\d(x) \, 2\pi r 
\end{equation}
and $w_\d$ as compact support.
By a change of variables we have
\begin{equation*}
 \int_{\Rn} \frac{x_j}{|x|^{n+s+1}}w_\d(x) \sin{(2\pi r x_j)} \, dx = r^s \int_{\Rn} \frac{x_j}{|x|^{n+s+1}}w_\d\left(\frac{x}{r} \right) \sin{(2 \pi x_j)} \, dx.
\end{equation*}
Recall that $\bar{w}_\d$ is the radial representation of $w_\d$.
By symmetry, the co-area formula and Fubini's theorem we have
\begin{align*}
 \int_{\Rn} \frac{x_j}{|x|^{n+s+1}}w_\d\left(\frac{x}{r} \right) \sin{(2\pi  x_j)} \, dx &=2 \int_{\{ x_j > 0\}} \frac{x_j}{|x|^{n+s+1}}w_\d\left(\frac{x}{r} \right) \sin{(2\pi  x_j)} \, dx \\
 & = 2 \int_{0}^{\infty}\frac{\bar{w}_\d \left(\frac{t}{r}\right)}{t^{n+s+1}} t^{n-1} \int_{\mathbb{S}_j^+} t z_j \sin(2 \pi t z_j) \, d \mathcal{H}^{n-1}(z) \, dt \\
 &=2 \int_{\mathbb{S}_j^+} z_j \int_{0}^{\infty} \frac{\bar{w}_\d \left(\frac{t}{r}\right)}{t^{s+1}} \sin(2 \pi t z_j) \, dt \, d \mathcal{H}^{n-1}(z) ,
 \end{align*}
where $\mathbb{S}_j^+ = \{ z \in \Rn : \, |z| = 1 , \, z_j > 0\}$. 
Finally, let us show that
\begin{equation}\label{eq:intpos}
 \int_{0}^{\infty} \frac{\bar{w}_\d \left(\frac{t}{r}\right)}{t^{s+1}} \sin(2 \pi t z_j) \, dt > 0
\end{equation}
for each $z \in \mathbb{S}_j^+$.
For this, consider the function $f(t) =  \frac{\bar{w}_\d \left(\frac{t}{r}\right)}{t^{s+1}}$ and express
\[
 \int_{0}^{\infty} \frac{\bar{w}_\d \left(\frac{t}{r}\right)}{t^{s+1}} \sin(2 \pi t z_j) \, dt = \sum_{k=0}^{\infty} \int_{\frac{k}{z_j}}^{\frac{k+1}{z_j}} f(t) \sin(2 \pi t z_j) \, dt .
\]
We have that each term in the sum is positive; indeed,
by splitting the integral in two through point $\frac{k+\frac{1}{2}}{z_j}$ and making the change of variables $t=t'+ \frac{1}{2z_j}$ in one of them, it is easy to obtain
\[
 \int_{\frac{k}{z_j}}^{\frac{k+1}{z_j}} f(t) \sin(2 \pi t z_j) \, dt = \int_{\frac{k}{z_j}}^{\frac{k+\frac{1}{2}}{z_j}} \left[ f(t) - f(t + \frac{1}{2z_j}) \right] \sin(2 \pi t z_j) \, dt \geq 0 ,
\]
since $\sin(2 \pi t z_j) > 0$ and $f$ is decreasing (as so is $\bar{w}_\d$).
In fact,
\[
 \int_{0}^{\frac{1}{2 z_j}} \left[ f(t) - f(t + \frac{1}{2z_j}) \right] \sin(2 \pi t z_j) \, dt >0 ,
\]
as $f$ is strictly decreasing in $[0, r b_0 \d]$, so \eqref{eq:intpos} holds, which concludes the proof.
\end{proof}

As can be seen from the proof above, the assumption that $\bar{w}_\d$ is decreasing can be weakened to the following: the function $f$ of the proof is decreasing in $t$ for all $r>0$.
This is equivalent to $f'(t) \leq 0$ for all $t >0$ and $r>0$, which in turn, is equivalent to the differential inequality
\[
 \bar{w}'_\d (t) \leq (s+1) \frac{\bar{w}_\d (t)}{t} , \qquad t \geq b_0 \d .
\]
Therefore, assumption \ref{item:wdecreasing}) on $\bar{w}_\d$ (see Section \ref{se: functional analysis}) could have been replaced with the above inequality, but, for ease of reading, we prefered to state that $\bar{w}_\d$ is decreasing.

The following result shows the convergence of the truncations of $\nabla Q_\d^s$.

\begin{lem} \label{lem: convdist pv dQ}
%	Let $0<s<1$ and $\d>0$. Then
The function $\nabla Q_\d^s$ can be identified with the tempered distribution defined componentwise as
\begin{equation}\label{eq:nablaQ}
 \langle \partial_j Q_\d^s , \f \rangle = - c_{n,s} \int_{\{ x_j>0\}} \frac{x_j}{|x|^{n+s+1}} w_\d(x) (\f(x)-\f(-x) )\, dx , \qquad j \in \{1, \ldots, n\} 
\end{equation}
and we have the convergence
\[
 \nabla Q_\d^s \chi_{B(0, \e)^c} \to \nabla Q_\d^s \qquad \text{in } \mathcal{S}' \quad \text{as } \e \to 0 .
\]
\end{lem}
\begin{proof}
	We recall from Lemma \ref{lem: kernel primitive} that 
\[
	 \nabla Q_\d^s (x) = -(n-1+s) \frac{\rho_\delta (x) }{|x|} \frac{x}{|x|} = - c_{n,s} \frac{x}{|x|^{n+s}} w_{\d} (x) .
\]
Thus, $\nabla Q_\d^s \chi_{B(0, \e)^c}$ is in $L^1 (\Rn)$ for each $\e>0$, so it can be identified with a tempered distribution.
Let $j \in \{1, \ldots, n\}$; we shall prove the desired convergence for the $j$-th component of $\nabla Q_\d^s$.

Using the notation $B^{\pm}_j (0,\e)^c=\{ x \in B(0,\e)^c : \pm x_j>0 \}$, we have 
	\begin{align*}
	\int_{B(0,\e)^c} \frac{x_j}{|x|^{n+s+1}} w_\d(x) \f(x) \, dx &=\int_{B^-_j (0,\e)^c} \frac{x_j}{|x|^{n+s+1}} w_\d(x) \f(x) \, dx +\int_{B^+_j (0,\e)^c} \frac{x_j}{|x|^{n+s+1}} w_\d(x) \f(x) \, dx \\
	&= \int_{B^+_j (0,\e)^c} \frac{x_j}{|x|^{n+s+1}} w_\d(x) (\f(x)-\f(-x) )\, dx .
	\end{align*}
%	after the change of variable $x=-\bar{x}$.
	By the mean value theorem, 
	\[
	\left|\frac{x_j}{|x|^{n+s+1}} w_\d(x) (\f(x)-\f(-x) ) \chi_{B^+_j (0,\e)^c} (x) \right|\leq
	\frac{2 \| \nabla \f\|_{\infty} \|  w_\d\|_{\infty}}{|x|^{n-1+s}} \chi_{B (0,\d)} (x) .
	\]
This shows that formula \eqref{eq:nablaQ} defines a tempered distribution; moreover, by dominated convergence we obtain that
	\[
	\int_{B^+_j (0,\e)^c} \frac{x_j}{|x|^{n+s+1}} w_\d(x) (\f(x)-\f(-x) )\, dx  \to \int_{\{ x_j>0\}} \frac{x_j}{|x|^{n+s+1}} w_\d(x) (\f(x)-\f(-x) )\, dx  
	\]
as $\e \to 0$. This proves the desired convergence.
\end{proof}

We now show the positivity of $\hat{Q}_{\d}^s$.
\begin{prop} \label{prop: Fourier transform well defined}
$\hat{Q}_{\d}^s (\xi) > 0$ for all $\xi \in \Rn$.
\end{prop}
\begin{proof}
Since $\hat{Q}_\d^s(0)=\| Q_\d^s\|_{L^1(\Rn)} > 0$, we have to show that $\hat{Q}_\d^s(\xi)>0$ for every $\xi \in \Rn \setminus \{ 0\}$. For this, we fix any $j \in \{1, \ldots, n \}$ and claim that, despite $\frac{\partial Q_\d^s}{\partial x_j} \notin L^1 (\Rn)$,
\begin{equation} \label{eq: Fourier transform of derivative}
 \widehat{ \frac{\partial Q_\d^s}{\partial x_j} } (\xi) = \frac{(n-1+s)}{\gamma(1-s)}i \int_{\Rn} \frac{x_j}{|x|^{n+s+1}} w_\delta(x) \sin{(2\pi  \xi \cdot x)} \, dx.
\end{equation}
This is shown at the end of the proof.
Assuming the validity of \eqref{eq: Fourier transform of derivative}, by Lemma \ref{le:radialsin} we obtain
\[
\frac{1}{i} \widehat{ \frac{\partial Q_\d^s}{\partial x_j} } (\xi_j e_j) > 0 , \qquad \xi_j > 0.
\]
Now, the formula
\[
2\pi i \xi_j \hat{Q}_{\d}^s (\xi_j e_j)= \widehat{\frac{\partial Q_\d^s}{\partial x_j}}(\xi_j e_j)
\]
holds in the sense of tempered distributions.
Since both terms are actually functions, the equality holds as functions for almost every point.
Moreover, since both functions are continuous, the equality holds everywhere.
We then conclude that
\[
\xi_j \hat{Q}_{\d}^s (\xi_je_j) > 0  , \qquad \xi_j > 0 .
\]
Consequently, since $\hat{Q}_\d^s$ is radial, $\hat{Q}_\d^s(\xi)>0$ for all $\xi \in \Rn$.

It remains to prove \eqref{eq: Fourier transform of derivative}.
By Lemma \ref{lem: kernel primitive},
\[
 \frac{\partial Q_\d^s}{\partial x_j} (x) = -(n-1+s) \frac{\rho_\delta (x) }{|x|} \frac{x_j}{|x|} .
\]
We have $\frac{\partial Q_\d^s}{\partial x_j} \chi_{B(0, \e)^c} \in L^1 (\Rn)$ for all $\e>0$, and by Lemma \ref{lem: convdist pv dQ},
\[
 \frac{\partial Q_\d^s}{\partial x_j} \chi_{B(0, \e)^c} \to \frac{\partial Q_\d^s}{\partial x_j} \quad \text{in } \mathcal{S}' \quad \text{as } \e \to 0,
\]
so
\[
 \F \left( \frac{\partial Q_\d^s}{\partial x_j} \chi_{B(0, \e)^c} \right) \to \F \left( \frac{\partial Q_\d^s}{\partial x_j} \right) \quad \text{in } \mathcal{S}' \quad \text{as } \e \to 0 .
\]
We now compute
\begin{align*}
 \F \left( \frac{\partial Q_\d^s}{\partial x_j} \chi_{B(0, \e)^c} \right) (\xi) & = -\frac{(n-1+s)}{\gamma(1-s)} \int_{B(0, \e)^c} \frac{x_j}{|x|^{n+s+1}} w_\delta (x) e^{-2\pi i \xi \cdot x} \, dx \\
 & = \frac{(n-1+s)}{\gamma(1-s)}i \int_{B(0,\e)^c} \frac{x_j}{|x|^{n+s+1}} w_\delta(x) \sin{(2\pi  \xi \cdot x)} \, dx ,
\end{align*}
where we have used the odd symmetry.
Now, by dominated convergence,
\[
 \int_{B(0,\e)^c} \frac{x_j}{|x|^{n+s+1}} w_\delta(x) \sin{(2\pi  \xi \cdot x)} \, dx \to \int_{\Rn} \frac{x_j}{|x|^{n+s+1}} w_\delta(x) \sin{(2\pi  \xi \cdot x)} \, dx
\]
because of the same argument as in \eqref{eq:dominatedsin}.
This proves \eqref{eq: Fourier transform of derivative}.
\end{proof}

With this, we can conclude the existence of $V_{\d}^s$ as a distribution.

\begin{prop} \label{pr: inverse tempered distribution}
%	Let $0<s<1$ and $\d>0$.
There exists a tempered distribution $V_{\d}^s$ whose Fourier transform is given by 
	\begin{equation}\label{eq:hatV}
	\hat{V}_{\d}^s(\xi)= -\frac{i \xi}{2\pi |\xi|^2} \frac{1}{\hat{Q}_{\d}^s(\xi)}.
	\end{equation}
\end{prop}
\begin{proof}
Denote by $W^s_{\d}$ the right-hand side of \eqref{eq:hatV}, which is well defined for $\xi \in \Rn \setminus \{0\}$ since $\hat{Q}_\d^s$ is positive (Proposition \ref{prop: Fourier transform well defined}).
Next, we subtract from  $W_{\d}^s$ a multiple of the function (and distribution) $\frac{-i\xi}{|\xi|}\frac{1}{|2\pi \xi|}$ of Lemma \ref{lemma: Fourier transform vector Riesz potential}\,\emph{\ref{item:Fourier transform of x/|x|^n})}:
\begin{equation} \label{eq: difference of V in 0}
	W_{\d}^s(\xi)- \frac{-i\xi}{2\pi|\xi|^2}\frac{1}{\hat{Q}_\d^s(0)} = -\frac{i \xi}{2\pi |\xi|^2} \left( \frac{1}{\hat{Q}_{\d}^s (\xi)} - \frac{1}{\hat{Q}_{\d}^s (0)} \right) .
\end{equation}
This function is in $L^\infty(B(0,1)^c)$, as a difference of functions in $L^\infty(B(0,1)^c)$ (see Proposition \ref{Prop: properties of the Fourier transform of Q}).
Let us see that it is also in $L^{\infty}(B(0,1))$.
By the mean value theorem, there exists $c>0$ such that for all $\xi \in B (0, 1)$,
\[
 \left| \frac{1}{\hat{Q}_{\d}^s (\xi)} - \frac{1}{\hat{Q}_{\d}^s (0)} \right| \leq c |\xi| . 
\]
As a result,
\begin{equation} \label{eq:W2}
 \left| W_{\d}^s(\xi)- \frac{-i\xi}{2\pi|\xi|^2}\frac{1}{\hat{Q}_\d^s(0)} \right| \leq \frac{c}{2\pi} ,
\end{equation}
so the function in \eqref{eq: difference of V in 0} is in $L^{\infty} (B(0, 1))$, and, hence, in $L^{\infty} (\Rn)$.
In particular, this function is a tempered distribution, and, by Lemma \ref{lemma: Fourier transform vector Riesz potential}\,\emph{\ref{item:Fourier transform of x/|x|^n})}, so is $W_{\d}^s$.
As the Fourier transform is an isomorphism from $\mathcal{S}'$ into itself, there exists $V_\d^s \in \mathcal{S}'$ such that \eqref{eq:hatV} holds.
\end{proof}

\subsection{Existence of $V_\d^s$ as a function}

In this subsection we prove that the distribution $V_\d^s$ of Proposition \ref{pr: inverse tempered distribution} is actually a function. First we notice that $\hat{V}_{\d}^s$ does not belong to any space where we can conclude directly that its Fourier transform is a function. The main drawback comes from the fact that the tail of $\hat{V}_\d^s$ is not integrable enough.
So as to tackle this, we exploit the fact that, at infinity, $\hat{V}_\d^s$ behaves like a homogeneous function with a known Fourier transform (namely, like $\hat{I_s}$).
Thus, we adapt the proof of \cite[Prop.\ 2.4.8]{Grafakos08a} (homogeneous function) to the non-homogeneous function $\hat{V}_\d^s$.

We first need the following decay estimate for the derivatives of $\hat{V}_\d^s$.

\begin{lem}\label{le:decaypartialV}
For every $\a \in \N^n$ there exists $C_{\a}>0$ such that for any $|\xi| \geq 1$,
\[
 \left| \p^{\a} \hat{V}^s_{\d} (\xi) \right| \leq \frac{C_{\a}}{|\xi|^{s (|\a| +1)}} .
\]
\end{lem}
\begin{proof}
In this proof we use the letter $C$ with some subindices to denote a generic postive constant independent of $\xi$; the relevant dependence is included in the subindices.
The value of the constant may vary from line to line.

Express $\hat{V}^s_{\d} = \frac{-i}{2\pi} \, g \, f$ with
\[
 g(\xi) = \frac{\xi}{|\xi|}, \qquad f = f_1 \circ g_1 , \qquad f_1 (t)= t^{-1}, \qquad g_1 (\xi) = |\xi| \, \hat{Q}^s_{\d} (\xi) .
\]
By Leibniz' formula,
\[
 \p^{\a} (g f) = \sum_{\b \leq \a} \binom{\a}{\b} \p^{\b} g \, \p^{\a - \b} f .
\]
Let $\b \in \N^n$.
By induction, it is easy to see that $\p^{\b} g (\xi)$ can be expressed as
\[
 \frac{P (\xi)}{|\xi|^{2|\b|+1}}
\]
for some $\Rn$-valued polynomial $P$, all of which components are of degree $|\b|+1$.
Therefore,
\begin{equation}\label{eq:partialg}
 \left| \p^{\b} g (\xi) \right| \leq \frac{C_{\b}}{|\xi|^{|\b|}} , \qquad \xi \in \Rn \setminus \{ 0 \} .
\end{equation}
We apply Fa\`a di Bruno's formula for the higher-order derivatives of a composition, and obtain that
\[
 \p^{\g} f = \sum_{k=1}^{|\g|} f_1^{k)} \circ g_1 \, G_k  
\]
where $G_k$ is a linear combination of products of $k$ partial derivatives of $g_1$, the order of which adds up $|\g|$.

We estimate the partial derivatives of $g_1$.
We express $g_1 = h \, \hat{Q}^s_{\d}$ with $h(\xi) = |\xi|$.
Since $\nabla h = g$, we have, by \eqref{eq:partialg}, that
\begin{equation}\label{eq:partialh}
 \left| \p^{\b} h (\xi) \right| \leq \frac{C_{\b}}{|\xi|^{|\b|-1}} , \qquad \xi \in \Rn \setminus \{ 0 \} .
\end{equation}
Now we show that
\begin{equation}\label{eq:partialQ}
 | \p^{\b} \hat{Q}_\d^s (\xi)| \leq \frac{C_{\b} }{|\xi|^{|\b|}} , \qquad \xi \in \Rn \setminus \{0\} .
\end{equation}

From Definition \ref{de:q} and Lemma \ref{lem: kernel primitive}, we have that $Q_\d^s$ is an $L^1$ function of compact support, smooth outside the origin, and that in a ball $B$ centred at the origin, one has
\[
 Q^s_{\d}(x) = \l_0 + \frac{\l_1}{|x|^{n-1+s}} , \qquad x \in B \setminus \{0\} 
\]
for some $\l_0, \l_1 \in \R$.
With this expression it is easy to see that
\[
 \left| \partial^{\b} (x^{\b} Q_\d^s (x) ) \right| = \left| \sum_{\g \leq \b} \binom{\b}{\g} \p^{\g} \left( x^{\b} \right) \p^{\b - \g} Q_\d^s (x) \right| \leq \frac{C_{\b}}{|x|^{n-1+s}} , \qquad x \in B \setminus \{0\} 
\]
for some $C_{\b} >0$.
Moreover, since $\partial^{\b}( x^{\b} Q_\d^s)$ is smooth outside the origin and has compact support, we conclude  that it is in $L^1 (\Rn)$.
Consequently, $\mc{F} \left( \partial^{\b}(x^{\b} Q_\d^s) \right)$ is bounded.
But
\[
 \mc{F} \left( \partial^{\b}((-2\pi i x)^{\b} Q_\d^s) \right) = (2 \pi i \xi)^{\b} \mc{F} ((-2\pi i x)^{\b} Q_\d^s) = (2\pi i \xi)^{\b} \partial^{\b} \hat{Q}_\d^s(\xi) ,
\]
which shows \eqref{eq:partialQ}.

Now, by Leibniz' formula, \eqref{eq:partialh} and \eqref{eq:partialQ},
\[
 \left| \p^{\a} g_1 (\xi) \right| \leq C_{\a} \sum_{\b \leq \a} \left| \p^{\b} h (\xi )\right| \left| \p^{\a - \b} \hat{Q}^s_{\d} (\xi) \right| \leq \frac{C_{\a}}{|\xi|^{|\a|-1}} , \qquad \xi \in \Rn \setminus \{0\} ,
\]
for some constant $C_{\a} >0$.
Hence, if we multiply $k$ partial derivatives of $g_1$, the order of which adds up $|\g|$, we obtain that
\[
 \left| G_k (\xi) \right| \leq \frac{C_{\g,k}}{|\xi|^{|\g| -k}} , \qquad \xi \in \Rn \setminus \{0\} ,
\]
for some constants $C_{\g,k} >0$.
On the other hand, by induction,
\[
 \left| f_1^{k)} (t) \right| = \frac{C_k}{t^{k+1}} , \qquad k \in \N , \quad t >0 ,
\]
for some constants $C_k >0$, and, hence,
\[
 \left| f_1^{k)} \circ g_1 (\xi) \right| \leq \frac{C_k}{\left( |\xi| \, \hat{Q}^s_{\d} (\xi) \right)^{k+1}} , \qquad k \in \N  , \quad  \xi \in \Rn \setminus \{0\}  .
\]
From Proposition \ref{Prop: properties of the Fourier transform of Q} we know that, for $|\xi| \geq 1$,
\[
 \left| \frac{1}{\hat{Q}^s_{\d} (\xi)} \right| \leq C \left| \xi \right|^{1-s} ,
\]
so
\[
 \frac{1}{\left( |\xi| \, \hat{Q}^s_{\d} (\xi) \right)^{k+1}} \leq \frac{C}{|\xi|^{s (k+1)}} .
\]
Thus,
\[
 \left| \p^{\g} f (\xi) \right| \leq \sum_{k=1}^{|\g|} \left| f_1^{k)} \circ g_1 (\xi) \right| \left| G_k (\xi) \right| \leq C_{\g} \sum_{k=1}^{|\g|} \frac{1}{|\xi|^{s (k+1) + |\g| -k}} \leq\frac{C_{\g}}{|\xi|^{s (|\g|+1)}} .
\]
We conclude that, for $|\xi| \geq 1$,
\[
 \left| \p^{\a} \hat{V}^s_{\d} (\xi) \right| \leq C_{\a} \sum_{\b \leq \a} \left| \p^{\b} g (\xi) \right| \left| \p^{\a - \b} f (\xi) \right| \leq C_{\a} \sum_{\b \leq \a} \frac{1}{|\xi|^{s (|\a| +1) + |\b| (1-s)}} \leq \frac{C_{\a}}{|\xi|^{s (|\a| +1)}} ,
\]
as desired.
\end{proof}

The decay estimate of Lemma \ref{le:decaypartialV} is not optimal.
In fact, a more refined argument can possibly improve decay \eqref{eq:partialQ} and show that the bound 
\[
 | \p^{\b} \hat{Q}^s_{\d} (\xi)| \leq \frac{C_{\a}}{|\xi|^{|\b|+1-s}} , \qquad \xi \in \Rn \setminus \{ 0 \}
\]
holds.
With that estimate, an adaptation of the proof of Lemma \ref{le:decaypartialV} would yield 
\[
 \left| \partial^{\alpha} \hat{V}_\d^s (\xi) \right| \leq \frac{C_{\alpha}}{|\xi|^{|\alpha| +s}} , \qquad |\xi| \geq 1 .
\]
Nevertheless, the bound of Lemma \ref{le:decaypartialV} is enough for our purposes in Theorem \ref{Th: inverse Fourier trasnform as a function}.
Before that, we need the following inverse Lipschitz estimate of the function $\frac{x}{|x|^{n-s+1}}$.

\begin{lem}\label{le:inverseLipchitz}
%Let $0<s<1$.
For every $R_1, R_2>0$ there exists $m>0$ such that for all $x \in B(0,R_1) \setminus \{ 0 \}$ and $h \in B(0,R_2) \setminus \{ x \}$,
\begin{equation}\label{eq:ineqm}
 m |h| \leq \left| \frac{x}{|x|^{n+1-s}} - \frac{x-h}{|x-h|^{n+1-s}} \right| .
\end{equation}
\end{lem} 
\begin{proof}
We divide the proof into four cases, according to the position of the points $x$ and $h$.
Let us define $G(x) = \frac{x}{|x|^{n+1-s}}$.

\smallskip

\emph{Case 1:} $2 |x| \leq |x-h|$.
Taking
\[
 m \leq \frac{1 - \frac{1}{2^{n-s}}}{R_1^{n-s} R_2}
\]
we have
\[
 \left| G(x) - G(x-h) \right| \geq \frac{1}{|x|^{n-s}} - \frac{1}{|x-h|^{n-s}}  \geq \left( 1 - \frac{1}{2^{n-s}} \right) \frac{1}{|x|^{n-s}} \geq \frac{1 - \frac{1}{2^{n-s}}}{R_1^{n-s}} \geq m R_2 \geq m |h| .
\]

\smallskip

\emph{Case 2:} $G(x) \cdot G(x-h) \leq 0$.
Taking
\[
 m \leq \frac{1}{R_1^{n-s} R_2}
\]
we have
\begin{align*}
 & \left| G(x) - G(x-h) \right| = \left( \left| G(x) \right|^2 + \left| G(x-h) \right|^2 - 2 G(x) \cdot G(x-h) \right)^{\frac{1}{2}}\geq \left| G(x) \right|\\
  & = \frac{1}{|x|^{n-s}} \geq  \frac{1}{R_1^{n-s}} \geq m R_2 \geq m |h| .
\end{align*}

\smallskip

\emph{Case 3:} $|x-h| \leq 2 |x|$ and
\begin{equation}\label{eq:case3}
 \min \left\{ \left| G(x) \right|^2 , \left| G(x-h) \right|^2 \right\} \leq G(x) \cdot G(x-h) .
\end{equation}
We observe that the inverse of $G$ is $G^{-1} (y) = \frac{y}{|y|^{\frac{n+1-s}{n-s}}}$, with derivative
\[
 D G^{-1} (y) = |y|^{- \frac{n-s+1}{n-s}} I - \frac{n+1-s}{n-s} |y|^{- \frac{n-s+1}{n-s} -2} y \otimes y  ,
\]
where $\otimes$ denotes the tensor product, so
\begin{equation}\label{eq:DG-1}
 | D G^{-1} (y) | \leq d_{n,s} \left| y \right|^{- \frac{n-s+1}{n-s}}
\end{equation}
for some constant $d_{n,s} >0$.
By the mean value theorem,
\begin{equation}\label{eq:hestimate}
 |h| = |G^{-1} (G (x)) - G^{-1} (G (x-h)) | \leq \left\| D G^{-1} \right\|_{L^{\infty} ([G(x), G (x-h)])} \left| G(x) - G(x-h) \right| .
\end{equation}
Now, using \eqref{eq:DG-1},
\[
 \| D G^{-1} \|_{L^{\infty} ([G(x), G (x-h)])} \leq d_{n,s} \max_{y \in [G(x), G (x-h)]} |y|^{- \frac{n-s+1}{n-s}} = d_{n,s} \left( \min_{y \in [G(x), G (x-h)]} |y| \right) ^{- \frac{n-s+1}{n-s}} .
\]

Elementary geometry (see Figure \ref{fi:GO}) shows that
\begin{equation}\label{eq:puntorecta}
 \min_{y \in [G(x), G (x-h)]} |y| = \begin{cases}
 \left| G(x-h) \right| & \text{if } \left( G(x-h) - G(x) \right) \cdot G (x-h) \leq 0 , \\
 \left| G(x) \right| & \text{if } \left( G(x-h) - G(x) \right) \cdot G (x) \geq 0 .
 \end{cases}
\end{equation}
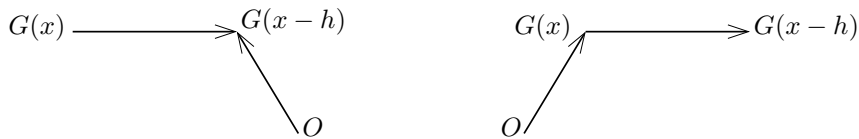
\begin{figure}[hb]
\begin{center}
\begin{overpic}[width=0.5\textwidth,tics=10]{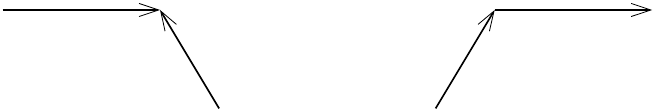}
\put (-9, 15) {\footnotesize $G(x)$}
\put (25, 16) {\footnotesize $G(x-h)$}
\put (34, 0) {\footnotesize $O$}
\put (65, 15) {\footnotesize $G(x)$}
\put (100, 15) {\footnotesize $G(x-h)$}
\put (63, 0) {\footnotesize $O$}
\end{overpic}
\small \caption{\label{fi:GO}
Position of the points $G(x)$, $G(x-h)$ and origin $O$ when $\left( G(x-h) - G(x) \right) \cdot G (x-h) \leq 0$ (left) and when $\left( G(x-h) - G(x) \right) \cdot G (x) \geq 0$ (right).
}
\end{center}
\end{figure}

Assumption \eqref{eq:case3} asserts that one of the two options of \eqref{eq:puntorecta} occurs, so
\[
  \min_{y \in [G(x), G (x-h)]} |y| \geq \min \left\{ \left| G(x-h) \right|, \left| G(x) \right| \right\}
\]
and, hence,
\[
  \left( \min_{y \in [G(x), G (x-h)]} |y| \right) ^{- \frac{n-s+1}{n-s}} \leq \left(  \min \left\{ \left| G(x-h) \right|, \left| G(x) \right| \right\} \right) ^{- \frac{n-s+1}{n-s}} = \max \{ |x|^{n-s+1} , |x-h|^{n-s+1} \} .
\]
Finally, since $|x-h| \leq 2 |x|$,
\begin{equation}\label{eq:maxx}
 \max \{ |x|^{n-s+1} , |x-h|^{n-s+1} \} \leq 2^{n-s+1} |x|^{n-s+1} \leq 2^{n-s+1} R_1^{n-s+1} .
\end{equation}
Going back to \eqref{eq:hestimate}, we find that  $|h| \leq 2^{n-s+1} d_{n,s} R_1^{n-s+1} | G(x) - G(x-h) |$, so inequality \eqref{eq:ineqm} holds for
\[
 m \leq \frac{1}{2^{n-s+1} d_{n,s} R_1^{n-s+1}} .
\]

\smallskip

\emph{Case 4:} $|x-h| \leq 2 |x|$ and
\begin{equation}\label{eq:case4}
 0 < G(x) \cdot G(x-h) < \min \left\{ \left| G(x) \right|^2 , \left| G(x-h) \right|^2 \right\} .
\end{equation}
Note first that inequality \eqref{eq:case4} cannot occur in dimension $n=1$.

Let $\g : [0,1] \to \Rn$ be any piecewise $C^1$ curve such that $\g(0) = G(x)$ and $\g(1) = G(x-h)$.
By the fundamental theorem of Calculus,
\begin{equation}\label{eq:curveg}
 |h| = \left| G^{-1} (\g(0)) - G^{-1} (\g(1)) \right| = \left| \int_0^1 \left( G^{-1} \circ \g \right)' (t) \, dt \right| \leq \max_{\g ([0,1])} \left| D G^{-1} \right| \ell (\g) ,
\end{equation}
where $\ell$ denotes the length of the curve.

Assumption \eqref{eq:case4} implies that none of the cases of \eqref{eq:puntorecta} occurs (hence none of the situations depicted in Figure \ref{fi:GO}), but the distance from the origin to the segment $[G(x) , G(x-h)]$ is attained at a point $P$ in the interior of the segment.
Assume that $|G(x) - P| \leq |G(x-h) - P|$, although the construction is totally analogous in the symmetric case $|G(x) - P| \geq |G(x-h) - P|$.
Let $Q$ be the point in the segment $[G(x) , G(x-h)]$ such that $P$ is the middle point between $G(x)$ and $Q$.
We define the curve $\g$ as follows.
The curve $\g$ starts at $G(x)$ and describes the arc of circumference of center the origin $O$ and radius $|G(x)|$ joining $G(x)$ with $Q$; among the two possible arcs, we choose that which subtends an angle of less than $\pi$ radians.
Then, $\g$ continues joining $Q$ and $G(x-h)$ with a straight line.
See Figure \ref{fi:gamma}.

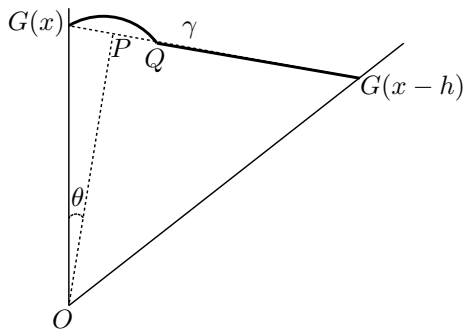
\begin{figure}[hb]
\begin{center}
\begin{overpic}[width=0.25\textwidth,tics=10]{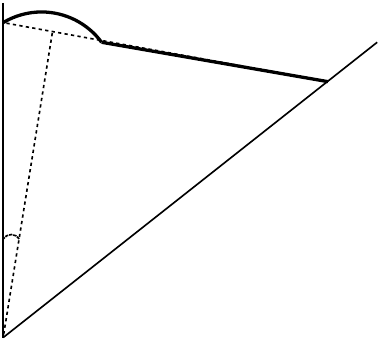}
\put (-4, -6) {\footnotesize $O$}
\put (1.5, 30) {\footnotesize $\theta$}
\put (-17, 82) {\footnotesize $G(x)$}
\put (13, 74) {\footnotesize $P$}
\put (23, 71) {\footnotesize $Q$}
\put (34, 81) {\footnotesize $\gamma$}
\put (86, 63) {\footnotesize $G(x-h)$}
\end{overpic}
\small \caption{The curve $\g$ (in thick line), the points $G(x), P, Q, G(x-h)$ (aligned, in dotted line), the origin $O$ and the angle $\theta$.\label{fi:gamma}}
\end{center}
\end{figure}
For this particular $\g$ we estimate the right hand-side of \eqref{eq:curveg}.
First, using \eqref{eq:DG-1},
\begin{equation}\label{eq:maxDg-1}
 \max_{\g ([0,1])} \left| D G^{-1} \right| \leq d_{n,s} \max_{y \in \g ([0,1])} |y|^{- \frac{n-s+1}{n-s}} = d_{n,s} \left| G(x) \right|^{- \frac{n-s+1}{n-s}} ,
\end{equation}
since, by construction of $\g$, the shortest distance of $\g([0,1])$ to the origin is $|G(x)|$.
In order to estimate $\ell (\g)$, let $\theta$ be the angle $\widehat{G(x) \, O \, P}$ if it is positive, or else the opposite angle $\widehat{P \, O \, G(x)}$, so that 
\[
 \sin \theta = \frac{\ell ([G(x),P])}{|G(x)|}
\]
and $\theta \in [0, \frac{\pi}{2}]$ because $0 \leq G(x) \cdot G(x-h)$.
Then
\[
 \ell (\g) = 2 \theta \left| G(x) \right| + \ell ([ Q, G(x-h) ]) .
\]
Now we use the elementary inequality
\[
 t \leq \frac{\pi}{2} \sin t , \qquad t \in [0, \frac{\pi}{2}]
\]
to obtain that
\[
  2 \theta \left| G(x) \right| \leq \pi \sin \theta \left| G(x) \right| = \pi \, \ell ([G(x),P]) = \frac{\pi}{2} \, \ell ([G(x),Q]) ,
\]
so
\begin{equation}\label{eq:ellg}
 \ell (\g) \leq \frac{\pi}{2} \, \ell ([G(x),Q]) +  \ell ([ Q, G(x-h) ]) \leq \frac{\pi}{2} \, \ell ([ G(x), G(x-h) ]) .
\end{equation}
Using \eqref{eq:maxDg-1} and \eqref{eq:ellg}, inequality \eqref{eq:curveg} becomes
\[
 |h| \leq \frac{\pi}{2} d_{n,s} \left| G(x) \right|^{- \frac{n-s+1}{n-s}} \, \ell ([ G(x), G(x-h) ]) .
\]
If we had assumed $|G(x) - P| \geq |G(x-h) - P|$ instead of $|G(x) - P| \leq |G(x-h) - P|$ we would have obtained
\[
 |h| \leq \frac{\pi}{2} d_{n,s} \left| G(x-h) \right|^{- \frac{n-s+1}{n-s}} \, \ell ([ G(x), G(x-h) ]) ,
\]
so, in either case,
\begin{align*}
 |h| & \leq \frac{\pi}{2} d_{n,s} \max \left\{ |G(x)|^{- \frac{n-s+1}{n-s}} , |G(x-h)|^{- \frac{n-s+1}{n-s}} \right\} \ell ([ G(x), G(x-h) ]) \\
 & = \frac{\pi}{2} d_{n,s} \max \left\{ |x|^{n-s+1} , |x-h|^{n-s+1} \right\} \left|  G(x) - G(x-h) \right| .
\end{align*}
Now we use \eqref{eq:maxx} and find that inequality \eqref{eq:ineqm} holds for
\[
 m \leq \frac{1}{2^{n-s} \pi \, d_{n,s} R_1^{n-s+1}} .
\]
%
%\emph{Case 5:} $\left| x \cdot (x-h) \right| \leq \frac{1}{2} |x| |x-h|$.
%We take
%\[
% m \leq \frac{\sqrt{3}}{2 R_1^{n-s} R_2} .
%\]
%Clearly, $|G(x) - G(x-h)|$ is bounded below by the distance of the point $G(x)$ to the line passing through $0$ and $G(x-h)$, so
%\begin{align*}
% & |G(x) - G(x-h)|^2 \geq \left| G(x) - G(x) \cdot \frac{x-h}{|x-h|} \frac{x-h}{|x-h|} \right|^2 = \left| G(x) \right|^2 - \left( G(x) \cdot \frac{x-h}{|x-h|} \right)^2 \\
% & \geq \frac{1}{|x|^{2 (n-s)}} - \frac{1}{4 |x|^{2(n-s)}} = \frac{3}{4 |x|^{2 (n-s)}} \geq \frac{3}{4 R_1^{2 (n-s)}} \geq m^2 R_2^2 \geq m^2 |h|^2 .
%\end{align*}

\end{proof}

Finally, we present the main result of this section: its statement includes that of Proposition \ref{pr:inversekernelbrief}, shows that $V_{\d}^s$ is actually a function and exhibits its main properties.

\begin{teo} \label{Th: inverse Fourier trasnform as a function} 
%	Let $0<s<1$ and $\d>0$.
There exists a vector radial function $V_{\d}^s \in C^{\infty}(\Rn \setminus \{ 0\}, \Rn)$ such that
	\begin{equation} \label{eq: Fourier transform of V 2}
		\hat{V}_{\d}^s(\xi)=- i \frac{\xi}{|\xi|} \frac{1}{\left| 2\pi \xi \right| \hat{Q}_{\d}^s(\xi)}.
	\end{equation}
Furthermore, the following properties hold:
\begin{enumerate}[a)]
\item\label{item:Vunique} $V_\d^s$ is the only $L^1_{\loc}$ function that satisfies
\begin{equation}\label{eq:VQconvolution}
 \int_{\Rn}V_\d^s(z) \, Q_\d^s(y-z) \, dz = \frac{1}{\sigma_{n-1}}\frac{y}{|y|^n}, \qquad y \in \Rn \setminus \{ 0 \} .
\end{equation}

\item\label{item:Va} There exists a Lipschitz bounded $W: \Rn \to \Rn$ (actually, $W \in C_0 (\Rn, \Rn)$ when $n\geq 2
$) such that
\begin{equation}\label{eq:VWR}
 V_{\d}^s (x) = W(x) + \frac{c_{n,-s}}{a_0} \frac{x}{|x|^{n+1-s}} .
\end{equation}

%\item\label{item:Vb} For each $x \in \Rn \setminus \{0\}$, \label{property of V b)}
%\[
% \lim_{\lambda \to 0^+} \lambda^{n-s} V_{\d}^s(\lambda x) = \frac{c_{n,-s}}{a_0} \frac{x}{|x|^{n+1-s}} .
%\]
%\textcolor{blue}{Yo quitar\'ia este apartado: no dice nada que no diga el apartado anterior y no lo usamos expl\'icitamente}

\item\label{item:Vc} For any $R >0$ there exists $M > 0$ such that for all $x \in B(0,R) \setminus \{ 0 \}$,
\[
 \left| V_\d^s(x)\right| \leq \frac{M}{|x|^{n-s}} .
\]
\item\label{item:Vd} For any $R_1, R_2 >0$ there exists $M > 0$ such that for all $x \in B(0,R_1) \setminus \{ 0 \}$ and $h \in B(0,R_2) \setminus \{ x \}$, \label{property of V d)}
\[
 \left|V_\d^s(x)-V_\d^s(x-h)\right| \leq M \left|\frac{x}{|x|^{n+1-s}}-\frac{x-h}{|x-h|^{n+1-s}} \right|.
\]

\end{enumerate}
\end{teo}
\begin{proof}
We first prove that there exists $V_{\d}^s \in C^{\infty}(\Rn \setminus \{ 0\}, \R^n)$ such that \eqref{eq: Fourier transform of V 2} holds.

We start as in the proof of \cite[Prop.\ 2.4.8]{Grafakos08a}. In order to see that $V_\d^s$ is $C^{\infty}$ away from the origin we note that $\mathcal{F}(\hat{V}_\d^s) = \tilde{V}_\d^s$ and shall see that $\mathcal{F}(\hat{V}_\d^s)$ is $C^M$ in $\Rn \setminus \{0\}$ for all $M$. Thus, fix $M \in \mathbb{N}$ and let $\alpha \in\N^n$ be any multiindex such that
\begin{equation}\label{eq:choicea}
 s (|\a| +1 ) - n \geq M .
\end{equation}
We take $\varphi \in C^{\infty}(\Rn)$ such that $\varphi = 1$ in $B(0,2)^c$ and $\varphi= 0$ in $B(0,1)$.
Write
\[
 u=\hat{V}_\d^s , \quad \quad u_0=(1-\varphi)u \quad \text{and} \quad u_\infty=\varphi \, u .
\]
On the one hand, $\partial^{\alpha} u= \partial^{\alpha} u_0 + \partial^{\alpha} u_\infty$ in the sense of distributions and also in $\Rn \setminus \{ 0 \}$.
On the other hand, as $u$ is smooth outside the origin, we have that $\partial^{\alpha} u_\infty$ is smooth and can calculate
\begin{equation*}
	\partial^{\alpha} u_\infty = \sum_{\b \leq \a} \binom{\a}{\b} \p^{\a - \b} \f \, \p^{\b} u .
\end{equation*}
Write
\begin{equation*}%\label{eq: Th V eq 1}
	 v = \partial^{\alpha} u_0 + \sum_{\substack{\b \leq \a \\ \b \neq \a}} \binom{\a}{\b} \p^{\a - \b} \f \, \p^{\b} u .
\end{equation*}
Then $v$ is a distribution with support in $B(0, 2)$, so $\hat{v}$ is $C^{\infty}$.
Moreover, $\p^{\a} u = v + \f \, \p^{\a} u$.
Thus, in order to see that $\widehat{\partial^{\alpha} u}$ is $C^M$, it remains to show that $\widehat{\varphi \, \partial^{\alpha} u}$ is $C^M$.
The function $\varphi \, \partial^{\alpha} u$ is $C^{\infty}$ and, by Lemma \ref{le:decaypartialV},
\begin{equation*}% \label{eq: Th V eq 2}
	 \left| \varphi (\xi) \partial^{\alpha} u (\xi) \right| \leq \frac{C_{\a}}{1+|\xi|^{s (|\a|+1)}} , \qquad \xi \in \Rn ,
\end{equation*}
Having in mind \eqref{eq:choicea}, a classical result relating the decay of a function at infinity with the regularity of its Fourier transform (see, e.g., \cite[Exercise 2.4.1]{Grafakos08a}) shows that $\widehat{\varphi \,\partial^{\alpha} u}$ is $C^M$.

Once we have shown that $\widehat{\partial^{\alpha} u}$ is $C^M$, we note that $\widehat{\partial^{\alpha} u} (\xi) = (2 \pi i \xi)^{\alpha} \hat{u} (\xi)$.
Let $\xi \in \Rn \setminus \{ 0 \}$; then $\xi_j \neq 0$ for some $j \in \{ 1, \ldots, M \}$.
Let $V$ be a neighbourhood of $\xi$ such that every $\eta \in V$ satisfies $\eta_j \neq 0$.
Let $m \in \N$ be such that $s (m +1 ) - n \geq M$ and let $\alpha$ be the multiindex $(0, \ldots, 0, m, 0, \ldots, 0)$, with the component $m$ in position $j$.
Then $\alpha$ satisfies \eqref{eq:choicea}.
Moreover, for any $\eta \in V$,
\[
 \hat{u} (\eta) = \frac{\widehat{\partial^{\alpha} u} (\eta)}{(2 \pi i \eta_j)^m} ,
\]
so $\hat{u}$ is of class $C^M$ in $\Rn \setminus \{0\}$ for every $M\in \N$, and therefore, so is $V_{\d}^s$.

Once we have that $V_{\d}^s$ is a function, since $\hat{V}_{\d}^s$ is radial and imaginary-valued, standard properties of the Fourier transform show that $V_{\d}^s$ must be radial and real-valued.

Next, we show that the function
\[
 Z(\xi) := \hat{V}_\d^s(\xi)- \frac{-i\xi}{a_0|\xi|}\frac{1}{|2\pi\xi|^s}
\]
decays to zero at infinity faster than any negative power of $|\xi|$.
For this we observe that
	\begin{equation} \label{eq: difference of V hat and vector Riesz potential}
		 	Z (\xi)  = -i\frac{\xi}{|\xi|} \frac{1}{|2\pi \xi|\hat{Q}_\d^s(\xi)}-\frac{-i \xi}{a_0|\xi|} \frac{1}{|2\pi \xi|^s} 
		 =-i \frac{\xi}{|\xi|}\frac{a_0|2\pi \xi|^{-1+s} - \hat{Q}_\d^s(\xi)}{a_0|2\pi \xi|^s\hat{Q}_\d^s(\xi)}.
	\end{equation}
The terms $|2\pi \xi|^s$ and $\hat{Q}_\d^s(\xi)$ in the denominator above only contribute as a power of $|\xi|$ in the growth at infinity (see Proposition \ref{Prop: properties of the Fourier transform of Q}).
Therefore, it remains to show that the numerator above $a_0|2\pi \xi|^{-1+s} - \hat{Q}_\d^s(\xi)$ decays faster at infinity  than any negative power of $|\xi|$.
Consider a $\f \in C_c^{\infty}(\Rn)$ with $\f_{B(0, \frac{1}{4})} = 1$ and $\f_{B(0, \frac{1}{2})^c} = 0$.
Then, recalling \eqref{eq: Riesz Potential},
	\begin{align*}
 a_0|2\pi \xi|^{-1+s} - \hat{Q}_\d^s(\xi) & = \mathcal{F}\left(\frac{a_0}{\gamma(1-s)|x|^{n-1+s}} - Q_\d^s(x) \right)  \\
 & = \mathcal{F}\left(\frac{a_0 \f}{\gamma(1-s)|x|^{n-1+s}} - Q_\d^s(x) \right) +\mathcal{F}\left(\frac{a_0(1-\f)}{\gamma(1-s)|x|^{n-1+s}} \right) .
	\end{align*}
Looking at the expression of $Q_\d^s$ (Definition \ref{de:q} and Lemma \ref{lem: kernel primitive}), we notice that the difference between  $\frac{a_0 \f}{\gamma(1-s)|x|^{n-1+s}}$ and $Q_\d^s(x)$ coincide with the constant $\frac{- z_0}{\gamma(1-s)}$ in $B(0, \min\{b_0 \d, \frac{1}{4} \})$, and both have compact support.
Therefore, its difference is a smooth function of compact support.
In particular, it is in the Schwartz space, as well as its Fourier transform:
\[
 \mathcal{F}\left(\frac{a_0 \f}{\gamma(1-s)|x|^{n-1+s}} - Q_\d^s(x) \right) \in \mathcal{S} .
\]
On the other hand, the function $\mathcal{F}\left( \frac{1-\f}{\gamma(1-s)|x|^{n-1+s}} \right)$ is treated in \cite[Example 2.4.9]{Grafakos08a}, and it is shown that its decay at infinity is faster than any negative power of $|\xi|$.

From expression \eqref{eq: difference of V hat and vector Riesz potential} we can see that $Z$ is in $L^1_{\loc}$ when $n \geq 2$.
Because of its decay at infinity, $Z$ is in $L^1$ when $n\geq 2$, so it has a Fourier transform $\hat{Z}$, which is $C_0$.
When $n=1$, in Lemma \ref{lem: convdist pv Z} it will be shown that $Z$ is a tempered distribution, so it has a Fourier transform $\hat{Z}$, which, in principle, is a tempered distribution.
But since both $\hat{V}_\d^s(\xi)$ and $\frac{i\xi}{|\xi|}\frac{1}{|2\pi\xi|^s}$ are Fourier transforms of functions (see Lemma \ref{lemma: Fourier transform vector Riesz potential}\,\emph{\ref{item:FourierRiesz1})} for the latter), we conclude that $\hat{Z}$ is a function.

We continue by proving that $\hat{Z}$ is Lipschitz.
We have $-\nabla \hat{Z} = \mathcal{F} (2\pi i \xi Z(\xi) )$.
The function $2\pi i \xi Z(\xi)$ is in $L^1_{\loc}$, as can be seen from expression \eqref{eq: difference of V hat and vector Riesz potential}.
Due to the decay of $Z$ at infinity, $2\pi i \xi Z(\xi)$ is in $L^1 (\Rn)$, so $\mathcal{F} (2\pi i \xi Z(\xi) )$ is bounded, and, hence, $\hat{Z}$ is Lipschitz.

We define $W$ as $W(x) = \hat{Z} (-x)$.
Taking inverse Fourier transforms to the expression
\[
 \hat{V}_\d^s(\xi) = Z(\xi) + \frac{-i\xi}{a_0|\xi|}\frac{1}{|2\pi\xi|^s}
\]
we obtain equality \eqref{eq:VWR} (see Lemma \ref{lemma: Fourier transform vector Riesz potential} for the inverse Fourier transform of the last term).
That expression, together with the fact that $W$ is continuous, shows that $\hat{V}_\d^s$ is in $L^1_{\loc}$.

In order to show \emph{\ref{item:Vunique})}, we note that equality \eqref{eq:VQconvolution} is equivalent to the equality of its Fourier transforms.
More precisely, the functions $V_\d^s$ and $Q_\d^s$ can also be seen as tempered distributions, and, in particular, $Q_\d^s$ with compact support. Hence, by Lemmas \ref{lem: Fourier transform of convolution of distributions} and \ref{lemma: Fourier transform vector Riesz potential}\,\emph{\ref{item:Fourier transform of x/|x|^n})} we have that equality \eqref{eq:VQconvolution} is equivalent to
\begin{equation}\label{eq:VQFourier}
	\hat{V}_\d^s(\xi) \, \hat{Q}_\d^s(\xi)= -i \frac{\xi}{|\xi|}\frac{1}{|2\pi \xi|} ,
\end{equation}
which holds due to \eqref{eq: Fourier transform of V 2}.
The uniqueness of $\hat{V}_\d^s$ also follows from this argument, since $\hat{V}_\d^s$ is uniquely determined by equality \eqref{eq:VQFourier}.
Thus, \emph{\ref{item:Vunique})} is proved.

Fact \emph{\ref{item:Va})} has been proved when $n \geq 2$, while for the case $n=1$ it only remains to show that $W$ is bounded: this is tackled in Appendix \ref{se:1D}.
%With this we have that for each $x \in \Rn \setminus \{0\}$ and $\lambda > 0$,
%\[
% \lambda^{n-s} V_{\d}^s(\lambda x) = \lambda^{n-s} W(\lambda x) + \frac{c_{n,-s}}{a_0} \frac{x}{|x|^{n+1-s}} ,
%\]
%so we readily obtain \emph{\ref{item:Vb})}.
%
With this, we have that given $R>0$, for all $x \in B (0, R) \setminus \{ 0 \}$,
\[
 \left| V_{\d}^s (x) \right| \leq \left\| W \right\|_{L^{\infty} (B(0, R))} + \frac{|c_{n,-s}|}{a_0} \frac{1}{|x|^{n-s}} \leq \left(  \left\| W \right\|_{L^{\infty} (\Rn)} R^{n-s} + \frac{|c_{n,-s}|}{a_0} \right) \frac{1}{|x|^{n-s}} ,
\]
which shows \emph{\ref{item:Vc})}.

As for inequality \emph{\ref{item:Vd})}, since $W$ is Lipschitz, we estimate
\begin{align*}
 \left| V_{\d}^s (x) - V_{\d}^s (x-h) \right| & \leq \| D W \|_{L^{\infty}(\Rn)} |h| + \frac{|c_{n,-s}|}{a_0} \left| \frac{x}{|x|^{n+1-s}} - \frac{x-h}{|x-h|^{n+1-s}} \right| \\
  & \leq M \left| \frac{x}{|x|^{n+1-s}} - \frac{x-h}{|x-h|^{n+1-s}} \right| ,
\end{align*}
for a suitable constant $M>0$ coming from Lemma \ref{le:inverseLipchitz}. The proof is complete.
\end{proof}

Part \emph{\ref{item:Va})} of Theorem \ref{Th: inverse Fourier trasnform as a function} shows that $V^s_{\d}$ behaves like $\frac{x}{|x|^{n+1-s}}$ around $0$.
It can also be seen that it behaves like $\frac{x}{|x|^n}$ at infinity.
Comparing these facts with the classical and fractional fundamental theorem of Calculus (see the Introduction), we have the following picture of $V^s_{\d}$: at $0$, it behaves like the kernel of the fractional fundamental theorem of Calculus, while, at infinity, like that of the classical fundamental theorem of Calculus.

\section{Poincar\'e, Morrey, Trudinger and Hardy inequalities} \label{se: Poincare}

In this section we will use the nonlocal fundamental theorem of Calculus (Theorem \ref{Theo: nonlocal fundamental theorem of calculus}) to prove inequalities in the spirit of Poincar\'e--Sobolev, Morrey, Trudinger and Hardy.
In those inequalities we will need a boundary condition implying that the function vanishes in a tubular neighbourhood of $\p \O$.

In order to describe more precisely that boundary condition, we recall the set $\O_{-\d} = \{ x \in \O : \dist (x, \p \O) > \d \}$ and define the subspace $H_0^{s,p,\d}(\O_{-\d})$ as the closure of $C_c^{\infty}(\O_{-\d})$ in $H^{s,p,\d}(\O)$:
\[
H_0^{s,p,\d}(\O_{-\d})=\overline{C_c^\infty(\O_{-\d})}^{H^{s,p,\d}(\O)}.
\]
It is immediate to check that any $u \in H_0^{s,p,\d}(\O_{-\d})$ satisfies $u=0$ a.e.\ in $\O_{\d} \setminus \O_{-\d}$ and $D^s_{\d} u = 0$ a.e.\ in $\O_{B,\d}$.
%At the moment we consider that the set of functions we are going to use for the following embeddings vanish on a tubular neighbourhood of radius $2\d$. In the future we hope this can be improved to just requiring a tubular neighbourhood of radius $\d$ through the use of an extension theorem in this framework.
%In fact, we will use several times the observation that $\supp D_\d^s u \subset \supp u + B(0,\d) \subset \O$ for $u \in H_0^{s,p,\d}(\O_{-\d})$.
We leave for a future work the issue of whether $H_0^{s,p,\d}(\O_{-\d})$ actually coincides with the set of $u \in H^{s,p,\d}(\O_{-\d})$ such that $u=0$ a.e.\ in $\O_{\d} \setminus \O_{-\d}$, so that $H_0^{s,p,\d}(\O_{-\d})$ can be regarded as a volumetric-type condition.
Finally, given $g \in H^{s,p,\d}(\O)$ we define the affine subspace $H^{s,p,\d}_g (\O_{-\d})$ as $g+ H^{s,p,\d}_0(\O_{-\d})$. 

As in the space $H^{s,p}(\Rn)$ (and $W^{s,p}$, too), the Sobolev conjugate exponent of a $p \in [1, \frac{n}{s})$ is
\begin{equation}\label{eq:p*s}
 p_s^*=\frac{np}{n-sp} .
\end{equation}
The Poincar\'e--Sobolev inequality in $H_0^{s,p,\d}(\O_{-\d})$ is as follows.
Its analogue in the fractional case can be found in \cite[Th.\ 1.8]{ShS2015}.
As theirs, our proof is based on (our version of) the nonlocal fundamental theorem of Calculus and the Hardy--Littlewood--Sobolev inequality, but we also take advantage of the comparison between the kernel $V^s_{\d}$ and the Riesz potential given by Theorem \ref{Th: inverse Fourier trasnform as a function}\,\emph{\ref{item:Vc})}.

\begin{teo}\label{th:Poincare H0delta}
	Let $1 < p < \infty$ be with $sp<n$.
	Then, there exists $C>0$ such that for all 
	$u \in H_0^{s,p,\d}(\O_{-\d})$, 
	\[
	\left\| u \right\|_{L^q (\O)} \leq C \left\| D_\d^s u \right\|_{L^p(\O)} 
	\]
	for every $q \in [1,p_s^*]$.
\end{teo}
\begin{proof}
	By density, it is enough to prove the inequality for $u \in C^{\infty}_c (\O_{-\d})$.
	
	Fix $x \in \O$ and let $C>0$ denote a constant whose value may vary through this process. Notice that $\supp D_\d^s u \subset \O$ and, by Proposition \ref{Lemma: convolución con gradiente clásico}, $D_\d^s u \in C^\infty(\Rn)$. 
	By Theorem \ref{Theo: nonlocal fundamental theorem of calculus} and Proposition \ref{pr:inversekernelbrief}, % and \eqref{eq:growthV},
	\begin{equation}\label{eq:upoint}
	\left| u (x) \right| \leq  \int_{\O} |D_\d^s u(y)| |V_\d^s(x-y)| \, dy \leq C \int_{\O} \frac{|D_\d^s u(y)|}{|x-y|^{n-s}} \, dy= C \left(I_{s}* |D_\d^s u|\right)(x) ,
	\end{equation}
where $I_s$ is the Riesz potential \eqref{eq: Riesz Potential}.
On the other hand, by the Hardy--Littlewood--Sobolev inequality (e.g., \cite[Ch.\ 4, Th.\ 2.1]{Mizuta}) we have that %for every $q\leq \frac{np}{n-sp}$
	\[
	\| I_{s}* |D_\d^s u| \|_{L^{p_s^*}(\Rn)} \leq C \| D_\d^s u\|_{L^p(\Rn)}.
	\]
	Therefore, for every $q\in [1, p_s^*]$,
	\[
	\left\| u \right\|_{L^q (\O)} \leq C \left\| u \right\|_{L^{p_s^*} (\O)} \leq C \left\| I_{s}* \left| D_\d^s u \right| \right\|_{L^{p_s^*}(\Rn)} \leq C \left\| D_\d^s u \right\|_{L^p(\Rn)} = C \left\| D_\d^s u \right\|_{L^p(\O)} .
	\]
\end{proof}

A nonlocal Poincar\'e inequality is obtained as a corollary. 

\begin{teo}\label{th:Poincare}
Let $1 < p < \infty$.
	Then there exists $C>0$ such that for all $u \in H_0^{s,p,\d}(\O_{-\d})$,
	\[
	\left\| u \right\|_{L^p (\O)} \leq C \left\| D_\d^s u \right\|_{L^p(\O)}.
	\]
\end{teo}
\begin{proof}
If $sp < n$, the result is a particular case of Theorem \ref{th:Poincare H0delta}.
If $sp \geq n$, we take any $q$ satisfying
\begin{equation}\label{eq:1qp}
  1 < q \leq p , \qquad s q < n \quad \text{and} \quad p \leq q_s^* ,
\end{equation}
which is easily seen to exist.
Indeed, if $n\geq 2$ we can take $q = \frac{n p}{n + sp}$, while if $n=1$ we choose any $q$ such that
\[
 1 < q < \frac{1}{s} , \qquad \frac{p}{1 + s p} \leq q , 
\]
which exists because $1 < \frac{1}{s}$ and $\frac{p}{1 + s p} < \frac{1}{s}$.

Once $q$ is chosen, by Theorem \ref{th:Poincare H0delta} and \eqref{eq:1qp}, we have for some $c_1, c_2, c_3 >0$,
\begin{equation*}%\label{eq:pqs}
	\left\| u \right\|_{L^p (\O)} \leq c_1 \left\| u \right\|_{L^{q_s^*} (\O)} \leq c_2 \left\| D_\d^s u \right\|_{L^q(\O)} \leq c_3 \left\| D_\d^s u \right\|_{L^p(\O)}  .
\end{equation*}
\end{proof}

Next we introduce a nonlocal analogue of Morrey's inequality, whose fractional version was shown in \cite[Th.\ 1.11]{ShS2015}.
Unlike their proof, which uses Morrey-type estimates of the Riesz transform, ours is based on the nonlocal fundamental theorem of Calculus in this context (Theorem \ref{Theo: nonlocal fundamental theorem of calculus}) together with the estimates of the kernel $V^s_{\d}$.

\begin{teo}\label{th:Morrey}
Let $1<p<\infty$ be such that $sp>n$. Then there exists $C>0$ such that for all $u \in H_0^{s,p,\d}(\O_{-\d})$,
	\begin{equation}\label{eq:Morreyae}
		\left| u(x) - u(y) \right| \leq C \left| x-y \right|^{s-\frac{n}{p}} \left\| D_\d^su \right\|_{L^p(\O)}, \qquad \text{a.e. } x, y \in \O
	\end{equation}
	and
\begin{equation}\label{eq:Morreysup}
\left\| u \right\|_{L^{\infty} (\O)} \leq C \left\| D_\d^s u \right\|_{L^p(\O)} .
\end{equation}
In addition, any $u \in H_0^{s,p,\d}(\O_{-\d})$ has a representative which is H\"older continuous of exponent $s-\frac{n}{p}$, and the continuous inclusion $H_0^{s,p,\d}(\O_{-\d}) \subset C^{0,s-\frac{n}{p}} (\overline{\O})$ holds.
\end{teo}
\begin{proof}
The core of the proof consists in showing that
\begin{equation}\label{eq:Morreyxy}
 \left| u(x) - u(y) \right| \leq C \left| x-y \right|^{s-\frac{n}{p}} \left\| D_\d^s u \right\|_{L^p(\O)}, \qquad  x, y \in \O 
\end{equation}
for all $u \in C_c^{\infty}(\O_{-\d})$.
Fix $x, y \in \O$ and $u \in C_c^{\infty}(\O_{-\d})$.
By Theorem \ref{Theo: nonlocal fundamental theorem of calculus} and Theorem \ref{Th: inverse Fourier trasnform as a function}\,\emph{\ref{property of V d)})} there exists $C>0$ such that
	\begin{equation} \label{eq: Morrey ineq 1} 
		\begin{split}
			|u(x)-u(y)| &=  
			 \left| \int_{\Rn} D_\d^s u(z) \cdot \left[ V_\d^s(x-z) - V_\d^s(y-z) \right] dz\right|   \\
			 &\leq \int_{\O} \left| V_\d^s(x-z)-   V_\d^s(y-z)\right| |D_\d^s u(z)|  dz \\
			& \leq C\int_{\O} \left| \frac{x-z}{|x-z|^{n+1-s}}-  \frac{y-z}{|y-z|^{n+1-s}} \right| \left| D_\d^s u(z) \right| dz  .
		\end{split}
	\end{equation}
Now define $r:=|x-y|$.
Continuing with \eqref{eq: Morrey ineq 1}, we have
\begin{equation} \label{eq: Morrey ineq 2} 
	\begin{split}
		\left| u(x) - u(y) \right| \leq & C\int_{B(x, 2r)} \left| x-z \right|^{s-n} |D_\d^s u(z)| dz + C\int_{B(x, 2r)} \left| y-z \right|^{s-n} \left| D_\d^s u(z) \right| dz \\
	& + C\int_{B(x, 2r)^c} \left| \frac{x-z}{|x-z|^{n+1-s}}-  \frac{y-z}{|y-z|^{n+1-s}} \right| \left| D_\d^s u(z) \right| dz .
	\end{split}
\end{equation}
For the first term we have that by H\"older's inequality,
\begin{equation} \label{eq: Morrey ineq 3} 
	\begin{split}
		\int_{B(x, 2r)} \left| x-z \right|^{s-n} \left| D_\d^s u(z) \right| dz & \leq \left( \int_{B(x, 2r)} \left| x-z \right|^{(s-n) p'} \, dz \right)^{\frac{1}{p'}} \left( \int_{B(x, 2r)} |D_\d^s u(z)|^p \, dz \right)^{\frac{1}{p}} \\
		& \leq (2r)^{s- \frac{n}{p}} \left( \frac{\s_{n-1} (p-1)}{sp - n} \right)^{\frac{1}{p'}} \left\| D_\d^s u \right\|_{L^p(\Rn)} ,
	\end{split}
\end{equation}
since $n + (s-n) p' = \frac{sp - n}{p-1} >0$.
With respect to the second term, we use the inclusion $B (x, 2r) \subset B (y, 3r)$ and an analogous calculation as in \eqref{eq: Morrey ineq 3} allows us to obtain
\begin{equation} \label{eq: Morrey ineq 4} 
\begin{split}
		\int_{B(x, 2r)} \left| y-z \right|^{s-n} \left| D_\d^s u(z) \right| dz & \leq \int_{B(y, 3r)} \left| y-z \right|^{s-n} \left| D_\d^s u(z) \right| dz \\
	& \leq (3r)^{s- \frac{n}{p}} \left( \frac{\s_{n-1} (p-1)}{sp - n} \right)^{\frac{1}{p'}}  \left\| D_\d^s u \right\|_{L^p(\Rn)} .
\end{split}
\end{equation}

Finally, so as to tackle the last term, by the fundamental theorem of Calculus,
\begin{align*}
	&\left| \frac{x-z}{|x-z|^{n+1-s}}-  \frac{y-z}{|y-z|^{n+1-s}} \right|= \left| \int_{0}^{1} \frac{d}{dt} \left[ \frac{tx+(1-t)y-z}{|tx+(1-t)y -z|^{n+1-s}} \right] dt \right|  \\
	& = \left| \int_{0}^1 \frac{x-y}{|tx+(1-t)y-z|^{n+1-s}} - (n+1-s) \frac{[tx+(1-t)y-z] [(tx+(1-t)y-z)\cdot (x-y)]}{|tx+(1-t)y-z|^{n+3-s}} \, dt \right| \\
	& \leq \int_{0}^1 \left[ \frac{r}{|tx+(1-t)y-z|^{n+1-s}} + (n+1-s) \frac{r}{|tx+(1-t)y-z|^{n+1-s}} \right] dt \\
	& = (n+2-s) r \int_{0}^1 \frac{1}{|tx+(1-t)y-z|^{n+1-s}} dt ,
\end{align*}
so
\begin{equation*}
	\begin{split}
		&\int_{B(x, 2r)^c}\left| \frac{x-z}{|x-z|^{n+1-s}}-  \frac{y-z}{|y-z|^{n+1-s}} \right|\left| D_\d^s u(z) \right| dz \\
		& \leq (n+2-s) r \int_0^1 \int_{B(x, 2r)^c} \left| tx + (1-t)y - z \right|^{s-n-1} \left| D_\d^s u(z) \right| dz \, dt .
	\end{split}
\end{equation*}
By H\"older's inequality,
\begin{align*}
 & \int_{B(x, 2r)^c} \left| tx + (1-t)y - z \right|^{s-n-1} \left| D_\d^s u(z) \right| dz \\
& \leq \left( \int_{B(x, 2r)^c} \left| tx + (1-t)y - z \right|^{(s-n-1)p'} \, dz \right)^{\frac{1}{p'}} \left\| D_\d^s u \right\|_{L^p(\Rn)} .
\end{align*}
Since $B(tx + (1-t)y, r) \subset B(x, 2r)$ for all $t \in [0,1]$, we have
\begin{align*}
	\int_{B(x, 2r)^c} \left| tx + (1-t)y - z \right|^{(s-n-1)p'} \, dz & \leq \int_{B(tx + (1-t)y, r)^c} \left| tx + (1-t)y - z \right|^{(s-n-1)p'} \, dz \\
	& = \frac{\s_{n-1}}{(n+1-s)p' - n} r^{n + (s-n-1)p'} ,
\end{align*}
since $n + (s-n-1) p' = - \frac{(1-s)p + n}{p-1} < 0$.
Putting together the last three inequalities, we can see that there exists $\tilde{C}=\tilde{C}(s,n,p) > 0$ such that
\begin{equation} \label{eq: Morrey ineq 5} 
\begin{split}
		\int_{B(x, 2r)^c}\left| \frac{x-z}{|x-z|^{n+1-s}}-  \frac{y-z}{|y-z|^{n+1-s}} \right| \left| D_\d^s u(z) \right| dz & \leq\tilde{C} \, r^{[n+(s-n-1)p']\frac{1}{p'}+1} \left\| D_\d^s u \right\|_{L^p(\Rn)} \\
	& =\tilde{C} \, r^{s-\frac{n}{p}} \left\| D_\d^s u \right\|_{L^p(\Rn)} .
\end{split}
\end{equation}
Then, inequality \eqref{eq:Morreyxy} follows combining \eqref{eq: Morrey ineq 2}, \eqref{eq: Morrey ineq 3}, \eqref{eq: Morrey ineq 4} and \eqref{eq: Morrey ineq 5}, as well as the inclusion $\supp D_\d^s u \subset \O$, which implies $\| D_\d^s u \|_{L^p(\Rn)}=\| D_\d^su \|_{L^p(\O)}$.
Once  \eqref{eq:Morreyxy} is established, inequality \eqref{eq:Morreyae} follows from a standard density argument.

In order to show inequality \eqref{eq:Morreysup} and the continuous inclusion $H_0^{s,p,\d}(\O_{-\d}) \subset C^{0,s-\frac{n}{p}} (\overline{\O})$, by a density argument, it is enough to prove \eqref{eq:Morreysup} for $u \in C^{\infty}_c (\O_{-\d})$.
Let $x \in \O$ and $x_0 \in \O \setminus \O_{-\d}$.
By \eqref{eq:Morreyxy},
\[
 \left| u(x) \right| = \left| u(x) - u (x_0) \right| \leq C \left| x - x_0 \right|^{s -\frac{n}{p}} \left\| D_\d^s u \right\|_{L^p(\O)} \leq C \left( \diam \O \right)^{s -\frac{n}{p}} \left\| D_\d^s u \right\|_{L^p(\O)} ,
\]
where $\diam$ stands for the diameter of a set. The proof is concluded.
\end{proof}

The limiting case $s p = n$ is covered by the following version of Trudinger's inequality.
Its proof is a straightforward adaptation of the classical one (see, e.g., \cite[Th.\ 7.15]{GiTr01}) but using inequality \eqref{eq:upoint}.
Its fractional version can be found in \cite[Th.\ 1.10]{ShS2015}.
We denote by $|\O|$ the measure of $\O$.

\begin{teo}\label{th:Trudinger}
Let $1<p<\infty$ be such that $s p =n$.
Then there exist $c_1, c_2 >0$ such that for all $u \in H_0^{s,p,\d}(\O_{-\d})$,
\[
 \int_{\O} \exp \left( \frac{|u (x)|}{c_1 \| D_\d^s u \|_{L^p (\O)}} \right)^{p'} dx \leq c_2 \left| \O \right| .
\]
\end{teo}
\begin{proof}
By a standard density argument, it is enough to prove the inequality for $C^{\infty}_c (\O_{- \d})$ functions, so let $u \in C^{\infty}_c (\O_{- \d})$ and set
\[
 g(x) = \int_{\O} \frac{|D_\d^s u(y)|}{|x-y|^{n-s}} \, dy .
\]
By \eqref{eq:upoint},
\[
 \left| u \right| \leq C \, g .
\]
while by \cite[Lemma 7.13]{GiTr01}, for some constants $c_1', c_2 >0$,
\begin{equation}\label{eq:GilbargTrudinger}
 \int_{\O} \exp \left( \frac{g(x)}{c_1' \| D_\d^s u \|_{L^p (\O)}} \right)^{p'} dx \leq c_2 \left| \O \right| .
\end{equation}
Putting together these two inequalities, we obtain the conclusion.
\end{proof}
We mention that the constants $c_1', c_2$ of \eqref{eq:GilbargTrudinger} do not depend on $\O$, but the constant $C$ of \eqref{eq:upoint} does.
That is why the constant $c_1$ of Theorem \ref{th:Trudinger} depends on $\O$, but not the constant $c_2$.

We end this section with the analogue of Hardy's inequality.
Its fractional version can be found in \cite[Th.\ 1.9]{ShS2015}, whose proof (as well as the classical one \cite{StWe58}) is easily adapted to our context.

\begin{teo}%\label{th:Hardy}
Let $1 < p < \infty$ be with $s p<n$.
Then, there exists $C>0$ such that for all $u \in H_0^{s,p,\d}(\O_{-\d})$, 
\[
 \left( \int_{\O} \frac{|u(x)|^p}{|x|^{s p}} \, d x \right)^{\frac{1}{p}} \leq C \left\| D^s_{\d} u \right\|_{L^p (\O)} .
\]
\end{teo}
\begin{proof}
As before, it is enough to establish the inequality for $u \in C^{\infty}_c (\O_{-\d})$.
The proof is just a combination of inequality \eqref{eq:upoint} together with the classical Hardy inequality for Riesz potentials due to Stein and Weiss \cite{StWe58} (see also \cite[Lemma 2.8]{ShS2015}):
\begin{align*}
 \int_{\O} \frac{|u(x)|^p}{|x|^{s p}} \, d x & \leq c_1 \int_{\O} \frac{(I_s * |D^s_{\d} u|)(x)^p}{|x|^{s p}} \, d x \leq c_1 \int_{\Rn} \frac{(I_s * |D^s_{\d} u|)(x)^p}{|x|^{s p}} \, d x \leq c_2 \left\| D^s_{\d} u \right\|_{L^p (\Rn)}^p \\
 & = c_2 \left\| D^s_{\d} u \right\|_{L^p (\O)}^p ,
\end{align*}
for some constants $c_1, c_2>0$.
\end{proof}

\section{Compact embeddings}\label{se: Compact Embedding}

In this section we will use the nonlocal fundamental theorem of Calculus (Theorem \ref{Theo: nonlocal fundamental theorem of calculus}) to prove compact embeddings of the spaces $H^{s,p,\delta}_g (\O_{-\d})$ into $L^q (\O)$ spaces.

We start with the following H\"older estimate of the function $\frac{x}{|x|^{n+1-s}}$.
In fact, this result is included in the proof of \cite[Prop.\ 3.14]{COMI2019}, but we provide a full proof for the comfort of the reader.

\begin{lem}\label{ineq: Claim 1}
	There exists a constant $C>0$, such that for every $s \in (0,1)$ and $h \in \Rn$,
	\begin{equation*} 
	\int_{\Rn} \left| \frac{z}{|z|^{n+1-s}}- \frac{z-h}{|z-h|^{n+1-s}} \right|dz \leq \frac{C \left| h \right|^s}{s(1-s)} .
	\end{equation*}
\end{lem}
\begin{proof}
We first show that there exists a constant $C>0$, such that for every $s \in (0,1)$ we have
	\begin{equation}\label{eq:w-e1} 
	\int_{\Rn} \left| \frac{w}{|w|^{n+1-s}}- \frac{w-e_1}{|w-e_1|^{n+1-s}} \right|dw \leq \frac{C}{s(1-s)} .
	\end{equation}
%	where $e_1$ is the first vector of the canonical basis of $\Rn$.
	On the one hand, we have that
	\begin{equation*}
	\int_{B(0,2)} \left| \frac{w}{|w|^{n+1-s}}- \frac{w-e_1}{|w-e_1|^{n+1-s}} \right|dw \leq C \int_{B(0,2)}\frac{1}{|w|^{n-s}}dw \leq C \, \frac{2^s}{s} \leq \frac{C}{s}.
	\end{equation*}
	On the other hand, for a fixed $w \in B(0,2)^c$,
	\begin{align*}
	& \left| \frac{w}{|w|^{n+1-s}}- \frac{w-e_1}{|w-e_1|^{n+1-s}} \right|= \left| \int_{0}^{1} \frac{d}{dt} \frac{w-te_1}{|w-te_1|^{n+1-s}}dt \right| \\ & = \left| \int_{0}^{1} (n+1-s) \frac{[(w-te_1) \cdot e_1] (w-te_1)}{|w-te_1|^{n+3-s}}- \frac{e_1}{|w-te_1|^{n+1-s}}dt \right| \leq C \int_{0}^{1} \frac{1}{|w-te_1|^{n+1-s}}dt .
	\end{align*}
	Now, for $w \in B(0,2)^c$ and $t \in [0,1]$ we have
	\[
	\left| w-te_1 \right| \geq |w| - t \geq |w| - 1 \geq \frac{1}{2} |w| ,
	\]
	so
	\[
	\int_0^1 \frac{1}{|w-te_1|^{n+1-s}} dt \leq 2^{n+1-s} \frac{1}{|w|^{n+1-s}} \leq 2^{n+1} \frac{1}{|w|^{n+1-s}} .
	\]
	By integration, we obtain that
	\begin{align*}
	\int_{B(0,2)^c}\left| \frac{w}{|w|^{n+1-s}}- \frac{w-e_1}{|w-e_1|^{n+1-s}} \right|dw \leq 
	C \int_{B(0,2)^c} \frac{1}{|w|^{n+1-s}}dw \leq C \, \frac{2^{-1+s}}{1-s} \leq \frac{C}{1-s}.
	\end{align*}
	This yields \eqref{eq:w-e1}.
	
In order to complete the proof, we take a rotation $R$ such that $R^T h=|h|e_1$.
	Then, making the change of variables $z=|h|Rw$ and using \eqref{eq:w-e1}, we arrive at
	\begin{align*}
	\int_{\Rn} \left| \frac{z}{|z|^{n+1-s}}- \frac{z-h}{|z-h|^{n+1-s}} \right|  dz 
	= |h|^s \int_{\Rn} \left| \frac{w}{|w|^{n+1-s}}- \frac{w-e_1}{|w-e_1|^{n+1-s}} \right|dw  
	\leq |h|^s  \frac{C}{s (1-s)}  .
	\end{align*}
\end{proof}

A key ingredient of the desired compactness result is the application of the Fr\'echet--Kolmogorov criterion, for which the following estimate on the translations is crucial.
The analogous result in the Sobolev case is classical \cite[Prop.\ 9.3]{Brezis}.
The next result is inspired by \cite[Prop.\ 3.14]{COMI2019}, where they proved a fractional version when $p=1$.

%\textcolor{gray}{\textit{Observaci\'on para nosotros. Aqu\'i, gracias a que suponemos por ahora soporte compacto en $\O_{-\d}$, la norma de la derecha sale definida sobre $\O$, y nos permite salir adelante sin usar un resultado de extensi\'on todav\'ia. De todas formas, en este resultado tambi\'en ten\'ia sentido que la integral de la derecha (la de la norma) estuviese definida sobre un conjunto con un radio $\d$ mayor que la integral de la derecha, ya que as\'i pasa en el caso cl\'asico (Brezis \cite[Proposition 9.3]{Brezis}), donde consideran $\omega \Subset \O$ tal que $|h|<\dist(\omega, \partial \O)$. En el caso clásico, para probar la inmersi\'on compacta luego se complementa con un resultado de extensi\'on. En nuestro caso sigue teniendo sentido ya que las $u$ están definidas en $\O_\d$.}}

\begin{prop} \label{Prop: acotaciones de la traslacion en Hsp}
\begin{enumerate}[a)]
\item\label{item:p>1}
Let $1< p < \infty$.
Then there exists $C>0$ such that for all $u \in H_0^{s,p,\d}(\O_{-\d})$ and $h \in \Rn$,
	\begin{equation}\label{eq: acotaciones de la traslacion en Hsp}
		\left( \int_{\O}|u(x+h)-u(x)|^p \, dx \right)^{\frac{1}{p}} \leq C \left| h \right|^s \left\| D_{\d}^s u \right\|_{L^p(\O)} .
	\end{equation} 

\item\label{item:p=1}
Let $p =1$.
Then for all $M >0$ there exists $C >0$ such that for all $u \in H_0^{s,p,\d}(\O_{-\d})$ and $h \in B(0, M)$, inequality \eqref{eq: acotaciones de la traslacion en Hsp} holds.

\end{enumerate}
\end{prop}
\begin{proof} By a standard density argument, it is enough to prove the result for $u \in C_c^{\infty}(\O_{-\d})$.

We start with case \emph{\ref{item:p>1})}.
Let us fix $M > 0$ such that $x + h \notin \O_{-\d}$ for all $x \in \O$ and $h \in B(0,M)^c$.
Then, by Theorem \ref{th:Poincare},
\[
 \left( \int_{\O}|u(x+h)-u(x)|^p \, dx \right)^{\frac{1}{p}} = \left\| u \right\|_{L^p (\O)} \leq C \left\| D_{\d}^s u \right\|_{L^p(\O)} \leq \frac{C}{M^s} \left| h \right|^s \left\| D_{\d}^s u \right\|_{L^p(\O)} ,
\]
and the proof is concluded in this case.

In the rest of the proof we consider $h \in B(0,M)$.
As $\supp D_\d^s u \subset \O$, there exists $R>0$ such that $D_\d^s u (x-z) = 0$ for all $x \in \O$ and $z \in B(0, R)^c$.
Let $x \in \O$.
By Theorem \ref{Theo: nonlocal fundamental theorem of calculus},
	\begin{equation} \label{eq:acotaciones de la traslacion en Hsp 2 primer paso}
		\begin{split}
			|u(x+h)-u(x)| &=  
			\left| \int_{\Rn} \left( V_\d^s(z)- V_\d^s(z+h) \right)\cdot D_\d^s u(x-z) dz\right|   \\
			& \leq \int_{B(0,R)} \left| V_\d^s(z)- V_\d^s(z+h)  \right| \left| D_\d^s u(x-z) \right| dz  .
		\end{split}
	\end{equation}
By Theorem \ref{Th: inverse Fourier trasnform as a function}\,\emph{\ref{item:Vd})}, there exists $C>0$ such that
	\begin{equation}\label{eq:VLipschiz}
		\left| V_\d^s(z) -V_\d^s(z+h) \right| \leq C \left| \frac{z}{|z|^{n+1-s}}- \frac{z+h}{|z+h|^{n+1-s}} \right|, 
	\end{equation}	
for all $z\in B(0,R)$.
Thus, applying H\"older's inequality to the right-hand side of \eqref{eq:acotaciones de la traslacion en Hsp 2 primer paso},
	\begin{align*}
		&|u(x+h)-u(x)| \leq \\
		&C \left( \int_{B(0,R)} \left| \frac{z}{|z|^{n+1-s}}- \frac{z+h}{|z+h|^{n+1-s}} \right| |D_\d^s u(x-z)|^p dz \right)^{\frac{1}{p}}
		\left(\int_{B(0,R)} \left| \frac{z}{|z|^{n+1-s}}- \frac{z+h}{|z+h|^{n+1-s}} \right|dz\right)^{\frac{1}{p'}}\\
		&\leq \left( \frac{C |h|^s}{s(1-s)}\right)^{\frac{1}{p'}} \left( \int_{B(0,R)}\left| \frac{z}{|z|^{n+1-s}}- \frac{z+h}{|z+h|^{n+1-s}} \right| |D_\d^s u(x-z)|^p dz \right)^{\frac{1}{p}} ,
	\end{align*}
	where we have used Lemma \ref{ineq: Claim 1}. %(\textcolor{brown}{con el operador $\left| \frac{z}{|z|^{n+1-s}}- \frac{z+h}{|z+h|^{n+1-s}} \right|$ trabajamos igual que en el caso de gamma convergencia}). 
	Next, we integrate and apply Fubini's theorem to obtain
	\begin{align*}
		\int_{\O} |u(x+h)-u(x)|^p \, dx&\leq \left( \frac{C|h|^s}{s(1-s)}\right)^{p/{p'}} \int_{B(0,R)} \left| \frac{z}{|z|^{n+1-s}}- \frac{z+h}{|z+h|^{n+1-s}} \right|\int_{\O} |D_\d^s u(x-z)|^p \, dx \, dz \\
		&\leq  \left( \frac{C|h|^s}{s(1-s)}\right)^{p/{p'}+1} \| D_\d^s u \|_{L^p(\Rn)}^p=\left( \frac{C|h|^s}{s(1-s)}\right)^{p} \| D_\d^s u \|_{L^p(\O)}^p ,
	\end{align*}
where we have applied Lemma \ref{ineq: Claim 1} again.
This completes the proof of \emph{\ref{item:p>1})}.

Now we prove \emph{\ref{item:p=1})} by following the same lines as in case \emph{\ref{item:p>1})}.
As $\supp D_\d^s u \subset \O$, there exists $R>0$ such that $D_\d^s u (x-z) = 0$ for all $x \in \O$ and $z \in B(0, R)^c$.
Fix $M > 0$ and consider $h \in B(0,M)$.
Let $x \in \O$.
As in \eqref{eq:acotaciones de la traslacion en Hsp 2 primer paso}--\eqref{eq:VLipschiz},
\[
|u(x+h)-u(x)| \leq C \int_{B(0,R)} \left| \frac{z}{|z|^{n+1-s}}- \frac{z+h}{|z+h|^{n+1-s}} \right| \left| D_\d^s u(x-z) \right| dz .
\]
Using Lemma \ref{ineq: Claim 1} we find that, for some $C_1 >0$
\begin{align*}
 \int_{\O} \left| u(x+h)-u(x) \right| dx & \leq C \int_{B(0,R)} \left| \frac{z}{|z|^{n+1-s}}- \frac{z+h}{|z+h|^{n+1-s}} \right| \int_{\O} |D_\d^s u(x-z)| \, dx \, dz \\
 & \leq \frac{C_1 |h|^s}{s(1-s)} \| D_\d^s u \|_{L^1(\O)} ,
\end{align*}
which concludes the proof. 
\end{proof}

The main result of this section is the following compact embedding, which is an analogue of the Rellich--Kondrachov theorem.
See \cite[Th.\ 2.2]{ShS2018} for the fractional case; their proof uses Ascoli--Arzel\`a's theorem to mollifiers of the sequence, while we prefer to invoke directly the Fr\'echet--Kolmogorov criterion.
Recall the notation $p_s^*$ from \eqref {eq:p*s}.

\begin{teo}\label{th:compactembedding}
Let $g \in H^{s,p,\d}(\O)$. Then, for any sequence $\{u_j\}_{j \in \mathbb{N}}\subset H_g^{s,p,\d}(\O_{-\d})$ such that
	\[
	u_j \weakc u \quad \text{ in } H^{s,p,\d}(\O),
	\]
	for some $u \in H^{s,p,\d}(\O)$, one has $u \in H_g^{s,p,\d}(\O_{-\d})$ and:
\begin{enumerate}[a)]
\item\label{item:compactp>1}
if $p>1$,
	\[
	u_j \to u \quad \text{ in } L^q(\O),
	\]
	for every $q$ satisfying
	\begin{equation*}
	\begin{cases}
		q \in [1,p_s^*) & \text{ if } sp<n, \\
		q \in [1, \infty) & \text{ if } sp=n, \\
		q \in [1,\infty] & \text{ if } sp>n.
	\end{cases}
	\end{equation*}
	
\item\label{item:compactp=1} if $p=1$,
\[
 u_j \to u \quad \text{ in } L^1(\O) .
\]

\end{enumerate}
\end{teo}
\begin{proof}
Clearly, $u \in H_g^{s,p,\d}(\O_{-\d})$, since $H_g^{s,p,\d}(\O_{-\d})$ is a closed affine subspace of $H^{s,p,\d}(\O_{-\d})$.
By linearity, we can assume $g=0$.

	The case $sp>n$ implies $p>1$ and follows from Theorem \ref{th:Morrey} and the Ascoli--Arzel\`a theorem. The case $sp=n$ reduces to the case $sp<n$ thanks to Proposition \ref{prop: simple inclussions}\,\emph{\ref{simple inclussion 1})}. Thus, we focus on the case $sp<n$.
Moreover, the case $q < p$ reduces to the case $q \geq p$ thanks to the inclusions $L^{p_2} (\O) \subset L^{p_1} (\O)$ for all $1 \leq p_1 \leq p_2$, so we can assume that $q \in [p, p_s^*)$.

Let $M>0$ be such that $\|u_j\|_{H^{s,p,\d}(\O)} \leq M$ for each $j \in \N$.

We start with \emph{\ref{item:compactp>1})}, so we assume $p>1$.
By Proposition \ref{Prop: acotaciones de la traslacion en Hsp} we have that for $j \in \N$ and $h \in \Rn$,
	\begin{equation} \label{eq: NL mean value Theorem H0delta}
	\| \tau_h u_j - u_j \|_{L^p(\O)} \leq C \left| h \right|^s \left\| D_{\d}^s u_j \right\|_{L^p(\O)},
	\end{equation}
with $\tau_h u_j = u_j (\cdot-h)$ and some $C>0$.
Next, as $p\leq q < p^*_s$, we can write
	\[
	\frac{1}{q}=\frac{\alpha}{p}+\frac{1-\alpha}{p^*_s} \qquad \text{for some } \alpha \in (0,1].
	\]
Using the interpolation inequality, \eqref{eq: NL mean value Theorem H0delta}, the triangular inequality and Theorem \ref{th:Poincare H0delta},
	\begin{align*}
	\left\| \tau_h u_j - u_j \right\|_{L^q(\O)} & \leq \left\| \tau_h u_j - u_j \right\|_{L^p(\O)}^{\a} \left\| \tau_h u_j - u_j \right\|_{L^{p^*_s}(\O)}^{1-\a} \\
	&\leq \left(C|h|^s \right)^{\a} \left\| D_{\d}^s u_j \right\|_{L^p(\O)}^{\a} \left( 2 \left\| u_j \right\|_{L^{p^*_s}(\O)} \right)^{1-\a} \\
	&\leq (2 C_1)^{1-\a} \left( C|h|^s \right)^{\a} \left\| D_{\d}^s u_j \right\|_{L^p(\O)} \leq  (2 C_1)^{1-\a} M \left( C|h|^s \right)^{\a} ,
	\end{align*}
for some $C_1 >0$. Thus,
\[
 \lim_{h \to 0} \sup_{j \in \N} \| \tau_h u_j - u_j \|_{L^q(\O)} = 0 .
\]
As a result, the Fr\'echet--Kolmogorov criterion leads to the compactness of $\{ u_j \}_{j \in \N}$ in $L^q(\O)$.

Now we show \emph{\ref{item:compactp=1})}, so we assume $p=1$.
Fix $M_1 >0$.
By Proposition \ref{Prop: acotaciones de la traslacion en Hsp} we have that for $j \in \N$ and $h \in B(0, M_1)$
\[
 \| \tau_h u_j - u_j \|_{L^1(\O)} \leq C \left| h \right|^s \left\| D_{\d}^s u_j \right\|_{L^1(\O)} \leq C \, M \left| h \right|^s ,
\]
for some $C>0$.
Again the Fr\'echet--Kolmogorov criterion concludes the compactness of $\{ u_j \}_{j \in \N}$ in $L^1(\O)$.
\end{proof}

Of course, Theorem \ref{th:Morrey} yields additionally the compact inclusion of $H_0^{s,p,\d}(\O_{-\d})$ into $C^{0,\beta} (\overline{\O})$ for any $0 < \beta < s-\frac{n}{p}$ under the range $s p > n$.

\section[Existence of minimizers of convex functionals]{Existence of minimizers and the Euler--Lagrange equation} \label{se: existence of minimizers}

In this final section, we prove the existence of minimizers of functionals of the form
\begin{equation}\label{eq:I}
 I(u) = \int_{\O} W (x, u(x), D^s_\d u (x)) \, dx
\end{equation}
under coercivity and convexity conditions.
We also show the corresponding (nonlocal) Euler--Lagrange equations satisfied by the minimizers.

From now on, $\mc{L}^n$ denotes the Lebesgue sigma-algebra in $\Rn$, whereas $\mc{B}$ and $\mc{B}^n$ denote the Borel sigma-algebras in $\R$ and $\Rn$, respectively.
The result on the existence of minimizers, which is a standard application of the direct method of the Calculus of Variations, is as follows.

\begin{teo}%\label{th:existence}
%Let $\O$ be a bounded open subset of $\Rn$.
Let $1 < p < \infty$.
Let $u_0 \in H^{s,p,\d} (\O)$.
Let $W : \O \times \R \times \Rn \to \R \cup \{ \infty \}$ satisfy the following conditions:
\begin{enumerate}[a)]
\item $W$ is $\mc{L}^n \times \mc{B} \times \mc{B}^n$-measurable.

\item $W (x, \cdot, \cdot)$ is lower semicontinuous for a.e.\ $x \in \O$.

\item For a.e.\ $x \in \O$ and every $y \in \R$, the function $W (x, y, \cdot)$ is convex.

\item\label{item:Ecoerc}
There exist $c>0$ and $a \in L^1 (\O)$ such that
\[
 W (x, y, z) \geq a (x) + c \left| z \right|^p
\]
for a.e.\ $x \in \O$, all $y \in \R$ and all $z \in \Rn$.
\end{enumerate}
Define $I$ as in \eqref{eq:I}, and assume that $I$ is not identically infinity in $H_{u_0}^{s,p,\d} (\O_{-\d})$.
Then there exists a minimizer of $I$ in $H_{u_0}^{s,p,\d} (\O_{-\d})$.
\end{teo}
\begin{proof}
Assumption \emph{\ref{item:Ecoerc})} shows that the functional $I$ is bounded below by $\int_{\Rn}a$.
As $I$ is not identically infinity in $H_{u_0}^{s,p,\d} (\O_{-\d})$, there exists a minimizing sequence $\{ u_j \}_{j \in \N}$ of $I$ in $H_{u_0}^{s,p,\d} (\O_{-\d})$.
Then, assumption \emph{\ref{item:Ecoerc})} implies that $\{ D^s_\d u_j \}_{j \in \N}$ is bounded in $L^p (\O, \Rn)$.
By Theorem \ref{th:Poincare} applied to $u_j - u_0$, we obtain that $\{ u_j - u_0\}_{j \in \N}$ and, hence, $\{ u_j \}_{j \in \N}$ are bounded in $L^p (\O)$.
Therefore, $\{ u_j \}_{j \in \N}$ is bounded in $H^{s,p,\d} (\O)$.
As $H^{s,p,\d} (\O)$ is reflexive (Proposition \ref{prop: espacio separable y reflexivo}), we can extract a weakly convergent subsequence.
Using Theorem \ref{th:compactembedding}, we obtain that there exists $u \in H^{s,p,\d} (\O)$ such that for a subsequence (not relabelled),
\begin{equation*}%\label{eq:convergence1poly}
 u_j \weakc u \text{ in } H^{s,p,\d} (\O) \quad \text{and} \quad u_j \to u \text{ in } L^p (\O) .
\end{equation*}
Moreover, $u \in H_{u_0}^{s,p,\d} (\O_{-\d})$.% since $H_0^{s,p,\d} (\O_{-\d})$ is a closed subspace of $H^{s,p,\d} (\O)$.

A standard lower semicontinuity result for convex functionals (see, e.g., \cite[Th.\ 7.5]{FoLe07}) shows that
\[
 I(u) \leq \liminf_{j \to \infty} I(u_j).
\]
Therefore, $u$ is a minimizer of $I$ in $H_{u_0}^{s,p,\d} (\O_{-\d})$ and the proof is concluded.
\end{proof}

We finally show the Euler--Lagrange equation satisfied by any minimizer.
The notation for partial derivatives is as follows: $D_y W (x, \cdot, z)$ is the derivative of $W (x, \cdot, z)$, and $D_z W (x, y, \cdot)$ is the derivative of $W (x, y, \cdot)$.

\begin{teo}%\label{th:EL}
Let $1 < p < \infty$.
Let $u_0 \in H^{s,p,\d} (\O)$.
Let $W : \O \times \R \times \Rn \to \R$ satisfy the following conditions:
\begin{enumerate}[a)]
\item $W (\cdot , y, z)$ is $\mc{L}^n$-measurable for each $y \in \R$ and $z \in \Rn$.

\item $W (x, \cdot, \cdot)$ is of class $C^1$ for a.e.\ $x \in \O$.

\item\label{item:growthW}
There exist $c>0$ and $a \in L^1 (\O)$ such that
\[
\left| W (x, y, z) \right| + \left| D_y W (x, y, z) \right| + \left| D_z W (x, y, z) \right| \leq a (x) + c \left( \left| y \right|^p + \left| z \right|^p \right) ,
\]
for a.e.\ $x \in \O$, all $y \in \R$ and all $z \in \Rn$.
\end{enumerate}
Define $I$ as in \eqref{eq:I}.
Let $u$ be a minimizer of $I$ in $H_{u_0}^{s,p,\d} (\O_{-\d})$.
Then, for every $\f \in C^{\infty}_c (\O_{-\d})$,
\begin{equation}\label{eq:WEL}
 \int_{\O} \left[ D_y W (x, u(x), D^s_\d u (x)) \, \f (x) + D_z W (x, u(x), D^s_\d u (x)) \cdot D^s_\d \f (x) \right] dx = 0 .
\end{equation}
If, in addition, $D_z W (\cdot, u(\cdot), D^s_\d u (\cdot)) \in C^1 (\overline{\O_{-\d}}, \Rn)$ then
\begin{equation}\label{eq:EL}
 D_y W (x, u(x), D^s_\d u (x)) = \diver^s_{\d} D_z W (x, u(x), D^s_\d u (x))
\end{equation}
for a.e.\ $x \in \O_{-\d}$.
\end{teo}
\begin{proof}
Using a standard argument, in order to show \eqref{eq:WEL} it is enough to check that one can differentiate under the integral sign in the function $t \mapsto I(u + t \f)$, since $u + t \f \in H_{u_0}^{s,p,\d} (\O_{-\d})$.
Assumption \emph{\ref{item:growthW})} shows that this is the case (see, e.g., \cite[Ch.\ 13, \S 2, Lemma 2.2]{Lang83}), so \eqref{eq:WEL} is proved.

Now we make the assumption $D_z W (\cdot, u(\cdot), D^s_\d u (\cdot)) \in C^1 (\overline{\O_{-\d}}, \Rn)$.
Then there exists a $C^1_c (\O, \Rn)$ extension of this function; we denote by $W_z$ any such extension.
In order to derive \eqref{eq:EL} from \eqref{eq:WEL}, we use Theorem \ref{th:Nl parts} to obtain
\begin{align*}
 \int_{\O} W_z (x) \cdot D^s_\d \f (x) \, dx =& - \int_{\O} \f (x) \diver_\delta^s W_z (x) \, dx \\
 &- (n-1+s) \int_{\O} \int_{\O_{B, \d}} \frac{\f(x) \, W_z (y)}{|x-y|} \cdot \frac{x-y}{|x-y|}\rho_\delta(x-y) \, dy \, dx .
\end{align*}
This last integral is zero; indeed, $\f(x) = 0$ for $x \in \O \setminus \O_{-\d}$, while for $x \in \O_{-\d}$ and $y \in \O_{B, \d}$ we have $\rho_\delta (x-y) = 0$.
Therefore,
\[
 \int_{\O} W_z (x) \cdot D^s_\d \f (x) \, dx = - \int_{\O} \f (x) \diver_\delta^s W_z (x) \, dx .
\]
We combine this equality with \eqref{eq:WEL} and apply the fundamental lemma of the Calculus of Variations to obtain that equality \eqref{eq:EL} holds for a.e.\ $x \in \O_{-\d}$.
\end{proof}

Note that \eqref{eq:EL} imposes an a.e.\ equality in $\O_{-\d}$ and not in $\O$, which is natural since in $\O_{\d} \setminus \O_{-\d}$ we already have the condition $u = u_0$.
Even though \eqref{eq:EL} prescribes a pointwise condition, it is nonlocal because of the presence of $D^s_{\d} u$.

\appendix

\appendixpage

\section{$V^s_{\d}$ when $n=1$}\label{se:1D}

In this appendix we explain the necessary changes in the proof of Theorem \ref{Th: inverse Fourier trasnform as a function}\,\emph{\ref{item:Va}}) when $n=1$.
We first need a result regarding the function $Z$ appearing therein.

\begin{lem} \label{lem: convdist pv Z}
Let $n =1$.
	Then:
\begin{enumerate}[a)]
\item\label{item:Za} The function $Z$ of \eqref{eq: difference of V hat and vector Riesz potential} can be identified with the tempered distribution
	\begin{equation}\label{eq:Z}
		\langle Z, \f \rangle = \int_0^{\infty} Z(\xi) ( \f (\xi) -\f(-\xi) ) \, d \xi , \qquad \f \in \mathcal{S}
	\end{equation}
	and we have the convergence
	\begin{equation}\label{eq:ZtoZ}
		Z \chi_{B(0, \e)^c} \to Z \quad \text{in } \mathcal{S}' \quad \text{as } \e \to 0 .
	\end{equation}

\item\label{item:Yb} The function
\begin{equation}\label{eq:Y}
 Y(\xi) = -\frac{i\xi}{2 \pi|\xi|^2}\frac{1}{\hat{Q}_{\d}^s(0)} \chi_{B(0,1)} (\xi)
\end{equation}
can be identified with the tempered distribution
\[
 \langle Y, \f \rangle = \int_0^{\infty} Y(\xi) ( \f (\xi) -\f(-\xi) ) \, d \xi , \qquad \f \in \mathcal{S} ,
\]
we have the convergence
\begin{equation}\label{eq:YtoY}
 Y \chi_{B(0, \e)^c} \to Y \quad \text{in } \mathcal{S}' \quad \text{as } \e \to 0
\end{equation}
and
\begin{equation}\label{eq::hatY}
 \hat{Y} (x) = \frac{-1}{\pi \hat{Q}_{\d}^s(0)} \int_0^1 \frac{1}{\xi} \sin (2 \pi \xi x)\, d \xi .
\end{equation}

\end{enumerate}
\end{lem}
\begin{proof}
We start with \emph{\ref{item:Za})}.
	Let us see that formula \eqref{eq:Z} defines a tempered distribution.
	By Propositions \ref{Prop: properties of the Fourier transform of Q} and \ref{prop: Fourier transform well defined}, there exists $C>0$ such that
	\[
	\left| Z(\xi) \right| \leq \frac{1}{2 \pi |\xi|} \frac{1}{\hat{Q}_{\d}^s(\xi)} + \frac{1}{a_0} \frac{1}{|2\pi\xi|^s} \leq \frac{C}{|\xi|} , \qquad |\xi| \leq 1 .
	\]
	Thus, by the mean value theorem 
	\begin{equation}\label{eq:estimateZ1}
		\left| \int_0^1 Z(\xi) ( \f (\xi) -\f(-\xi) ) \, d \xi \right| \leq 2 C \| \f' \|_{\infty} .
	\end{equation}
	On the other hand, in the proof of Theorem \ref{Th: inverse Fourier trasnform as a function}\,\emph{\ref{item:Va})}, we saw that $Z$ decays to $0$ at infinity faster than any negative power of $|\xi|$.
	In particular, there exists $C>0$ such that
	\[
	\left| Z(\xi) \right| \leq \frac{C}{\xi^2} , \qquad |\xi| \geq 1 .
	\]
	Consequently, 
	\begin{equation}\label{eq:estimateZ2}
		\left| \int_1^{\infty} Z(\xi) ( \f (\xi) -\f(-\xi) ) \, d \xi \right| \leq 2 C \int_1^{\infty} \frac{1}{\xi^2} \, d \xi \, \| \f \|_{\infty} .
	\end{equation}
	Estimates \eqref{eq:estimateZ1} and \eqref{eq:estimateZ2} show that $Z$ defined by \eqref{eq:Z} is in $\mathcal{S}'$.
	
As before, the fact that $Z$ decays to $0$ at infinity faster than any negative power of $|\xi|$ implies that $Z \chi_{B(0, \e)^c} \in L^1 (\R)$ for all $\e>0$.
	In particular, $Z \chi_{B(0, \e)^c}$ considered as a distribution acts as follows: for each $\f \in \mathcal{S}$,
	\[
	\langle Z \chi_{B(0, \e)^c} , \f \rangle = \int_{B(0, \e)^c} Z(\xi) \, \f (\xi) \, d \xi = \int_{\e}^{\infty} Z(\xi) (\f(\xi)-\f(-\xi) )\, d\xi .
	\]
	As we saw in \eqref{eq:estimateZ1}--\eqref{eq:estimateZ2}, the function $Z(\xi) (\f(\xi)-\f(-\xi) )$ is in $L^1 ((0, \infty))$, so, by dominated convergence, we have that
	\[
	\int_{\e}^{\infty} Z(\xi) (\f(\xi)-\f(-\xi) )\, d\xi \to \int_0^{\infty} Z(\xi) (\f(\xi)-\f(-\xi) )\, d\xi \qquad \text{as } \e \to 0 ,
	\]
	which justifies the identification of the function $Z$ with the distribution \eqref{eq:Z} and shows the convergence \eqref{eq:ZtoZ}.
	
Thus, \emph{\ref{item:Za})} is proved.
The proof of \emph{\ref{item:Yb})} is analogous and we only write the details for the expression of $\hat{Y}$.
The function $Y\chi_{B(0, \e)^c}$ is in $L^1 (\Rn)$ for $0 < \e <1$ and
\[
 \F \left( Y\chi_{B(0, \e)^c} \right) = \int_{B(0,\e)^c} Y (\xi) e^{- 2 \pi i x \xi} \, d \xi = \frac{-1}{\pi \hat{Q}_{\d}^s(0)} \int_{\e}^1 \frac{1}{\xi} \sin (2 \pi \xi x)\, d \xi ,
\]
because of odd symmetry.
As the function $\xi \mapsto \frac{1}{\xi} \sin (2 \pi \xi x)$ is in $L^1 (0,1)$, and convergence \eqref{eq:YtoY} holds, we obtain expression \eqref{eq::hatY}.
\end{proof}

We are in a position to prove Theorem \ref{Th: inverse Fourier trasnform as a function}\,\emph{\ref{item:Va})}.

\begin{proof}[Proof of Theorem \ref{Th: inverse Fourier trasnform as a function}\,\emph{\ref{item:Va})} when $n=1$]
It remains to show that $W$, or, equivalently, $\hat{Z}$ is bounded.
To that end, it is useful to introduce the function $Y$ of \eqref{eq:Y} and express
\[
 \hat{Z}= \mathcal{F}(Z\chi_{B(0,1)} - Y) + \hat{Y} + \mathcal{F}(Z\chi_{B(0,1)^c}).
\]
Since $Z\chi_{B(0,1)^c} \in L^1(\R)$, by the Riemann--Lebesgue Lemma $\mathcal{F}(Z\chi_{B(0,1)^c}) \in C_0(\R)$.
Now we study the function
\[
 Z\chi_{B(0,1)} (\xi) - Y(\xi) = \left[ -\frac{i\xi}{2 \pi|\xi|^2} \left( \frac{1}{\hat{Q}_{\d}^s(\xi)} - \frac{1}{\hat{Q}_{\d}^s(0)} \right) - \frac{-i\xi}{a_0|\xi|}\frac{1}{|2\pi\xi|^s}  \right] \chi_{B(0,1)} (\xi) .
\]
Now, as we saw in \eqref{eq: difference of V in 0}--\eqref{eq:W2},

\[
 \sup_{\xi \in B(0,1)} \left| -\frac{i\xi}{2 \pi|\xi|^2} \left( \frac{1}{\hat{Q}_{\d}^s(\xi)} - \frac{1}{\hat{Q}_{\d}^s(0)} \right) \right| < \infty ,
\]
and, on the other hand,
\[
 \left| \frac{-i\xi}{a_0|\xi|}\frac{1}{|2\pi\xi|^s} \right| \leq \frac{1}{a_0}\frac{1}{|2\pi\xi|^s} ,
\]
which is integrable in $B(0,1)$.
Therefore, $Z\chi_{B(0,1)} - Y \in L^1 (\R)$, so $\mathcal{F}(Z\chi_{B(0,1)} - Y) \in C_0 (\R)$.

It remains to show that $\hat{Y}$ is bounded.
By Lemma \ref{lem: convdist pv Z}\,\emph{\ref{item:Yb})},
\[
 \hat{Y} (x) = \frac{-1}{\pi \hat{Q}_{\d}^s(0)} \int_0^1 \frac{1}{\xi} \sin (2 \pi \xi x)\, d \xi = \frac{-1}{\pi \hat{Q}_{\d}^s(0)} \int_0^x \frac{1}{\xi} \sin (2 \pi \xi)\, d \xi .
\]
This latter function is known to be bounded.
Therefore, $\hat{Z}$ is bounded and the proof is concluded. 
\end{proof}

\section{Fourier analysis results}\label{se:Fourier}

In this appendix we collect several Fourier analysis results needed throughout the paper.
They are possibly known to experts, but we have not found a precise reference.
 %The first result generalizes to $L^p(\R^n)$ spaces a classical result appearing in the literature for  $L^2(\Rn)$ or Schwartz spaces. 

First, we compute the Fourier transform of the vectorial version of the Riesz potential.

\begin{lem} \label{lemma: Fourier transform vector Riesz potential}
\begin{enumerate}[a)]
\item\label{item:FourierRieszn} If $n \geq 2$ and $0<\a<n-1$, then
\begin{equation*}% \label{eq: Fourier multiplier riesz T-P}
	\mathcal{F} \left( \frac{n-\a-1}{\gamma(1+\a)}\frac{x}{|x|^{n-\a+1}} \right) (\xi) = -i \frac{\xi}{|\xi|} |2 \pi \xi|^{-\a} = -i \frac{\xi}{|\xi|} \hat{I}_{\a}.
\end{equation*}

\item\label{item:FourierRiesz1} If $n=1$ and $0< s <1$, then
\begin{equation*}
	\mathcal{F} \left( c_{1,-s} \frac{x}{|x|^{2-s}} \right) = -\frac{ i \xi}{| \xi|}\frac{1}{|2\pi \xi|^s}.
\end{equation*}

\item\label{item:Fourier transform of x/|x|^n} $\displaystyle \mathcal{F}\left(\frac{1}{\s_{n-1}} \frac{x}{|x|^n} \right)(\xi) = -i\frac{\xi}{|\xi|}\frac{1}{|2\pi \xi|}$.
\end{enumerate}
\end{lem}
\begin{proof}
Fix $j \in \{ 1, \ldots, n\}$.
	On the one hand, we have that
	\[
	\frac{1}{\gamma(1+\a)}\frac{\partial }{\partial x_j}\frac{1}{|x|^{n-(\a+1)}} = -\frac{n-\a-1}{\gamma(1+\a)}\frac{x_j}{|x|^{n-\a+1}}.
	\]
	Thus, 
	\begin{equation} \label{eq: riesz potential transform 1}
		\frac{1}{\gamma(1+\a)} \mathcal{F} \left(\frac{\partial }{\partial x_j}\frac{1}{|x|^{n-(\a+1)}}\right) (\xi) = -\frac{n-\a-1}{\gamma(1+\a)} \mathcal{F} \left( \frac{x_j}{|x|^{n-\a+1}} \right) (\xi),
	\end{equation}
On the other hand, by standard properties of the Fourier transform, and, in particular, by \eqref{eq: Riesz Potential},
	\begin{equation} \label{eq: riesz potential transform 2}
		\frac{1}{\gamma(1+\a)} \mathcal{F} \left( \frac{\partial }{\partial x_j}\frac{1}{|x|^{n-(\a+1)}}\right) (\xi)= 2\pi i \xi_j \hat{I}_{1+\a}=2\pi i \xi_j |2\pi \xi|^{-(1+\a)}=i \frac{\xi_j}{|\xi|} |2\pi \xi|^{-\a}.
	\end{equation}
	Putting together \eqref{eq: riesz potential transform 1} and \eqref{eq: riesz potential transform 2} we obtain the conclusion of \emph{\ref{item:FourierRieszn})}.
	
	Now we present the proof of \emph{\ref{item:FourierRiesz1})}.
%, which, in fact, constitutes a proof of \emph{\ref{item:FourierRieszn})}, too. Pero hay que hacer una peque�a modificaci�n en la notaci�n para tratarlo componente a componente
We recall formula \eqref{eq:FFTC} of the fractional version of the fundamental theorem of Calculus: for every $u \in C^{\infty}_c(\R)$,
	% (in fact, for every $u \in \mathcal{S}$), \textcolor{blue}{Poner la demostraci\'on para $\mathcal{S}$}
\begin{equation*}
		u(x) = c_{1,-s} \int_{\R} D^s u(y) \frac{x-y}{|x-y|^{2-s}}dy.
\end{equation*}
   Next, we take Fourier transform and use the formula for the convolution of a distribution with a Schwartz function:
   \begin{equation*}
   	\hat{u}(\xi)=\widehat{D^s u}(\xi) \mathcal{F}\left( c_{1,-s}\frac{x}{|x|^{2-s}} \right)=\frac{2\pi i\xi}{|2\pi \xi|}|2\pi \xi|^s \hat{u}(\xi) \mathcal{F}\left( c_{1,-s}\frac{x}{|x|^{2-s}} \right) ,
   \end{equation*}
where we have used the explicit expression for Fourier transform of $D^s u$ (see \cite[Th.\ 1.4]{ShS2015} or \cite[Lemma 3.1]{BeCuMC21}). Now, we multiply both terms by $-i 2\pi \xi $ and obtain that
\begin{equation*}
	-i 2\pi \xi\hat{u}(\xi)=|2\pi \xi|^{1+s}\hat{u}(\xi) \mathcal{F}\left( c_{1,-s}\frac{x}{|x|^{2-s}} \right).
\end{equation*}
Since that equality holds for every $u \in C^{\infty}_c(\R)$, statement \emph{\ref{item:FourierRiesz1})} follows.

Finally, we prove \emph{\ref{item:Fourier transform of x/|x|^n})}.
Since $\frac{x}{|x|^n} \in L^1(B(0,1)) + L^\infty(B(0,1)^c)$ we have that $\frac{x}{|x|^n}$ belongs to $\mathcal{S}'$, and so does its Fourier transform.
Let $\f \in C^{\infty}_c (\Rn)$. 
We apply the Fourier transform to the representation formula of Proposition \ref{prop: classical representation result}, obtaining that
	\begin{equation*}
		\hat{\f}(\xi)=\widehat{\nabla \f}(\xi) \cdot \mathcal{F}\left(\frac{x}{\s_{n-1} |x|^n}\right)(\xi) = 2\pi i \xi\hat{ \f}(\xi) \cdot \mathcal{F}\left(\frac{x}{\s_{n-1} |x|^n}\right)(\xi).
	\end{equation*}
	Since this is true for every $\f \in C^{\infty}_c (\Rn)$ we infer that
\[
 1 = 2\pi i \xi \cdot \mathcal{F}\left(\frac{x}{\s_{n-1} |x|^n}\right)(\xi).
\]
Therefore, there exists a function $g : \Rn \to \Rn$ such that $\xi \cdot g(\xi) = 0$ and
\[
 \mathcal{F}\left(\frac{x}{\s_{n-1} |x|^n}\right)(\xi) = -i\frac{\xi}{|\xi|}\frac{1}{|2\pi \xi|} + g(\xi) .
\]
On the other hand, $\mathcal{F}\left(\frac{x}{\s_{n-1} |x|^n}\right)$ must be a vector radial function, as the Fourier transform of a vector radial function.
Consequently (recall Definition \ref{de:radial}), there exists $\bar{g}: [0, \infty) \to \R$ such that $g (\xi) = \xi \, \bar{g} (|\xi|)$.
Thus, $|\xi|^2 \bar{g} (|\xi|) = 0$, so $\bar{g} = 0$ and, hence, $g=0$ a.e.
The proof is concluded.
\end{proof}

Now, we recall the following definitions and properties about the convolution and Fourier transform of tempered distributions.
Recall from Section \ref{subse:Fourier} the notation $\tilde{f}$ for the reflection of $f$.
\begin{rem} \label{remark duality Fourier}
Let $u,v \in \mathcal{S}'$.
\begin{enumerate}[a)]
		\item $\tilde{v} \in \mathcal{S}'$ is defined as
		\[ \langle \tilde{v}, \f \rangle=\langle v, \tilde{\f}\rangle \quad \forall \, \f \in \mathcal{S}. \]
		\item\label{item:conv} Assume that $\tilde{v}* \f \in \mathcal{S}$ for every $\f \in \mathcal{S}$.
		Then the tempered distribution $v* u$ is defined as
		\[
		\langle v*u, \f \rangle= \langle u, \tilde{v}*\f \rangle \quad \forall \, \f \in \mathcal{S}.
		\]
	\item\label{item:tilde} We have that $\mathcal{F}(\hat{v})= \tilde{v}$.
\end{enumerate}
\end{rem}

Finally, we show that the product property of the Fourier transform of a convolution also holds for tempered distributions. 

\begin{lem} \label{lem: Fourier transform of convolution of distributions}
	Let $V,Q \in \mathcal{S}'$ be such that $Q$ is a distribution with compact support. Then
	\begin{equation*}
		\widehat{V*Q}=\hat{V} \, \hat{Q}.
	\end{equation*}
\end{lem}
\begin{proof}
	Firstly, we recall that the convolution $V*Q$ is well defined since $\tilde{Q}*\f \in \mathcal{S}$ for every $\f \in \mathcal{S}$ (see \cite[Th.\ 2.3.20]{Grafakos08a}), and its action is defined as in Remark \ref{remark duality Fourier}\,\emph{\ref{item:conv})}:
	\begin{equation*}
		\langle V*Q , \f \rangle=	\langle V , \f*\tilde{Q} \rangle \quad \text{ for every } \f \in \mathcal{S} .
	\end{equation*}
Now, by definition of the Fourier transform in the sense of distributions, for every $\f \in \mathcal{S}$,
	\begin{equation} \label{eq: process of the Fourier transform of a distribution with compact support}
		\langle \widehat{V*Q}, \f \rangle=	\langle V*Q, \hat{\f} \rangle=	\langle V, \hat{\f}*\tilde{Q} \rangle.
	\end{equation}
Next, by the Fourier transform of a convolution (of a distribution with a Schwartz function), 
\[\hat{\f}*\tilde{Q}= \mathcal{F}(\f \mathcal{F}^{-1}(\tilde{Q}))=\mathcal{F}(\f \, \hat{Q}),\]
 since $\mathcal{F}^{-1}(\tilde{Q})=\hat{Q}$ (see Remark \ref{remark duality Fourier}\,\emph{\ref{item:tilde})}). This also tells us that $\f \, \hat{Q}$ belongs to $\mathcal{S}$, because so does $\hat{\f}*\tilde{Q}$ (by the bijection of the Fourier transform in $\mathcal{S}$). Actually, it is known that $\hat{Q}$ is a smooth function (see \cite[Th.\ 2.3.21]{Grafakos08a}).
  Therefore, continuing with \eqref{eq: process of the Fourier transform of a distribution with compact support} and using again the duality of the Fourier transform,
 \begin{align*}
 	\langle \widehat{V*Q}, \f \rangle=	\langle V, \hat{\f}*\tilde{Q} \rangle=\langle V, \mathcal{F}(\f \, \hat{Q})\rangle=	\langle \hat{V}, \f \, \hat{Q} \rangle.
 \end{align*} 
As $\f \, \hat{Q} \in \mathcal{S}$, the product $\hat{V} \hat{Q}$ is well defined in a distributional sense and
\begin{equation*}
	\langle \hat{V} \hat{Q}, \f  \rangle=\langle \hat{V}, \f \, \hat{Q} \rangle=\langle \widehat{V*Q}, \f \rangle.
\end{equation*}
Consequently, the desired formula holds.
\end{proof}

\section*{Acknowledgements} 

We thank Davide Barbieri for useful discussions concerning the Fourier transform. This work has been supported by the Agencia Estatal de Investigaci\'on of the Spanish Ministry of Research and Innovation, through projects MTM2017- 83740-P and PID2020-116207GB-I00 (J.C.B. and J.C.), and MTM2017-85934-C3-2-P (C.M.-C.), and Junta de Comunidades de Castilla-La Mancha through project SBPLY/19/180501/000110 (J.C.B. and J.C.).

\bibliography{bibliography}{}
\bibliographystyle{siam}

\end{document}